\documentclass[reqno, 12pt]{amsart}
\usepackage[a4paper, margin=1in]{geometry}
\title[Semi-integral Brauer--Manin and quadric orbifolds]{Semi-integral Brauer--Manin obstruction and quadric orbifold pairs}
\date{\today}
\author{Vladimir Mitankin}
\address{Vladimir Mitankin, Institute of Mathematics and Informatics, Bulgarian Academy of Sciences, Akad. Georgi Bonchev 8, 1113 Sofia, Bulgaria.}
\address{Institut für Algebra, Zahlentheorie und Diskrete Mathematik,
Leibniz Universität Hannover, Welfengarten 1, 30167 Hannover, Germany.}
\email{v.mitankin@math.bas.bg}
\author{Masahiro Nakahara}
\address{Masahiro Nakahara, Department of Mathematics, University of Washington, Seattle, WA 98195, USA.}
\email{mn75@uw.edu}
\author{Sam Streeter}
\address{Sam Streeter, School of Mathematics, University of Bristol, Woodland Road, Bristol, BS1 8UG, UK.}
\email{sam.streeter@bristol.ac.uk}
\makeatletter
\@namedef{subjclassname@2020}{\textup{2020} Mathematics Subject Classification}
\makeatother
\usepackage{amsmath}
\usepackage{amssymb}
\usepackage{amsthm}
\usepackage{amsrefs}
\usepackage{enumitem}
\usepackage{hyperref}

\usepackage{tabularx}
\usepackage{pifont}

\usepackage{relsize}

\usepackage{tikz-cd}
\usepackage{color,soul}
\usepackage{mathrsfs}
\usepackage{mathtools}
\usepackage{colonequals}
\theoremstyle{definition}
\newtheorem{mydef}{Definition}
\newtheorem{note}[mydef]{Note}
\newtheorem{remark}[mydef]{Remark}
\newtheorem{question}[mydef]{Question}
\newtheorem{example}[mydef]{Example}
\theoremstyle{plain}
\newtheorem{theorem}[mydef]{Theorem}
\newtheorem*{theorem*}{Theorem}
\newtheorem{proposition}[mydef]{Proposition}
\newtheorem{lemma}[mydef]{Lemma}

\newcommand{\thistheoremname}{}
\newtheorem{genericthm}[mydef]{\thistheoremname}

\newtheorem*{genericthm*}{\thistheoremname}
\newenvironment{namedthm*}[1]
  {\renewcommand{\thistheoremname}{#1}%
   \begin{genericthm*}}
  {\end{genericthm*}}

\numberwithin{equation}{section}
\numberwithin{mydef}{section}
\let\originalleft\left
\let\originalright\right
\renewcommand{\left}{\mathopen{}\mathclose\bgroup\originalleft}
\renewcommand{\right}{\aftergroup\egroup\originalright}

\DeclareMathOperator{\Bl}{Bl}
\DeclareMathOperator{\Br}{Br}

\DeclareMathOperator{\cl}{cl}

\DeclareMathOperator{\DDelta}{\widetilde{\Delta}}

\DeclareMathOperator{\ev}{ev}

\DeclareMathOperator{\Proj}{Proj}

\DeclareMathOperator{\red}{red}

\DeclareMathOperator{\Spec}{Spec}

\DeclareMathOperator{\ord}{ord}

\def\Z{\ifmmode{{\mathbb Z}}\else{${\mathbb Z}$}\fi}
\def\Q{\ifmmode{{\mathbb Q}}\else{${\mathbb Q}$}\fi}
\def\P{\ifmmode{{\mathbb P}}\else{${\mathbb P}$}\fi}
\def\H{\ifmmode{{\mathrm H}}\else{${\mathrm H}$}\fi}
\def\R{\ifmmode{{\mathbb R}}\else{${\mathbb R}$}\fi}
\def\F{\ifmmode{{\mathbb F}}\else{${\mathbb F}$}\fi}
\def\O{\ifmmode{{\calO}}\else{${\calO}$}\fi}
\newcommand{\A}{\mathbb{A}}

\newcommand{\calA}{{\mathcal A}}
\newcommand{\calB}{{\mathcal B}}
\newcommand{\calC}{{\mathcal C}}
\newcommand{\calD}{{\mathcal D}}
\newcommand{\calE}{{\mathcal E}}
\newcommand{\calF}{{\mathcal F}}

\newcommand{\calH}{{\mathcal H}}

\newcommand{\calL}{{\mathcal L}}

\newcommand{\calO}{{\mathcal O}}
\newcommand{\calP}{{\mathcal P}}
\newcommand{\calQ}{{\mathcal Q}}

\newcommand{\calU}{{\mathcal U}}
\newcommand{\calV}{{\mathcal V}}

\newcommand{\calX}{{\mathcal X}}
\newcommand{\calY}{{\mathcal Y}}
\newcommand{\calZ}{{\mathcal Z}}
\renewcommand{\(}{\left(}
\renewcommand{\)}{\right)}

\newcommand{\ZZ}{\mathbb{Z}}
\newcommand{\QQ}{\mathbb{Q}}
\newcommand{\Zp}{\mathbb{Z}_p}

\newcommand{\RR}{\mathbb{R}}
\newcommand{\CC}{\mathbb{C}}
\newcommand{\FF}{\mathbb{F}}
\newcommand{\Fp}{\mathbb{F}_p}

\newcommand{\Fv}{\mathbb{F}_v}
\renewcommand{\AA}{\mathbb{A}}
\newcommand{\PP}{\mathbb{P}}
\newcommand{\bfa}{\mathbf{a}}

\newcommand{\bfx}{\mathbf{x}}
\newcommand{\sA}{\mathcal{A}}

\newcommand{\sD}{\mathcal{D}}

\newcommand{\sO}{\mathcal{O}}
\newcommand{\sQ}{\mathcal{Q}}

\newcommand{\sU}{\mathcal{U}}
\newcommand{\sV}{\mathcal{V}}
\newcommand{\sX}{\mathcal{X}}
\newcommand{\scrA}{\mathscr{A}}
\newcommand{\scrB}{\mathscr{B}}
\newcommand{\scrX}{\mathscr{X}}
\newcommand{\scrY}{\mathscr{Y}}
\DeclareMathOperator{\inv}{inv}

\renewcommand{\epsilon}{\varepsilon}

\newcommand{\nequiv}{\not\equiv}

\DeclareMathOperator{\valp}{v_{\it p}}

\renewcommand{\val}{v}

\DeclareMathOperator{\cc}{\textup{C}}
\DeclareMathOperator{\dd}{\textup{D}}
\DeclareMathOperator{\str}{\textup{st}}
\DeclareMathOperator{\ns}{\textup{ns}}
\DeclareMathOperator{\tr}{\textup{ns}}
\DeclareMathOperator{\fin}{\textup{fin}}

    \DeclareFontFamily{U}{wncy}{}
    \DeclareFontShape{U}{wncy}{m}{n}{<->wncyr10}{}
    \DeclareSymbolFont{mcy}{U}{wncy}{m}{n}
    \DeclareMathSymbol{\Sh}{\mathord}{mcy}{"58} 
\begin{document}
\begin{abstract}
We study local-global principles for two notions of semi-integral points, termed Campana points and Darmon points. In particular, we develop a semi-integral version of the Brauer--Manin obstruction interpolating between Manin's classical version for rational points and the integral version developed by Colliot-Th\'el\`ene and Xu. We determine the status of local-global principles, and obstructions to them, in two families of orbifolds naturally associated to quadric hypersurfaces. Further, we establish a quantitative result measuring the failure of the semi-integral Brauer--Manin obstruction to account for its integral counterpart for affine quadrics.
\end{abstract}
\maketitle
\setcounter{tocdepth}{1}
\tableofcontents

\section{Introduction}

The theory of Campana points has recently received considerable attention thanks to its appeal from both geometric and arithmetic perspectives. Geometrically, the theory offers a bridge between rational and integral points, and with it the hope of shedding new light on problems in the integral realm. Arithmetically, it provides a framework for studying \emph{powerful} solutions to equations. Recall that, given an integer $m \geq 2$, an integer $n$ is \emph{$m$-full} if $p^m$ divides $n$ for all primes $p$ dividing $n$. Several recent works have developed and explored rational points concepts in the Campana setting, including a Manin-type conjecture \cites{PSTVA, STR, SHU1, SHU2} as well as the Hilbert property and weak approximation \cite{NS}.

In this paper we formulate and study the Hasse principle (Definition~\ref{def:CHP}) and strong and weak approximation (Definition~\ref{def:CWA}) for Campana points and a related type of rational points which we name \emph{Darmon points} (Definition \ref{def:dar}), coming from Darmon's study of \emph{$M$-curves} \cite{DAR}. While Campana points meet the orbifold divisor with zero or high intersection multiplicity relative to its weights (see Definition \ref{def:cam}), for Darmon points, the multiplicity is a multiple of the relevant weight. Darmon points then correspond to perfect power solutions. We shall refer to Campana points and Darmon points collectively as \emph{semi-integral points}.

Our primary focus is the development of a Brauer--Manin obstruction for semi-integral points (Section \ref{section:BMO}), which allows us to study new notions of local-global principles for semi-integral points. A key ingredient is the introduction of a new natural choice of adelic space admitting a well-defined Brauer--Manin pairing and notions of strong and weak approximation in the semi-integral setting. An important feature of the semi-integral Brauer--Manin obstruction is that it interpolates between Manin's classical version for rational points and the integral Brauer--Manin obstruction of Colliot-Th\'el\`ene and Xu \cite{CTX}*{\S1} (Remark~\ref{rem:BMOinterpolation}), allowing us to highlight similarities and differences between the semi-integral obstruction and its rational and integral counterparts.

A powerful tool in studying the Brauer--Manin obstruction for rational points is Harari's formal lemma \cite{HARA}. Roughly, the lemma tells us that, given a Brauer class $\sA$ on an open subvariety $U$ of a variety $X$ not coming from a Brauer class on $X$, the Brauer--Manin pairing with $\sA$ is non-constant at infinitely many places. It is often used to study properties related to strong approximation (e.g. \cite[Prop.~2.6]{CTX13}). As a first application of the semi-integral Brauer--Manin obstruction, we establish an analogue of Harari's formal lemma for semi-integral points in Theorems~\ref{thm:HarariA} and \ref{thm:HarariB}. 

Let $X$ be a smooth proper variety over a number field $K$ and $\{D_\alpha\}_{\alpha \in\scrA}$ be the set of all irreducible effective Cartier divisors $D_\alpha\subset X$. Let $D = \sum_{\alpha \in \scrA} \epsilon_\alpha D_{\alpha}$ be a $\QQ$-divisor on $X$, where $\epsilon_\alpha = 1 - 1/m_\alpha$ for weights $m_\alpha \in \ZZ_{\ge 1}$ with $m_\alpha \neq 1$ for only finitely many $\alpha$. Set $D_{\red} := \bigcup_{\epsilon_\alpha \neq 0} D_\alpha$ and $U := X \setminus D_{\red}$. Further, set $D_{\inf} := \bigcup_{\epsilon_\alpha = 1} D_\alpha$, $D_{\fin} := D_{\red} \setminus D_{\inf}$, and for $\varepsilon_\alpha \neq 0, 1$, set $D_{\alpha, \fin} := D_\alpha \setminus D_{\inf}$ (Definition~\ref{def:D-inf-and-fin}). Finally, set $V := X \setminus D_{\inf}$.

Given a finite set $S$ of places of $K$ containing all archimedean places and an $\O_S$-model $(\calX,\calD)$ of the Campana orbifold $(X,D)$ (Definitions~\ref{def:orbifold} and \ref{def:model}), we will denote by $(\calX,\calD)^{\cc}(\O_S)$ and $(\calX,\calD)^{\dd}(\O_S)$ the sets of Campana and Darmon points respectively. See Definition \ref{def:adelicCampana} for the notion of the semi-integral $T$-adelic space. As we shall see, semi-integral points on $U$ behave very differently from those on $D_{\fin}$. Let $(\sX,\sD)_{\str}^{\cc}(\AA_{K,S})$ be the set of all adelic Campana points on $U$, and let $(\sX,\sD)_{\tr}^{\cc}(\AA_{K,S})$ be the adelic Campana points lying on $D_{\fin}$. Define $(\sX,\sD)_{\str}^{\dd}(\AA_{K,S})$ and $(\sX,\sD)_{\tr}^{\dd}(\AA_{K,S})$ analogously (Definition~\ref{def:adelicCampana}). Strict points pair naturally with $\Br U$, while non-strict points feature in pairings with each $\Br D_{\alpha,\fin}$, $\epsilon_\alpha \neq 0,1$ (Definition~\ref{def:SIBMO}). Since $X$ is smooth, we have $\Br V \subset \Br U$ \cite[Thm.~3.5.7]{CTS2}, and all adelic points pair naturally with $\Br V$. Lastly, denote by $\Br (X, D)$ the set of elements of $\Br U$ whose residue along $D_{\alpha,\fin}$ (as a divisor of $V$) is $m_\alpha$-torsion for each $m_\alpha \neq 1$ (Definition~\ref{def:BrXD}). We then have
\[
  \begin{split}
  (\sX, \sD)_{\str}^{\cc}(\AA_{K,S})^{\Br} &\subseteq (\sX,\sD)_{\str}^{\cc}(\AA_{K,S})^{\Br V} \subseteq (\sX, \sD)_{\str}^{\cc}(\AA_{K,S}), \\ 
  (\sX, \sD)_{\str}^{\dd}(\AA_{K,S})^{\Br} &\subseteq (\sX,\sD)_{\str}^{\dd}(\AA_{K,S})^{\Br(X, D)} \subseteq (\sX, \sD)_{\str}^{\dd}(\AA_{K,S}). 
  \end{split}
\]

Recall that an $n$-torsion Brauer element $\sA$ is \emph{prolific at $v$} if its local invariant map at $v$ takes all possible values $k/n$, $0 \leq k \leq n-1$. We call $\sA$ \emph{prolific at $T$} if $\sum_{v \in T} \inv_v \sA$ takes all possible values $k/n$, $0 \leq k \leq n-1$ for some choice of adeles in the corresponding adelic set. We call $\sA$ prolific at $T$ with respect to $U$ if it is prolific at the set of strict semi-integral points. Our first result shows that extra elements in $\Br U$ do not obstruct the Hasse principle, but they almost always obstruct strong approximation.

\begin{theorem} \label{thm:main0}
Let $(X,D)$ be a Campana orbifold with $\O_S$-model $(\calX,\calD)$.
  \begin{enumerate}[label=\emph{(\roman*)}] 
    \item \label{m0i} There is no Brauer--Manin obstruction to the Campana Hasse principle if \\ $(\calX,\calD)_{\str}^{\cc}(\A_{K, S})^{\Br V}\neq\emptyset$ and $\Br U / \Br K$ is finite.
    
    \item \label{m0ii} Assume that $D_{\alpha,\fin}$ is regular for all $\epsilon_{\alpha} \neq 0,1$. There is no Brauer--Manin obstruction to the Darmon Hasse principle if $(\calX,\calD)_{\str}^{\dd}(\A_{K, S})^{\Br (X, D)}\neq\emptyset$ and $\Br U / \Br K$ is finite.

    \item \label{m0iii} Campana strong approximation off $T$ fails if $(\calX, \calD)_{\str}^{\cc}(\AA_{K, S}^T) \neq \emptyset$ and there exists $\sA \in \Br U \setminus \Br V$ such that $\sA$ is not prolific at $T$ with respect to $U$.
    
    \item \label{m0iv} Darmon strong approximation off $T$ fails if $(\calX, \calD)_{\str}^{\dd}(\AA_{K, S}^T) \neq \emptyset$ and there  exists $\sA \in \Br U \setminus \Br (X,D)$ such that $\sA$ is not prolific at $T$ with respect to $U$.
  \end{enumerate}
\end{theorem}

Note that, if $D_{\inf} = \emptyset$ and $\Br X=\Br K$ (e.g.\ if $X$ is rational) and $\Br U/\Br K$ is finite, then (i) tells us that there is no Brauer--Manin obstruction to the Campana Hasse principle for strict points on $(\calX,\calD)$. The finiteness conditions in (i) and (ii) are necessary in order to apply Harari's formal lemma. The requirements on $\sA$ not being prolific are also necessary as demonstrated in Proposition~\ref{prop:values-of-inv-at-T}. Provided that $T$ consists of large places, any element of the smaller Brauer groups $\Br V$ in (iii) or $\Br (X,D)$ (if $D_{\alpha,\fin}$ is regular for all $\epsilon_{\alpha} \neq 0,1$) in (iv) has constant invariant map at places of $T$ by Lemma~\ref{lema:trivpairing}, so whether such elements play a role in obstructing Campana or Darmon strong approximation off $T$ amounts to understanding the behaviour of their local invariant maps at finitely many places. This is not true for a general Brauer element. Theorem~\ref{thm:HarariA} shows that we may allow the finite set $T$ in Theorem~\ref{thm:main0} to contain arbitrary large places if they are chosen carefully. As an application of Theorem \ref{thm:main0}, we show that the Brauer--Manin pairing can still detect the existence or lack of an obstruction to the Hasse principle and strong approximation for semi-integral points even when the local invariant map of a Brauer group element is non-constant at infinitely many places (see Theorem \hyperref[thm:main3]{\ref*{thm:main3}\emph{\ref*{m3i}}}).

\begin{remark}
In the framework of the Brauer--Manin obstruction for integral points on arbitrary varieties or for rational points on proper varieties, the local invariant map for each element of the Brauer group vanishes outside a finite set of places. However, for Campana orbifolds, if one of the weights $m_\alpha$ and the order of one of the Brauer elements are coprime, a behaviour akin to Harari's formal lemma \cite[Thm.~13.4.1]{CTS2} is present on the level of Darmon points. This is shown in Theorem~\ref{thm:HarariA} for general orbifolds and is exploited crucially in the proofs of Theorems~\ref{thm:main0} and \ref{thm:main3}.
\end{remark}

Of particular interest is the appearance of the ``stacky'' Brauer group $\Br\left(X,D\right)$ (see Remark \ref{rmk:rootstack}), which is indicative of the close relationship between Darmon points and stacky points (as evidenced by Example \ref{ex:dar}), the latter being a topic of significant interest in several recent papers \cites{BP, LW23, NX1, NX2, SAN23}.

\subsection*{Two families of orbifolds}

In the second part of the paper, we will use the semi-integral Brauer--Manin obstruction alongside techniques from birational geometry and analytic number theory to study the Hasse principle and approximation properties for two families of orbifolds associated to quadric hypersurfaces. As was demonstrated in \cite{CTX}, these provided prototypical examples to illustrate the features of the integral Brauer--Manin obstruction. Our goal here is to do the same for semi-integral points.

Let $n \in \ZZ_{\ge 0}$ and let $Q \subset \PP^{n+1}$ be a smooth quadric hypersurface. Let $H\subset \P^{n+1}$ be a hyperplane, and set $\Delta := Q \cap H$. Let $\calQ$ be the closure of $Q$ in $\P^{n+1}_{\O_S}$, and let $\DDelta$ be the closure of $\Delta$ in $\sQ$. Letting $m \in \mathbb{Z}_{\geq 2}$, define the $\Q$-divisors $Q_m:=\(1-1/m\)Q$ on $\mathbb{P}^{n+1}$ and $\Delta_m:=\(1-1/m\)(Q\cap H)$ on $Q$. Define $\sQ_m$ and $\DDelta_m$ analogously. Our focus will be on the two orbifolds $(Q, \Delta_m)$ and $(\P^{n+1},Q_m)$.

Taking $m=1$ in each family, all local-global principles hold. Indeed, all local-global principles then become those for rational points on the ambient variety, and $\P^{n+1}$ and $Q$ satisfy strong approximation and the Hasse principle, hence our stipulation that $m \geq 2$.

The first part of this study concerns weak approximation for semi-integral points on the aforementioned orbifolds. More precisely, we are interested in \emph{Campana/Darmon weak approximation} (Definition \ref{def:CWA}). We begin with the following result on weak approximation when $Q$ plays the role of orbifold divisor.

\begin{theorem} \label{thm:main1}
The orbifold $(\mathbb{P}^{n+1}_{\O_S},\calQ_m)$
\begin{enumerate}[label=\emph{(\roman*)}]
\item \label{m1i} satisfies Campana weak approximation and $(\P^{n+1}_{\O_S},\calQ_m)^{\cc}(\O_S)\neq\emptyset$;
\item \label{m1ii} fails Darmon weak weak approximation if $m$ is even.
\end{enumerate}
\end{theorem}

For even $m$, we show that $(\mathbb{P}^{n+1}_{\O_S},\calQ_m)^{\dd}(\O_S)$ is in fact thin (Proposition \ref{prop:Darmonthin}), implying the failure of Darmon weak weak approximation. We also show in Example \ref{ex:DWWAfail} that Darmon weak approximation can fail for odd $m$. The case $n=0$ of Theorem \hyperref[thm:main1]{\ref*{thm:main1}\emph{\ref*{m1i}}} was established in \cite{NS}*{Prop.~3.15}; prior to the present work, this was the only example of Campana weak approximation for projective space with a non-linear orbifold divisor. 

We then derive the following result for $Q$ playing the role of ambient variety. As is well-known, a quadric hypersurface with a rational point is birational to projective space; we exploit this fact to pass from Campana points for projective space with a quadric divisor to Campana points for a quadric with a linear orbifold divisor (a hyperplane section).

\begin{theorem} \label{thm:main2}
The orbifold $(\calQ,\DDelta_m)$ satisfies Campana weak approximation for all $n\geq 1$.
\end{theorem}

The proof of Theorem~\ref{thm:main2} relies upon the behaviour of Campana points under birational transformations. We develop a new method in Proposition \ref{prop:ebom} which establishes sufficient conditions for the preservation of both semi-integrality and integrality of points under certain birational maps. This key result allows us to deduce Theorem \ref{thm:main2} by transferring to the setting of Theorem \ref{thm:main1}, swapping the degrees of ambient variety and orbifold divisor. An interesting direction for future research is to consider how one may exploit and expand upon this result in order to bring to bear the tools of birational geometry on problems concerning semi-integral and integral points. Note that, in this case, such a transfer principle fails for Darmon points between the two families. Indeed, while Theorem~\hyperref[thm:main1]{\ref*{thm:main1}\emph{\ref*{m1ii}}} tells us that Darmon weak approximation fails in the even-weight case for $(\P^{n+1}_{\O_S},\calQ_m)$, this property may hold in the even-weight case for $(\calQ,\DDelta_m)$ (Example~\ref{ex:DWA}).

Our next result concerns the status of the remaining local-global principles for $(\calQ,\DDelta_m)$. For any finite set $T$ of places of $K$, we denote by $(\sQ, \DDelta_m)^*(\AA_{K, S \cup T}^{T})^{\Br}$ the projection of $(\sQ, \DDelta_m)^*(\A_{K, S \cup T})^{\Br}$ to $(\sQ, \DDelta_m)^*(\A_{K, S}^T)$, as explained in Definition~\ref{def:Br-T-adeles}.

\begin{theorem} \label{thm:main3}
Assume that $n \ge 2$ and $\Delta = Q \cap H$ is a smooth divisor. Let $d\in K$ be the discriminant of a quadratic form defining $Q$. Let $v_0$ be a place where $\Delta(K_{v_0})\neq\emptyset$.

\begin{enumerate}[label=\emph{(\roman*)}]
\item \label{m3i} The set $(\sQ, \DDelta_m)^*(\sO_{S\cup\{v_0\}})$ is dense in $(\sQ, \DDelta_m)^*(\AA_{K, S \cup \{v_0\}}^{\{v_0\}})^{\Br}$ for $* \in \{C,D\}$.
    
\item \label{m3ii} Assume that either $n\geq3$ or $d$ is a square. Then $\Br (Q \setminus \Delta) / \Br K=0$ and $(\sQ, \DDelta_m)$ satisfies both Campana and Darmon strong approximation off $\{v_0\}$. 
    
\item \label{m3iii} Assume that $n=2$ and $d$ is not a square. Then $\Br (Q \setminus \Delta) / \Br K\simeq\Z/2\Z$, generated by some class $\sA$; since $1 < m < \infty$, there exists a density-$1/2$ subset $\Omega_{K,\sA}\subset \Omega_K$ such that the following hold for any finite subset $T\subset \Omega_{K,\sA}$:

	\begin{enumerate}[label=\emph{(\alph*)}]
	\item \label{m3iiia} If $(\sQ, \DDelta_m)_{\str}^{\cc}(\AA_{K, S}^T) \neq \emptyset$, then Campana strong approximation off $T$ fails;
      
	\item \label{m3iiib} If $m$ is odd and $(\sQ, \DDelta_m)_{\str}^{\dd}(\AA_{K, S}^T) \neq \emptyset$, then Darmon strong approximation off $T$ fails but if $v_0 \in S$ the Darmon Hasse principle holds.
	\end{enumerate}

\end{enumerate}
\end{theorem}

If $\Br (Q \setminus \Delta) / \Br K$ is non-trivial with generator $\sA$, then $\Omega_{K,\sA}$ above is the set of places of $K$ where the local invariant map for $\sA$ is constant. It is independent of the choice of $\sA$, since any two Brauer elements generating $\Br (Q \setminus \Delta) / \Br K$  differ by a constant algebra. We will see in the proof of Theorem~\ref{thm:main3} that $\Omega_{K,\sA}$ indeed has density $1/2$ in $\Omega_K$, since it differs by finitely many places from the set of places for which $d$ is a quadratic residue. 

For integral points, an additional so-called \emph{obstruction at infinity} may occur further to the Brauer--Manin obstruction. To avoid the presence of an obstruction at infinity, one normally assumes that the real points are not compact (cf. \cite[Thm.~6.1]{CTX}). The exclusion of a place $v_0$ such that $\Delta(K_{v_0}) \neq \emptyset$ in Theorem~\ref{thm:main3} is an analogue of this assumption, and it guarantees that the Brauer--Manin obstruction developed in Section~\ref{section:BMO} explains all failures of the Campana/Darmon Hasse principle or Campana/Darmon strong approximation away from $v_0$. This follows from Proposition~\ref{prop:BMOequiv} as we explain in the proof of Theorem~\ref{thm:main3}. Proposition~\ref{prop:BMOequiv} implies that for general orbifolds, the closure of the set of strict Campana/Darmon points being dense in their adelic space is equivalent to the analogous statement for integral or rational points.

\begin{remark}
If $n \ge 1$ and $\Delta = Q \cap H$ is singular, then $\Br(Q \setminus \Delta) \cong \Br K$. This is proved in Proposition \ref{prop:main3non-smooth-b} together with the Hasse principle and strong approximation for strict Campana and Darmon points off any finite set of places $T \subset S$ for $(\sQ,\DDelta_m)$.
\end{remark}

Theorems~\ref{thm:main1} and \ref{thm:main2} tell us that both families of orbifolds satisfy the Campana Hasse principle and Campana weak approximation for any choice of weight. However, the Hasse principle for Darmon points of even weight for the orbifolds of Theorem \ref{thm:main2} may fail due to a Brauer--Manin obstruction (see Example~\ref{ex:quadrics-even}). Such failures imply that there are counterexamples to Darmon weak and strong approximation in this family, since each property together with non-empty adelic space implies the Darmon Hasse principle. 

For the existence of $S$-integral points on $\sQ \setminus \DDelta$, it is necessary that $(\sQ \setminus \DDelta)(\AA_{\sO_S}) \neq \emptyset$, a condition which implies that $(\sQ, \DDelta_m)_{\str}^{\dd}(\AA_K) \neq \emptyset$. Moreover, as $S$-integral points are Darmon points, the lack of strict Darmon points on $(\sQ, \DDelta_m)$ for some $m \in \ZZ_{\ge 1}$ forces $\sQ \setminus \DDelta$ to have no $S$-integral points. One may thus pose the following natural question.

\begin{question}
  \label{qes:DHPtoIHP}
  Does the Hasse principle for Darmon points on $(\sQ, \DDelta_m)$ not lying on $\Delta$ for all finite $m \ge 1$ imply the Hasse principle for $S$-integral points on $\sQ \setminus \DDelta$?
\end{question}

It is conceivable that, at places of bad reduction, local Darmon points close to $\Delta$ in the even-weight case may come from adelic points in $(\sQ, \DDelta_m)^{\dd}(\AA_K)^{\Br}$, while there may be a Brauer--Manin obstruction to the existence of integral points on $\sQ \setminus \DDelta$. This suggests that the answer to the above question is negative. We give an example which shows that the answer is negative for cubic surfaces with a plane section (Example \ref{ex:DMBO-does-not-imply-IBMO}).

Note that even if $(\sQ, \DDelta_m)^{\dd}(\O_S) \cap (Q \setminus \Delta)(K) \neq \emptyset$ for all $m \in \Z_{\ge 1}$, it may be that $(\sQ \setminus \DDelta)(\AA_{\O_S}) = \emptyset$. Indeed, the quadric $3(x - y)(x+ y) = (t - 4z)(t + 4z)$ with orbifold divisor $\Delta_m = \(1-1/m\)Z\left(t\right)$ has no $2$-adic integral points, but $[x:y:z:t] = [1:1:4^{m-1}:4^m]$ is a Darmon point away from $t = 0$ for any $m \in \ZZ_{\ge 1}$.

We now specialise to diagonal quadrics to explore Question~\ref{qes:DHPtoIHP} in greater depth. In our final result, we quantify the number of integral Hasse principle failures such that the Darmon Hasse principle holds for any weight. One may think of this as measuring the failure of the Darmon Brauer--Manin obstruction to account for its integral counterpart. 

For any quadruple $(a_1, a_2, a_3, n) \in \ZZ^4$ with $a_1a_2a_3n \neq 0$ and $a_1, a_2$ and $a_3$ not all of the same sign, which is assumed to avoid the presence of an obstruction at infinity, let $Q_{\bfa, n}$ be the quadric $a_1x^2 + a_2y^2 + a_3z^2 - nt^2 = 0$ , and let $\Delta_m$ be the $\Q$-divisor $\(1-1/m\)Z(t)$. Let $(\calQ_{\bfa, n},\DDelta_m)$ be the $\Z$-model of $(Q_{\bfa, n},\Delta_m)$ with $\calQ_{\bfa,n}$ the surface in $\P^3_\Z$ defined by the equation above divided through by $\gcd\(a_1, a_2, a_3,n\)$. Then for any $B \in \R_{\ge 1}$, set
\[
  N_n(B) := \#\left\{ 0 < |a_1|, |a_2|, |a_3| \le B: 
  \begin{array}{l}
    a_1, a_2, a_3 \text{ not all the same sign,}\\ 
    (\sQ_{\bfa, n} \setminus \DDelta) \text{ fails integral Hasse principle but} \\
    (\sQ_{\bfa, n}, \DDelta_m)_{\str}^{\dd}\(\ZZ\) \neq \emptyset \text{ for all } m \in \Z_{\ge 1}.
  \end{array} 
  \right\}.
\] 
Our last result features upper and lower bounds for this quantity. 

\begin{theorem}
  \label{thm:main4}
  Fix a non-zero integer $n$. Then, as $B$ goes to infinity,
  \[
    B^{3/2}\log B \ll_n N_n(B) \ll_n B^{3/2}\(\log B\)^{3/2}.
  \]
\end{theorem}

For fixed non-zero $n \in \ZZ$, let $N'_n\(B\)$ denote the number of failures of the integral Hasse principle in the family $\sQ_{\bfa, n} \setminus \DDelta$ considered above with $0 < |a_1|, |a_2|, |a_3| \le B$. This quantity was first studied in \cite[Thm.~1.2, Cor.~1.4]{MIT} where similar bounds with different powers of $\log B$ are established. These bounds were later improved in \cite[Thm~1.1]{SAN}. As $N_n(B) \le N_n'(B)$, the upper bound here follows from the upper bound in \cite[Thm~1.1]{SAN}. An interesting feature of the lower bound of Theorem~\ref{thm:main4} is that it is a further improvement of the lower bound in \cite[Thm~1.1]{SAN}. Thus Theorem~\ref{thm:main4} produces almost the expected magnitude of $N_n(B)$ (and also of $N_n'(B)$ when combined with \cite[Thm~1.1]{SAN}) up to a factor of $(\log B)^{1/2}$. It would be of interest to see whether this magnitude matches the upper bound given here. A major obstacle is the fact that it is not possible to find a uniform generator for $\Br (\sQ_{\bfa, n} \setminus \DDelta) / \Br \QQ$ across the whole family \cite[Cor.~4.4]{UEM}. Obtaining necessary and sufficient conditions to detect the presence of a Brauer--Manin obstruction amounts to selecting a large proper subfamily for which a uniform generator exists.

This paper is organised as follows. We begin in Section \ref{section:SIP} by defining Campana and Darmon points at the local, global and adelic levels, culminating in the definition of local-global principles in the semi-integral setting. In Section \ref{section:BMO}, we develop a semi-integral version of the Brauer--Manin obstruction in order to explore obstructions to local-global principles, and we prove a semi-integral version of Harari's formal lemma which is then used to prove Theorem \ref{thm:main0}. We turn from algebra to geometry in Section \ref{section:maps}, where we explore the behaviour of semi-integral points under certain maps. Having established this behaviour, we directly prove Theorem \hyperref[thm:main1]{\ref*{thm:main1}\emph{\ref*{m1i}}} and use a ``lifting and descending'' procedure to prove Theorem \ref{thm:main2} in Section \ref{section:CPQ}. We turn to Darmon points in Section \ref{section:D1}, which is devoted to the proof of Theorem~\hyperref[thm:main1]{\ref*{thm:main1}\emph{\ref*{m1ii}}}, and in Section \ref{section:D2}, where we prove Theorem~\ref{thm:main3}. Finally, we establish the counting result of Theorem \ref{thm:main4} in Section \ref{section:fail}.

\subsection*{Conventions}

\subsubsection*{Geometry}
Given a ring $R$, we write $\Spec R$ for the spectrum of $R$ with the Zariski topology. An $R$-scheme is a scheme $X$ together with a morphism $X \rightarrow \Spec R$. The set of $R$-points $X\left(R\right)$ of $X$ is the set of sections of the structure morphism $X \rightarrow \Spec R$. If $X = \Proj T$ for some ring $T$ and $f \in T$, we denote by $Z\left(f\right) \subset X$ the closed subscheme $\Proj T/\left(f\right)$. Given a morphism $\Spec S \rightarrow \Spec R$, we denote by $X_S$ the fibre product $X \times_{\Spec R} \Spec S$. Given $n \in \Z_{>0}$, we denote by $\A^n_R$ and $\P^n_R$ the affine and projective $n$-spaces over $R$ respectively, omitting the ground ring $R$ if it is clear from context. A variety over a field $F$ is a geometrically integral separated scheme of finite type over $F$.
\subsubsection*{Number theory}
Given a number field $K$, we denote by $\Omega_K$ the set of places of $K$. We denote by $\Omega_K^\infty \subset \Omega_K$ the finite subset of archimedean places of $K$. For $v \in \Omega_K$, we denote by $K_v$ the completion of $K$ at $v$; if $v \not\in \Omega_K^\infty$, then $\O_v$ denotes the ring of $v$-adic integers in $K_v$, and $\pi_v$ and $\F_v$ denote a uniformiser for $\O_v$ and the residue field of $\O_v$ respectively. We set $q_v := \#\F_v$. Given a finite subset $S \subset \Omega_K$ containing $\Omega_K^\infty$, we denote by $\O_S$ the ring of $S$-integers of $K$ and by $\A_{\O_S}$ the ring $\prod_{v \in S}K_v \times \prod_{v \not\in S}\O_v$ given the product topology. When $S = \Omega_{K}^\infty$, write $\O_K$ for $\O_S$ and $\A_{\O_K}$ for $\A_{\O_S}$. Given a finite subset $T \subset \Omega_K$, we denote by $\A_K^T$ the $T$-adele ring given by the restricted product
\[
\sideset{}{'}\prod_{v \in \Omega_K \setminus T}(K_v,\O_v)\colonequals\left\{(x_v)\in\prod_{v \in \Omega_K \setminus T}K_v \ : \ x_v\in\O_v\textup{ for all but finitely many }v\right\},
\]
recovering the adele ring $\A_K$ when $T = \emptyset$. We denote by $v_p$ the $p$-adic valuation for a rational prime $p \in \QQ$.
\subsubsection*{Arithmetic geometry}
Let $X$ be a variety over $K$. Let $v$ be a non-archimedean place of $K$, and let $S \subset \Omega_K$ be a finite set containing $\Omega_K^\infty$. Let $R \in \{\O_v,\O_S\}$ and $F = \text{Frac}R$. An \emph{$R$-model} of $X$ is a flat separated $R$-scheme $\calX$ of finite type together with an isomorphism between $X$ and the generic fibre of $\calX$. Given $x \in X\left(F\right)$, we denote by $\overline{x} \in \sX\left(R\right)$ the closure of $x$ in $\sX$. Suppose given an $\O_S$-scheme $\calY$, a place $v \not\in S$ and a finite set $S' \subset \Omega_K$ containing $S$. Then we denote by $\calY_{S'}$ and $\calY_v$ the base changes $\calY_{\O_{S'}}$ and $\calY_{\F_v}$ respectively.

\subsubsection*{Analysis}
Given functions $f: \RR \rightarrow \CC$ and $g: \RR \rightarrow \RR$ with $g(x) >0$ for large $x$, we write $f(x) = O(g(x))$ or $f(x) \ll g(x)$ to indicate that there exists a constant $C>0$ and $x_0 \in \R$ such that $|f(x)| \leq Cg(x)$ for all $x \geq x_0$. We denote by $\mu: \mathbb{Z}_{>0} \rightarrow \mathbb{C}$ the M\"obius function.

\subsection*{Acknowledgements}

We are grateful to Martin Bright, Tim Browning, Jean-Louis Colliot--Th\'el\`ene, Daniel Loughran, Boaz Moerman, Marta Pieropan, Tim Santens, Alec Shute, Sho Tanimoto and Haowen Zhang for useful comments and suggestions. We thank the International Centre for Mathematical Sciences and the organisers of the conference ``Rational Points on Higher-Dimensional Varieties'' (namely Daniel Loughran, Rachel Newton and Efthymios Sofos) for the opportunity to present and discuss this project. We also thank Marta Pieropan for the invitation to present this work at the Intercity Number Theory Seminar in Utrecht. Lastly, we are indebted to the anonymous referee for their careful analysis and crucial insights which have greatly improved the quality of this article. The first author was partially supported by grant DE 1646/4-2 of the Deutsche Forschungsgemeinschaft during his stay at Leibniz Universit\"at Hannover and by scientific program "Enhancing the Research Capacity in Mathematical Sciences (PIKOM)", No. DO1-67/05.05.2022 of the Ministry of Education and Science of Bulgaria. The third author was supported by the University of Bristol and the Heilbronn Institute for Mathematical Research.

\section{Semi-integral points} \label{section:SIP}

In this section we introduce and gather information on semi-integral points. We define both Campana and Darmon points at the local, global and adelic levels, introducing both weak and strong approximation as well as the Hasse principle in the semi-integral setting.

\subsection{Campana orbifolds}

We begin by introducing the notion of orbifold developed and studied by Campana across several articles \cites{CAM04,CAM05,CAM11,CAM11a,CAM15}.
Let $k$ be any field.

\begin{mydef}
\label{def:orbifold}
A \emph{(Campana) orbifold} over $k$ is a pair $(X,D)$, where $X$ is a smooth proper variety over $k$ and $D$ is a $k$-rational $\mathbb{Q}$-divisor on $X$ of the form
\[
D = \sum_{\alpha \in \scrA} \epsilon_\alpha D_{\alpha},
\]
with $\epsilon_\alpha = 1-1/m_{\alpha}$, where the $D_{\alpha}$ are distinct irreducible effective Cartier divisors on $X$ and $m_{\alpha} \in \mathbb{Z}_{\geq 1} \cup \{\infty\}$ (with the convention $1/\infty = 0$) satisfies $m_\alpha = 1$ for all but finitely many $\alpha \in \scrA$. We call $m_\alpha$ the \emph{weight} of $D_\alpha$. We denote the support $\cup_{\epsilon_\alpha \neq 0}D_\alpha$ of $D$ by $D_{\red}$. We say that $(X,D)$ is \emph{smooth} if $D_{\red}$ is a strict normal crossings divisor.
\end{mydef}

\begin{note}
We draw the reader's attention to the following conventions.
\begin{enumerate}[label=(\roman*)]
\item When dealing with general orbifolds in this section and Theorem \ref{thm:main0}, we do not employ notation which makes the weights of the $\QQ$-divisor explicit, so that $D$ is understood as a $\QQ$-divisor with fixed components and weights. However, when dealing with the two families of orbifolds which are the focus of this paper, we find it prudent to include the weight $m$ in the notation: we use a subscript $m$ to indicate the weight of the one underlying $\QQ$-component, so that $D$ is understood as an effective Cartier divisor with $D_m = \left(1-1/m\right)D$ the associated $\QQ$-divisor.
\item In Campana's original formulation (see \cites{CAM05,CAM11,CAM15}), the ambient variety $X$ is assumed only to be normal (i.e.\ not necessarily smooth), and the constituents of the orbifold divisor are Weil divisors rather than Cartier divisors. However, it suffices for the present work to restrict to the case of smooth $X$ (in which Cartier and Weil divisors are equivalent), hence we impose smoothness in the above definition.
\item As in \cite[\S3]{PSTVA}, we allow $m_\alpha = 1$ (i.e.\ $\epsilon_\alpha = 0$). However, in contrast with \emph{loc.\ cit.}, we do not include the corresponding $D_\alpha$ in the support of $D$.
\end{enumerate}
\end{note}

\subsection{Models and multiplicities}
Let $K$ be a number field, and let $S \subset \Omega_K$ be a finite set containing $\Omega_K^\infty$. Let $(X,D)$ be a Campana orbifold over $K$.

\begin{mydef}
\label{def:model}
An \emph{$\mathcal{O}_S$-model} of $(X,D)$ is a pair $(\calX,\calD)$, where $\calX$ is a proper $\O_S$-model of $X$, and $\mathcal{D} = \sum_{\alpha \in \scrA}
\epsilon_\alpha\calD_\alpha$
for $\calD_\alpha$ the Zariski closure of $D_\alpha$ in $\mathcal{X}$.
\end{mydef}

Let $v \in \Omega_K \setminus S$. Given an $\mathcal{O}_S$-model $(\calX,\calD)$ of $(X,D)$, denote by $\calD_{\beta_v}$ the $\O_v$-components of the $\calD_\alpha$, writing $\beta_v \mid \alpha$ to signify that $\calD_{\beta_v} \subset \calD_\alpha$.

Let $P \in X(K_v)$. By properness, we get an induced point $\mathcal{P}_v \in \mathcal{X}(\O_v)$. Given a closed subscheme $\mathcal{Z} \subset \mathcal{X}_{\O_v}$, we have the scheme-theoretic intersection $\calP_v \cap \calZ$, meaning the fibre product of $\calP_v: \Spec \O_v \rightarrow \calX_{\O_v}$ and the closed immersion $i_\calZ: \calZ \hookrightarrow \calX_{\O_v}$:

\begin{center}
\begin{tikzcd}
\calP_v \cap \calZ \arrow[r] \arrow[d] & \calZ \arrow[d,"i_\calZ"] \\
\Spec \O_v \arrow[r,"\calP_v"] & \calX_{\O_v}.
\end{tikzcd}
\end{center}

Since closed immersions are stable under base change, we have $\calP_v \cap \calZ \cong \Spec \O_v/I$ for an ideal $I \subset \O_v$. Further, we have $I = (0)$ or $I = (\pi_v^n)$ for some $n \in \mathbb{Z}_{\geq 0}$. We may therefore make the following definition.

\begin{mydef}
The \emph{intersection multiplicity} $n_v(\mathcal{Z},P)$ is $\infty$ when $I = (0)$ (equivalently, when $P \in Z = \mathcal{Z} \times_{\Spec \O_v} \Spec K_v$) and the integer $n$ such that $I = (\pi_v^n)$ otherwise. 
\end{mydef}

When $\calZ$ is the zero locus of one function on $\calX_{\O_v}$, the intersection multiplicity $n_v(\calZ,P)$ recovers the valuation of this function at $P$, as illustrated by the following example.

\begin{example}
Take $X = \mathbb{P}^1_{\QQ}$ with coordinates $x_0,x_1$ and $Z = Z(x_0)$, and denote by $\calZ$ the closure of $Z$ in $\mathbb{P}^1_{\ZZ_p}$ for a given prime $p$. Let $P = [a:b] \in \P^1(\QQ)$ with coprime integral coordinates $a,b$. Then $n_v(\calZ,P) = v_p(a)$.
\end{example}

\begin{remark}
In \cite{STR}, $n_v(\calD_\alpha,P)$ was defined to be the sum $\sum_{\beta_v \mid \alpha} n_v(\calD_{\beta_v},P)$, with $n_v(\calD_{\beta_v},P)$ defined as above. This coincides with our definition when $\mathcal{D}_\alpha$ is Cartier, e.g.\ when $\calX$ is regular (see Proposition \hyperref[prop:compat]{\ref*{prop:compat}\emph{\ref*{prop:compati}}}). This follows from the equality $\mathcal{I}_{\mathcal{P}_v} + \prod_{\beta_v \mid \alpha} \left(\mathcal{I}_{\mathcal{P}_v} + \mathcal{I}_{\mathcal{D}_{\beta_v}}\right) = \mathcal{I}_{\mathcal{P}_v} + \prod_{\beta_v \mid \alpha} \mathcal{I}_{\mathcal{D}_{\beta_v}}$ of associated ideal sheaves on $\calX_{\O_v}$. 
\end{remark}

\begin{mydef}
  \label{def:D-inf-and-fin}
Set $D_{\inf} := \bigcup_{\epsilon_\alpha = 1}D_\alpha$ and $\sD_{\inf} := \bigcup_{\epsilon_\alpha = 1}\sD_\alpha$. 
Analogously, for each $\alpha \in \calA$, set $D_{\alpha,\inf}:= D_{\alpha} \cap D_{\inf}$ and $\calD_{\alpha,\inf}:= \calD_{\alpha} \cap \calD_{\inf}$. Then set $D_{\alpha,\fin}:= D_{\alpha} \setminus D_{\alpha,\inf}$ and $\calD_{\alpha,\fin}:= \calD_{\alpha} \setminus \calD_{\alpha,\inf}$. Finally, set $D_{\fin}:= D_{\red} \setminus D_{\inf}$ and $\calD_{\fin}:= \calD_{\red} \setminus \bigcup_{\epsilon_\alpha = 1}\calD_\alpha$. 
\end{mydef}

\subsection{Campana points}
We are now ready to define Campana points and Darmon points.

\begin{mydef} \label{def:cam}
Let $v \not\in S$, and assume that $n_v(\calD_\alpha,P) = 0$ for all $\epsilon_\alpha = 1$. 

\begin{enumerate}[label=(\roman*)]
\item We say that $P$ is a \emph{$v$-adic weak Campana point} if
\[
\sum_{\alpha \in \scrA, \ \epsilon_\alpha\neq0}\frac{1}{m_\alpha}n_v(\calD_\alpha,P) \in \Q_{\geq 1} \cup \{0,\infty\}.
\]
We denote the set of local weak Campana points by $(\calX,\calD
)^{\cc}_{\mathbf{w}}(\O_v)$.

\item We say that $P$ is a \emph{$v$-adic Campana point} if for all $\epsilon_\alpha < 1$, we have
\[
n_v\(\calD_{\alpha},P\) \in \mathbb{Z}_{\geq m_\alpha} \cup \{0,\infty\}.
\]
We denote the set of local Campana points by $(\calX,\calD)^{\cc}(\O_v)$.
\item We say that $P$ is a \emph{$v$-adic strong Campana point} if, for all $\epsilon_\alpha < 1$, we have
\[
n_v\left(\calD_{\beta_v},P\right) \in \mathbb{Z}_{\geq m_\alpha} \cup \{0,\infty\} \textrm{ for all $\beta_v \mid \alpha$.}
\]
We denote the set of local strong Campana points by $(\calX,\calD)^{\cc}_{\mathbf{s}}(\O_v)$.
\end{enumerate}

For $v \in S$, we say that $P \in X(K_v)$ is a \emph{$v$-adic Campana/strong Campana/weak Campana point} if $P \in (X \setminus D_{\inf}) (K_v)$ and use the same notation for the resulting set.

When the place $v$ is clear from context, we use the term ``local'' in place of $v$-adic.

Given $S' \supset S$, we say that $P \in (X \setminus D_{\inf}) (K)$ is an \emph{($S'$-integral) Campana point} if $P \in (\calX,\calD)^{\cc}(\O_v)$ for all $v \in \Omega_K \setminus S'$. We denote the set of $S'$-integral Campana points by $(\calX,\calD)^{\cc}(\O_{S'})$. (Note that this equals $(\calX_{S'},\calD_{S'})^{\cc}(\O_{S'})$.) We define the sets $(\calX,\calD)^{\cc}_\mathbf{w}(\O_{S'})$ and $(\calX,\calD)^{\cc}_\mathbf{s}(\O_{S'})$ of weak and strong $S'$-integral Campana points analogously. When $S = S'$, we refer to $S'$-integral Campana points as \emph{(global) Campana points}, ditto for their weak and strong counterparts.
\end{mydef}

\begin{note}
The second definition above is the definition of Campana points used in the recent paper of Pieropan, Smeets, Tanimoto and V\'arilly-Alvarado \cite{PSTVA}, while the third definition (strong Campana points) is the definition of Campana points employed by the third author in \cite{STR}. The latter definition was adopted in order to reconcile differences between the Manin-type conjecture \cite[Conj.~1.1,~p.~59]{PSTVA} and the counting result \cite[Thm.~1.4]{STR} for certain mildly singular orbifolds associated to norm forms. By separating these two notions, we hope to clarify future work on Campana points.
\end{note}

\subsection{Darmon points}

We now define Darmon points, following ideas in \cite{DAR}. While Campana points correspond to $m$-full values, Darmon points correspond to $m$th powers.

\begin{mydef} \label{def:dar}
Let $v \not\in S$, and assume that $n_v(\calD_\alpha,P) = 0$ for all $\epsilon_\alpha = 1$.
\begin{enumerate}[label=(\roman*)]
\item We say that $P$ is a $v$-adic \emph{Darmon point} if
\[
m_\alpha \mid n_v\left(\calD_{\alpha},P\right)
\]
for all $\epsilon_\alpha < 1$.
We denote the set of local Darmon points by $(\calX,\calD)^{\dd}(\O_v)$.
\item We say that $P$ is a $v$-adic \emph{strong Darmon point} if for all $\epsilon_\alpha < 1$, we have
\[
m_\alpha \mid n_v\left(\calD_{\beta_v},P\right) \textrm{ for all $\beta_v \mid \alpha$.}
\]
We denote the set of local strong Darmon points by $(\calX,\calD)^{\dd}_\mathbf{s}(\O_v)$.
\end{enumerate}
For $v \in S$ we say that $P$ is a \emph{$v$-adic Darmon/strong Darmon point} if $P \in (X \setminus D_{\inf}) (K_v)$.

Given $S' \supset S$, we say that $P \in (X \setminus D_{\inf}) (K)$ is an \emph{($S'$-integral) Darmon point} if it is a $v$-adic Darmon point for all $v \in \Omega_K \setminus S'$.
We denote the set of $S'$-integral Darmon points by $(\calX,\calD)^{\dd}(\O_{S'})$ and define $(\calX,\calD)^{\dd}_\mathbf{s}(\O_{S'})$ analogously. When $S = S'$, we refer to $S'$-integral Darmon/strong Darmon points as \emph{(global) Darmon/strong Darmon points}. 
\end{mydef}

In this paper, we do not study strong Darmon points. However, it would be interesting to know whether the notion of strong Darmon points is in some sense appropriate for mildly singular orbifolds as that of strong Campana points proved to be in \cite{STR}.

\subsection{Stacky points}

As noted in the work of Nasserden and Xiao \cite[\S2]{NX1}, \cite[\S3]{NX2}, there is a correspondence between $M$-curves (Campana orbifolds of dimension $1$) and \emph{stacky curves} on the geometric level: the data of a stacky curve (namely the sizes of the automorphism groups attached to its stacky points) gives rise to an orbifold structure on its coarse moduli space, with a similar construction in the other direction. One may wonder whether such a correspondence exists on the arithmetic level, in particular whether the integral points of a stacky curve correspond to Darmon points on the associated $M$-curve. The following example, inspired by \cite[\S5]{BP}, gives an instance in which this is the case.

\begin{example} \label{ex:dar}
Endow $\mathbb{P}^1_{\mathbb{Z}}$ with the action of $\mu_2$ sending $[x_0:x_1]$ to $[x_0:-x_1]$. This action fixes the points $[1:0]$ and $[0:1]$. Let $\scrX$ be the quotient stack $[\mathbb{P}^1/\mu_2]$, and denote by $\rho: \scrX \rightarrow \mathbb{P}^1$ the associated morphism of stacks. Note that $\scrX$ is isomorphic to the root stack $\mathbb{P}^1[\sqrt{[1:0]},\sqrt{[0:1]}]$, see \cite[Thm.~4.10]{DK}. The coarse moduli space of $\scrX$ is $\mathbb{P}^1$. As observed in \cite[\S5]{BP}, the set of $R$-points of $\scrX$ for $R \in \{\mathbb{Z}_p,\mathbb{Z}\}$ decomposes as
\[
\scrX\left(R\right) = \coprod_{t \in R^*/R^{*2}}\pi_t \left(\scrY_t\left(R\right)\right)
\]
for $\scrY_t$ the twisted cover coming from the action of the $\mu_2$-torsor $T_t = \Spec R[u]/(u^2 - t)$. Set $D := \(1-1/2\)\left([1:0] + [0:1]\right)$ and let $\sD$ be the closure of $D$ in $\mathbb{P}^1_\ZZ$. For $R \in \{\mathbb{Z}_p,\mathbb{Z}\}$,
\[
\left(\rho \circ \pi_t\right)\left(\scrY_t\left(R\right)\right) = (\mathbb{P}^1_{\ZZ},\sD)^{\dd}\left(R\right).
\]
\end{example}

Later (Remark \ref{rmk:rootstack}), we will see a tantalising hint that this stacky--Darmon correspondence extends to the Brauer--Manin obstruction.

\subsection{Integral points}

Lastly, we introduce the following terminology for local and global points seen in the integral Brauer--Manin obstruction of Colliot-Th\'el\`ene and Xu \cite{CTX}.

\begin{mydef}
Let $U$ be a variety (not necessarily proper) over a number field $K$. Let $S \subset \Omega_K$ be a finite subset containing $\Omega_K^\infty$, and let $\calU$ be an $\O_S$-model of $U$. For $v \not\in S$, we say that $P \in U(K_v)$ is a \emph{$v$-adic integral point} of $U$ (with respect to the model $\calU$) if $P$ extends to an $\O_v$-point $\calP \in \calU(\O_v)$. By abuse of notation, we write $\calU(\O_v)$ for the set of $v$-adic integral points of $U$ with respect to $\calU$, considered as a subset of $U(K_v)$. We say that $P \in U(K)$ is an \emph{($S$-)integral point of $U$} with respect to the model $\calU$ if $P \in \calU(\O_v)$ for all $v \not\in S$. We denote the set of integral points of $\calU$ by $\calU(\O_S)$.
\end{mydef}

For an orbifold $(X,D)$ over $K$ with $\O_S$-model $(\calX,\calD)$, we have the following inclusions:
\begin{center}
\begin{tikzcd}[row sep=small, column sep = tiny]
& \left(\calX\setminus\calD_{\red}\right)\left(\O_v\right) \arrow[r, phantom, sloped, "\subset"]
& \left(\calX,\calD\right)^{\dd}_{\mathbf{s}}\left(\O_v\right) \arrow[r, phantom, sloped, "\subset"] \arrow[d, phantom, sloped, "\subset"]
& \left(\calX,\calD\right)^{\cc}_{\mathbf{s}}\left(\O_v\right) \arrow[d, phantom, sloped, "\subset"] & \\

&
& \left(\calX,\calD\right)^{\dd}\left(\O_v\right) \arrow[r, phantom, sloped, "\subset"]
& \left(\calX,\calD\right)^{\cc}\left(\O_v\right) \arrow[d, phantom, sloped, "\subset"] & \\

& 	& 	& \left(\calX,\calD\right)^{\cc}_{\mathbf{w}}\left(\O_v\right) \arrow[d, phantom, sloped, "\subset"] & \\
&   &   & X\left(K_v\right). \\
\end{tikzcd}
\end{center}

\subsection{Adelic points and local-global principles} \label{subsec:semi-integral-adeles}

In this section we introduce adelic spaces for semi-integral points with a view to formulating semi-integral versions of both the Brauer--Manin obstruction and local-global principles.

\begin{note} \label{note:*}
We shall discuss both Campana and Darmon points at once by using the symbol $*$. The corresponding statement for Campana (respectively, Darmon) points will then be recovered by substituting $\cc$ or Campana (respectively, $\dd$ or Darmon) for $*$. For example, $*$-points specialises to Campana or Darmon points.
\end{note}

Recall that $(X,D)$ is a Campana orbifold over a number field $K$ with $\O_S$-model $(\calX,\calD)$ for some finite set $S$ with $\Omega_K^\infty \subset S \subset \Omega_K$ and $U := X \setminus D_{\textrm{red}}$. Set $\calU := \calX \setminus \calD_{\textrm{red}}$.

\begin{note}
Write $D = \sum_{\alpha \in \scrA}\left(1-1/m_\alpha\right)D_\alpha$. While any point on $D_{\fin}$ trivially satisfies the semi-integral condition for any $D_\alpha$ on which it lies, it may fail to satisfy the semi-integral condition for other $\QQ$-components of $D$. A simple example is furnished by the orbifold $\left(\mathbb{P}^2_\QQ,\left(1-1/2\right)\left(D_0 + D_1\right)\right)$, where $D_i = Z\left(x_i\right)$ for $\mathbb{P}^2$ with coordinates $x_0,x_1,x_2$. Taking the obvious $\ZZ$-model $\left(\mathbb{P}^2_\ZZ,\left(1-1/2\right)\left(\calD_0 + \calD_1\right)\right)$, we observe that $n_p\left(\calD_0,P\right) = \infty$ and $n_p\left(\calD_1,P\right) = 1$ for $P = [0:p:1]$ and a chosen prime $p$.
\end{note}

By definition, the set of semi-integral points $(\sX, \sD)^*(\sO_S)$ may contain some $K$-rational points of $D_{\fin}$. Note that when $\epsilon_\alpha < 1$ for all $\alpha \in \scrA$ (i.e.\ when $(X,D)$ is \emph{Kawamata log terminal}), we have $D_{\fin} = D_{\red}$. While the inclusion of these points may seem natural from the topological or analytic point of view (as well as from the number-theoretic perspective), their inclusion becomes troublesome when studying local-global principles. Loosely, this is because $D_{\fin}$ is of lower dimension than $X$, and its $K$-rational points may behave very differently from the $K$-rational and $S$-integral points on $X \setminus D_{\red}$. To resolve this, we break the set of semi-integral points into two disjoint subsets:
\[
  \(\sX,\sD\)^*\(\sO_S\)
  = \(\(\sX,\sD\)^*\(\sO_S\) \cap U(K)\) \coprod \(\(\sX,\sD\)^*\(\sO_S\) \cap  \sD_{\fin}\left( \sO_S \right)\).
\]
We shall refer to $(\sX,\sD)_{\str}^*(\sO_S) \coloneqq  (\sX,\sD)^*(\sO_S) \cap U(K)$ as the set of \emph{strict semi-integral points}, while $(\sX,\sD)_{\tr}^*(\sO_S) \coloneqq (\sX,\sD)^*(\sO_S) \cap \sD_{\fin}(\sO_S)$ will be referred to as the set of \emph{non-strict semi-integral points}. We extend this philosophy via the analogous definitions for local semi-integral points. Each global strict semi-integral point lies in $\sU(\sO_v)$ for all but finitely many places $v$. This motivates the following definition.

\begin{mydef} \label{def:adelicCampana}
Given a finite set $T \subset \Omega_K$, we define the sets of \emph{strict $T$-adelic semi-integral points} and \emph{non-strict $T$-adelic semi-integral points} of $(\sX,\sD)$ (denoted by $(\sX,\sD)_{\str}^*(\AA_{K, S}^T)$ and by $(\sX,\sD)_{\tr}^*(\AA_{K, S}^T)$ respectively) by
  \[
    \begin{split}
      \(\sX,\sD\)_{\str}^*\(\AA_{K, S}^T\)
      &:= \prod_{v \in S \setminus (S \cap T)} U(K_v) \times \sideset{}{'}\prod_{v \in \Omega_K \setminus (S\cup T)}\( \(\sX,\sD\)^*\(\sO_v\) \cap U(K_v),\sU\(\sO_v\)\) 
      , \\
      \(\sX,\sD\)_{\tr}^*\(\AA_{K, S}^T\)
      &:=\bigcup_{\substack{\alpha \in \scrA \\ \epsilon_\alpha \neq 0,1 }}
      \left( \prod_{v \in S \setminus (S \cap T)} D_{\alpha, \fin}(K_v) 
      \times \prod_{v \in \Omega_K \setminus (S\cup T)} \sD_{\alpha, \fin}\(\sO_v\) \cap \(\sX,\sD\)^*\(\sO_v\) 
      \right).
    \end{split}
  \]
We equip $(\sX,\sD)_{\str}^*(\AA_{K, S}^T)$ with the restricted product topology, while $\(\sX,\sD\)_{\tr}^*\(\AA_{K, S}^T\)$ is considered as a subset of $\prod_{v \in S \setminus (S \cap T)}D_{\fin}(K_v) \times \prod_{v \in \Omega_K \setminus (S\cup T)} \sD_{\fin}(\sO_v)$ whose product topology it inherits. We define the set of \emph{$T$-adelic semi-integral points} as the coproduct
  \[
    \(\sX,\sD\)^*\(\AA_{K, S}^T\) 
    := \(\sX,\sD\)_{\str}^*\(\AA_{K, S}^T\) \coprod \(\sX,\sD\)_{\tr}^*\(\AA_{K, S}^T\),
  \]
endowed with the coproduct topology, which we call the \emph{adelic topology} on $(\sX,\sD)^*(\A^T_{K, S})$.

We define the \emph{adelic semi-integral points} to be the $T$-adelic semi-integral points for $T = \emptyset$ and denote them by $(\calX,\calD)^*(\AA_{K, S})$. We define the sets of \emph{strict} and \emph{non-strict adelic semi-integral points} analogously. We omit $S$ when it is clear from context.
\end{mydef} 

\begin{note}
\label{note:adelic-semi-integral}
We make the following remarks on adelic semi-integral points.
\begin{enumerate}
\item Note that $\(\sX,\sD\)^*\(\AA_{K, S}^T\) = \(\sX,\sD\)^*\(\AA_{K, S \cup T}^T\)$, so, going from adelic points to $T$-adelic points, we lose information about integrality/the model at places in $T$, i.e.\ such information is seen only away from $S \cup T$. Thus, when studying $S'$-integral points in the $T$-adelic points (in particular, when studying weak/strong approximation), it is natural to assume that $S' \supset S \cup T$, i.e.\ to ``drop'' integrality at places in $T$, as we do in Definition \ref{def:CWA}.

\item By our coproduct definition, a $T$-adelic semi-integral point is either strict at every place not in $S$ or non-strict at every place not in $S$.

\item When $D = \epsilon_\alpha D_\alpha$ with $\epsilon_\alpha \neq 0,1$, every local (respectively, global, $T$-adelic) point on $D_\alpha$ is a non-strict semi-integral point on $\left(\calX,\calD\right)$, so that $\(\sX,\sD\)_{\ns}^*\(\AA_K^T\) = \calD_{\alpha}\left(\AA_K^T\right)$. This is the situation in which we find ourselves for both of the two main orbifold families in this paper.
\end{enumerate}
\end{note}

We see that $(\sX,\sD)_{\str}^*(\O_{S \cup T})$ embeds diagonally in $(\sX,\sD)_{\str}^*(\AA_{K, S}^T)$, ditto $(\sX,\sD)_{\ns}^*(\O_{S \cup T})$ in $(\sX,\sD)_{\ns}^*(\AA_{K, S}^T)$, thus so does $(\sX,\sD)^*(\sO_{S \cup T})$ in $(\sX,\sD)^*(\AA_{K, S}^T)$.

The sets $\sU(\sO_v)$ and $\sD_{\alpha, \fin}(\sO_v)$ are compact for any non-archimedean place $v$, as is $(\sX,\sD)^*(\sO_v)$ by \cite[Rem.~3.5]{PSTVA} (the Darmon case follows in the same way as the Campana case), while $U(K_v)$ and $D_{\alpha, \fin}(K_v)$ are locally compact and Hausdorff, thus our definition of adelic space still enjoys the property of being locally compact.

Note that the inclusion of sets
\[
  \(\sX, \sD\)^*\(\AA_{K, S}^T\) 
  \subset \prod_{v \in \Omega_K \setminus T} \(X \setminus D_{\inf}\)(K_v)
\]
makes $(\sX, \sD)^*(\AA_{K, S}^T)$ into a topological space with respect to the subspace topology that we shall refer to as the product topology. In general, the adelic topology on $(\sX, \sD)^*(\AA_{K, S}^T)$ is finer than the subspace topology. For example, $\prod_{v \in S \setminus (S \cap T)} U(K_v) \times \prod_{v \in \Omega_K \setminus (S \cup T)} \sU(\sO_v)$ is open in $(\sX, \sD)_{\str}^*(\AA_{K, S}^T)$ with respect to the adelic topology. On the other hand, the same set is not open with respect to the subspace topology unless $\sU(\sO_v) = (\sX, \sD)^*(\sO_v) \cap U(K_v)$ for all but finitely many $v \in \Omega_K \setminus (S \cup T)$, and the local semi-integral conditions will generally make this equality fail.

\begin{note}
We point out another change in notation. Adelic Campana points were defined in \cite{NS} and \cite{PSTVA} as the unrestricted product of local Campana points, while here, we restrict the product with respect to local integral points. As we shall see in the next section, working with the restricted product allows us to define a semi-integral version of the Brauer--Manin pairing, which is one of the primary motivations of this work.
\end{note}

\begin{remark}\label{rem:adelicspace}
We can define even finer adelic spaces than the one given in Definition \ref{def:adelicCampana}. The main motivation is that doing so allows us to define possibly stronger obstructions to local-global principles, but at the cost of increasing the computational difficulty (see Remark \ref{rmk:BMO}). Let the irreducible components of $D_{\red}$ be $D_1,\dots,D_r$, and let $\Sigma$ be a collection of subsets of $\{1,\ldots,r\}$. For any non-empty set $A\in \Sigma$ which is maximal (in the sense that there does not exist $A'\in \Sigma$ such that $A\subsetneq A'$), set $D_A:=\cap_{i\in A}D_i\setminus D_{\inf}$. On the other hand, set $D_{\emptyset} := U$. Iteratively define, for any $A\in\Sigma$, the locally closed set
\[
D_{A}:=\(\bigcap_{i\in A}D_i\)\setminus \(D_{\inf} \cup \bigcup_{A\subsetneq A'\in \Sigma}D_{A'}\).
\]
Assume that $\cup_{A \in \Sigma}D_A = X\setminus D_{\inf}$, which is equivalent to $\emptyset,\{i\}\in\Sigma$ for $i=1,\ldots,r$. Set
\[
(\sX,\sD)_A^*(\A_{K, S}^T)\colonequals \prod_{v \in \Omega_K \setminus (S\cup T)}\(\sD_A\(\sO_v\) \cap \(\sX,\sD\)^*(\O_v)\)
      \times \prod_{v \in S \setminus (S \cap T)} D_A(K_v),
\]
and finally
\[
(\sX,\sD)_{\Sigma}^*(\A_{K, S}^T)\colonequals\bigcup_{A\in \Sigma}(\sX,\sD)_A^*(\A_{K, S}^T).
\]
The smallest set of adelic points occurs when we take $\Sigma$ to consist of all subsets. Roughly, this set of adelic points consists of collection of local semi-integral points which all lie on certain intersections of the divisors $D_i$ dictated by the set $\Sigma$. Note that the definition of adelic set as above is recovered by choosing $\Sigma$ to consist of $\emptyset$ and all one-element subsets  of $\{1,\dots,r\}$. We choose this as our definition as we believe it offers a nice compromise between computability and strength of the Brauer--Manin obstruction. 
\end{remark}

\begin{remark} \label{rem:1comp}
Another approach for semi-integral points for orbifold divisors with several components comes from the following observation: writing $D = \sum_{i=1}^r \epsilon_i D_i$, we have
\[
\(\sX,\sD\)^*\(R\) = \cap_{i=1}^r\(\sX,\sD_i\)^*\(R\) 
\]
for $R \in \{\O_v,\O_S\}$, and so we may in principle study the semi-integral points of $\(\sX,\sD\)$ by studying those of the one-component sub-orbifolds $\(\sX,\sD_i\)$.
\end{remark}

Recall that if $V$ is a variety over $K$ with an $S$-integral model $\calV$, its adelic space $V(\AA_{K})$ consists of those points $(P_v) \in \prod_{v \in \Omega_K} V(K_v)$ such that $P_v$ is a $v$-adic integral point for all but finitely many $v\notin S$. Recall also that $\sV(\A_{\O_{S}}) = \prod_{v \in S} V(K_v) \times \prod_{v \in \Omega_K \setminus S} \sV(\sO_v)$. As before, an added superscript $T$ will stand for the projection to $T$-adeles in both adelic spaces above. We will shortly be concerned with the question of whether the existence of local semi-integral points at all places implies the existence of a global semi-integral point. To that end recall the following definition for rational points. 

\begin{mydef}
We say that a family of varieties $\mathscr{F}$ over a number field $K$ satisfies the \emph{Hasse principle} if the following implication holds for all $V \in \mathscr{F}$:
\[
V\left(\A_{K}\right) \neq \emptyset \implies V\left(K\right) \neq \emptyset.
\]
\end{mydef}

Before defining approximation properties for semi-integral points, we first recall the definitions of weak and strong approximation for varieties over number fields. We also introduce strong approximation for integral models, which is closely related to the version of strong approximation for integral models seen in \cite{LM} and the approximation property seen in \cite[Thm.~6.1]{CTX}.

\begin{mydef} \label{def:approx}
  Let $V$ be a variety (not necessarily projective) over a number field $K$ with an $S$-integral model $\calV$. Let $T$ be a finite set of places of $K$.
  \begin{enumerate}[label=(\roman*)]
    \item We say that $V$ satisfies \emph{weak approximation off $T$} if $V(K)$ is dense in $\prod_{v \not\in T} V(K_v)$ under the product topology. We say that $V$ satisfies \emph{weak weak approximation} (\emph{WWA}) if it satisfies weak approximation off some set $T$, and \emph{weak approximation} (\emph{WA}) if it satisfies weak approximation off $T = \emptyset$.

    \item We say that $V$ satisfies \emph{strong approximation off $T$} if $V(K)$ is dense in $V(\A_{K}^T)$ under the adelic topology. We say that $V$ satisfies \emph{strong approximation} (\emph{SA}) if it satisfies strong approximation off $T = \emptyset$.

    \item In the setting of the previous definition, we say that $\sV$ satisfies \emph{integral strong approximation off $T$} if $\sV\left(\O_{S \cup T}\right)$ is dense in $\sV(\A_{\O_{S \cup T}}^T) = \sV(\A_{\O_S}^T)$. We say that $\sV$ satisfies integral strong approximation (ISA) if it satisfies ISA off $T = \emptyset$. 
  \end{enumerate}
\end{mydef}

Note that, for a proper variety $V$, we have $V(\A_K) = \prod_{v \in \Omega_K}V(K_v)$, hence strong approximation is equivalent to weak approximation.

Roughly, weak approximation is satisfied if any finite collection of local points can be simultaneously approximated by a rational point, while strong approximation if satisfied if such an approximation can be found so that the approximating point is locally integral at all other places. In particular, strong approximation for $V$ off $T$ implies that any finite collection of local \textbf{integral} points can be approximated by an $S \cup T$-integral point of $\sV$, i.e.\ that integral strong approximation for $\sV$ off $T$ holds, but for the converse implication to hold, as we shall see in Proposition \ref{prop:BMOequiv}, one needs ISA off $T$ to hold for any extension of the model $\sV$ to $\sO_{S'}$, where $S'$ is a finite set of places containing $S$.

We now give the definitions of the above local-global principles for semi-integral points.

\begin{mydef} 
  \label{def:CHP}
  Let $\mathscr{F}$ be a collection of $\O_S$-models for Campana orbifolds. We say that $\mathscr{F}$ satisfies the \emph{Hasse principle for $*$-points ($*$HP)} if the following implication holds for all $(\calX,\calD) \in \mathscr{F}$:
  \[
    \left(\calX,\calD\right)^*\left(\A_{K, S}\right) \neq \emptyset \implies \left(\calX,\calD\right)^*\left(\O_S\right) \neq \emptyset.
  \]
\end{mydef}

In a similar fashion, the Hasse principle may be defined for strict or non-strict semi-integral points by restricting to the desired component of the adelic space.

Recall that Definition~\ref{def:adelicCampana} ensures that $(\sX, \sD)^*(\AA_{K, S \cup T}^T) = (\sX, \sD)^*(\AA_{K, S}^T)$. This leads to the following definitions of weak and of strong approximation.

\begin{mydef} \label{def:CWA}
Let $(\calX,\calD)$ be an $\O_S$-model for a Campana orbifold $(X,D)$, and let $T$ be a finite set of places of $K$.
\begin{enumerate}[label=(\roman*)]
\item We say that $(\calX,\calD)$ satisfies \emph{weak approximation for $*$-points off $T$} (\emph{$*$WA off $T$}) if $(\calX,\calD)^*(\O_{S \cup T})$ is dense in $(\sX, \sD)^*(\AA_{K, S}^T)$ for the product topology. We say that $(\calX,\calD)$ satisfies \emph{weak weak approximation for $*$-points} (\emph{$*$WWA}) if it satisfies $*$WA off some finite set of places $T$. We say that $(\calX,\calD)$ satisfies \emph{weak approximation for $*$-points ($*$WA)} if it satisfies $*$WA off $T = \emptyset$.
\item We say that $(\sX, \sD)$ satisfies \emph{strong approximation for $*$-points off $T$} (\emph{$*$SA} off $T$) if $(\sX, \sD)^*(\sO_{S \cup T})$ is dense in $(\sX, \sD)^*(\AA_{K, S}^T)$ for the adelic topology.
We say that $(\calX,\calD)$ satisfies \emph{strong approximation for $*$-points ($*$SA)} if it satisfies $*$SA off $T = \emptyset$.
\end{enumerate}
\end{mydef}

Concretely, if we are given $(P_v) \in (\sX, \sD)_{\str}^*(\AA_{K, S}^T)$, then strong approximation for $*$-points off $T$ means that for any finite set of places $R$ containing $S \cup T$, and for any open sets $W_v \subset U(K_v)$ for $v \in S \setminus T$ and $W_v \subset (\sX, \sD)_{\str}^*(\O_v)$ for $v \in R \setminus (S \cup T)$ with $P_v \in W_v$  for all $v \in R \setminus T$, there is a point of $(\sX, \sD)_{\str}^*(\sO_{S\cup T})$ inside
\[
  \prod_{v\in R\setminus T} W_v\times \prod_{v\notin R} \sU(\O_v).
\]
Similarly, if $(P_v) \in (\sX, \sD)_{\tr}^*(\AA_{K, S}^T)$, then we can approximate $(P_v)$ by an $S \cup T$-semi-integral point lying in the non-strict component of the $T$-adelic space.

See Proposition \ref{prop:BMOequiv} for relationships between different notions of strong approximation. One may define strict and non-strict strong, weak and weak weak approximation in a similar fashion, but it appears that the results about the density of semi-integral points are more interesting when the whole adelic space is considered.

\section{Semi-integral Brauer--Manin obstruction} \label{section:BMO}

In this section we develop a semi-integral version of the Brauer--Manin obstruction in order to study the Hasse principle and strong and weak approximation in this context.

\subsection{Classical Brauer--Manin obstruction}

We begin with a brief overview of the Brauer--Manin obstruction for rational points (an excellent reference for which is \cite[\S13.3]{CTS2}) and the version for integral points developed by Colliot--Th\'el\`ene and Xu in \cite{CTX} in the case where the underlying variety is not necessarily smooth.

\begin{mydef}
Let $V$ be a scheme. The \emph{Brauer group} of $V$ is $\Br V = H^2_{\text{\'et}}(V,\mathbb{G}_m)$. Given a commutative ring $R$, we write $\Br R$ for $\Br \Spec R$.
\end{mydef}

Here, the subscript \'et denotes \'etale cohomology and $\mathbb{G}_m$ denotes the \'etale sheaf defined by $\mathbb{G}_m(U) = \O_U(U)^*$ for $U$ an \'etale $V$-scheme.

By functoriality, given a commutative ring $R$, any $R$-valued point $P: \Spec R \rightarrow V$ of $V$ induces a map $\Br V \rightarrow \Br R$. Fixing $\calA \in \Br V$, we obtain an \emph{evaluation map} $\ev_\calA: V(R) \rightarrow \Br R$. We denote the image of $P \in V(R)$ by $\calA(P)$.

\begin{lemma} \cite[Thm.~1.5.34(i)]{POO}
Let $k$ be a local field. There is a canonical injection $\Br k \hookrightarrow \QQ/\ZZ$ with image
\[
\begin{cases}
\mathbb{Q} / \mathbb{Z} & \textrm{ if $k$ is non-archimedean}, \\
\mathbb{Z} / 2\mathbb{Z}  & \textrm{ if $k = \mathbb{R}$}, \\
0 & \textrm{ if $k = \mathbb{C}$}.
\end{cases}
\]
\end{lemma}

Given a place $v$ of a number field $K$, denote by $\inv_v: \Br K_v \hookrightarrow \QQ/\ZZ$ the injection from the previous lemma.

Now let $V$ be a variety over $K$ (not necessarily smooth or proper). The utility of the Brauer group in local-global questions for rational points comes from the following commutative diagram for each $\calA \in \Br V$:

\begin{center}
\begin{tikzcd}
& & V\left(K\right) \arrow[r,hook] \arrow[d,"\ev_\calA"] & V\left(\A_K\right) \arrow[d,"\ev_\calA"] & & \\
& 0 \arrow[r] & \Br K \arrow[r] & \bigoplus_v\Br K_v \arrow[r,"\sum_v \inv_v"] & \QQ/\ZZ \arrow[r] & 0.
\end{tikzcd}
\end{center}

Exactness of the bottom row follows from the Albert--Brauer--Hasse--Noether theorem and global class field theory. Commutativity implies that the set $V(\A_K)^{\Br} = \{(P_v)_v \in V(\A_K): \sum_v \inv_v \calA(P_v) = 0 \text{ for all } \calA \in \Br V\}$ contains $V(K)$, hence we have inclusions
\[
V\left(K\right) \subset V\left(\A_K\right)^{\Br} \subset V\left(\A_K\right).
\]
In particular, emptiness of $V(\A_K)^{\Br}$ implies that $V(K)$ is empty. Further, the pairing of adelic points and Brauer elements is continuous \cite[Prop.~10.5.2]{CTS2}, hence $V(\A_K)^{\Br}$ is closed in $V(\A_K)$. These observations motivate the following definitions.

\begin{mydef}
Let $V$ be a variety over a number field $K$ and let $T$ be a finite set of places of $K$.
\begin{enumerate}[label=(\roman*)]
\item We say that $V$ has a \emph{Brauer--Manin obstruction to the Hasse principle} if $V(\A_K) \neq \emptyset$ but $V(\A_K)^{\Br} = \emptyset$.
    
\item We say that $V$ has a \emph{Brauer--Manin obstruction to weak approximation off $T$} if the closure of the projection of $V(\A_K)^{\Br}$ to $\prod_{v \notin T} V(K_v)$ (with respect to the product topology) does not equal the whole space.

\item We say that $V$ has a \emph{Brauer--Manin obstruction to strong approximation off $T$} if the projection of $V(\A_K)^{\Br}$ to $V(\A_K^T)$ is not equal to $V(\A_K^T)$.
\end{enumerate}
We omit ``off $T$'' above if we may take $T = \emptyset$.
\end{mydef}

Let $\sV$ be an integral model of $V$ over $\sO_S$. As explained in the previous subsection, if $V$ is not proper, then the sets $V(K_v)$ and $\sV(\sO_v)$ do not necessary coincide. Under the inclusion $\sV(\AA_{\sO_S}) = \prod_{v \in S } V(K_v) \times \prod_{v \in \Omega_K \setminus S} \sV(\sO_v) \subset V(\AA_K)$, one can restrict the Brauer--Manin pairing to obtain the integral Brauer--Manin set $\sV(\AA_{\sO_S})^{\Br} = \sV(\AA_{\sO_S}) \cap V(\A_K)^{\Br}$.
This gives a chain of inclusions 
\[
  \sV(\sO_S) \subset \sV(\AA_{\sO_S})^{\Br} \subset \sV(\AA_{\sO_S}).
\]
For $T \subset \Omega_K$ finite, let $\sV(\AA_{\sO_{S \cup T}}^T)^{\Br}$ denote the projection of $\sV(\AA_{\sO_{S \cup T}})^{\Br}$ to $\sV(\AA_{\sO_S}^T)$. We then have a corresponding chain of inclusions for $T$-adelic points:
\[
  \sV(\sO_{S \cup T}) \subset \sV(\AA_{\sO_{S \cup T}}^T)^{\Br} \subset \sV(\AA_{\sO_S}^T).
\]

\begin{note}
Note that $\sV(\AA_{\sO_{S\cup T}}^T) = \sV(\AA_{\sO_S}^T)$, but that the same does not necessary hold for projections of Brauer-Manin sets. It is possible that $\sV(\AA_{\sO_S}^T)^{\Br} \subsetneq \sV(\AA_{\sO_{S\cup T}}^T)^{\Br}$ if the local invariant maps of elements of $\Br V$ take values on $K_v$-points which are not obtained on $\sO_v$-points at places $v \in T$. To make things worse, $\sV(\AA_{\sO_S}^T)^{\Br}$ may not contain $\sV(\sO_{S \cup T})$ for the same reason. This is also the situation we find ourselves when it comes to semi-integral points. We make this note explicit with the following example. 

  Let $K = \QQ$, $S = \{\infty\}$, $T = \{5\}$ and $V$ be the smooth two-dimensional affine quadric
  \[
    V : 5x^2 - 25y^2 + 16z^2 = 1.
  \]
  Let $\sV$ be the natural choice of $\sO_S$-model of $V$. We have $\sO_S = \ZZ$ and $\sO_{S \cup T} = \ZZ[1/5]$. As we show in Section~\ref{section:fail} the quaternion algebra $\sA = (1 - 4z, 5)$ generates $\Br V/\Br \QQ = \ZZ / 2\ZZ$. Moreover, $\inv_p \sA : \sV(\Zp) \rightarrow \QQ/\Z$ vanishes at $p \neq 5$, while $\inv_5 \sA$ is equal to 1/2 at any point in $\sV(\ZZ_5)$. Thus $\sV(\AA_{\sO_S})^{\Br} = \emptyset$, hence $\sV(\AA_{\sO_S}^T)^{\Br} = \emptyset$ as well. On the other hand, it is clear that $P = (1, 2/5, 0) \in \sV(\sO_{S \cup T}) \subseteq \sV(\AA_{\sO_{S \cup T}})^{\Br}$, hence it projects down to a point in $\sV(\AA_{\sO_{S\cup T}}^T)^{\Br}$. We conclude that $\sV(\AA_{\sO_S}^T)^{\Br} \subsetneq \sV(\AA_{\sO_{S\cup T}}^T)^{\Br}$ and $\sV(\sO_{S \cup T}) \not\subseteq \sV(\AA_{\sO_{S}}^T)^{\Br}$.
\end{note}
  
It should be noted that if $V$ is not proper, then the sum of the local invariant maps is still well-defined, since for any element $\sA \in \Br V$, there exists a finite set of places outside of which $\inv_v \circ \ev_{\sA}$ vanishes on $\sV(\sO_v)$. Thus $\sV(\AA_{\sO_{S \cup T}}^T)^{\Br} \subset \sV(\AA_{\sO_S}^T)$ is closed. This leads to the following definitions.

\begin{mydef}
  Let $V$ be a variety over a number field $K$ with an $\sO_S$-integral model $\sV$.
  \begin{enumerate}[label=(\roman*)]
  \item We say that $\sV$ has a \emph{Brauer--Manin obstruction to the $S$-integral Hasse principle} if $\sV(\A_{\sO_S}) \neq \emptyset$ but $\sV(\A_{\sO_S})^{\Br} = \emptyset$.
  \item Let $T \subset \Omega_K$ be finite subset. We say that $\sV$ has a \emph{Brauer--Manin obstruction to integral strong approximation off $T$} if $\sV(\AA_{\sO_{S \cup T}}^T)^{\Br} \neq \sV (\AA_{\sO_S}^T)$.
  \end{enumerate}
\end{mydef}

We give an example of a cubic surface with Darmon points for any weight but no integral points because of a Brauer--Manin obstruction to the integral Hasse principle.

\begin{example}
  \label{ex:DMBO-does-not-imply-IBMO}
  Let $X \subset \PP^3_{\Q}$ be the cubic surface given by
  \[
    X: \quad y^2 z - (4x - z)(16x^2 + 20xz + 7z^2) - t^3 = 0.
  \] 
  Let $m \ge 1$ be an integer and denote by $D_m$ the $\Q$-divisor $(1 - 1/m)Z(t)$. Fix the integral model $(\sX, \sD_m)$ of $(X, D_m)$ given by the same equations over $\ZZ$. As explained in \cite[\S2]{KT} the quaternion Azumaya algebra $(-t^3/y^3, z/y)$ is a non-trivial Brauer element which gives a Brauer--Manin obstruction for integral points on $\sX \setminus \sD_{\red}$, since the local invariant maps vanish at all places except the $2$-adic one, where its value is $1/2$. However, for any $m \in \ZZ_{\ge 1}$ there is a point $[-4^{m-1}: 1: 0: 4^m] \in (\sX, \sD_m)^{\dd}(\ZZ)$.
\end{example}

\subsection{Orbifold Brauer--Manin}
To begin with, we have the inclusions
\[
  \sU\(\AA_{\sO_S}\) \subset \(\sX, \sD\)_{\str}^*\(\AA_{K, S}\) \subset U(\AA_K).
\]
Both of the Brauer--Manin obstructions for $K$-rational and $S$-integral points on $U$ are defined with respect to the same Brauer group, $\Br U$, but come with different choices of adelic space encoding the nature of the global points. This allows us to define a good notion of a Brauer--Manin obstruction for strict semi-integral points.

\begin{mydef} \label{def:SIBMO}
The inclusion $(\sX, \sD)_{\str}^*(\AA_{K, S}) \subset U(\AA_K)$ allows us to pair any strict semi-integral adelic point with $\Br U$ by restricting the existing Brauer--Manin pairing $\Br U \times U(\AA_K) \rightarrow \QQ/\ZZ$. We define the Brauer--Manin set $(\sX, \sD)_{\str}^*(\AA_{K, S})^{\Br} = (\sX, \sD)_{\str}^*(\AA_{K, S})^{\Br U}$ for strict $*$-points to be the right kernel of this pairing. Note that a point in $(\sX, \sD)_{\tr}^*(\AA_{K, S})$ may not pair with an element of $\Br U$, simply because such a point does not lie on $U$. Instead, we define the non-strict part of the Brauer--Manin set $\(\sX, \sD\)_{\tr}^*\(\AA_{K, S}\)^{\Br}$ as the preimage of $\bigcup_{\epsilon_\alpha \neq 0, 1} D_{\alpha, \fin}(\A_K)^{\Br D_{\alpha, \fin}}$ under the natural inclusion $\(\sX, \sD\)_{\tr}^*\(\AA_{K, S}\)\hookrightarrow \bigcup_{\epsilon_\alpha \neq 0, 1} D_{\alpha, \fin}(\A_K)$. We then define
\[
  \(\sX, \sD\)^*\(\AA_{K, S}\)^{\Br}
  := \(\sX, \sD\)_{\str}^*\(\AA_{K, S}\)^{\Br} \coprod \(\sX, \sD\)_{\tr}^*\(\AA_{K, S}\)^{\Br}.  
\]
\end{mydef}

Given subsets $B \subset \Br U$ and $B_\alpha \subset \Br D_{\alpha, \fin}$ for $\varepsilon_\alpha \neq 0,1$, we can define an intermediate obstruction set $(\sX,\sD)^*(\A_{K, S})^{B, (B_\alpha)_\alpha}$ by restricting the Brauer--Manin pairing to only $B$ for the strict points and taking the preimage of $\bigcup_{\epsilon_\alpha \neq 0, 1} D_{\alpha, \fin}(\A_K)^{B_\alpha}$ for the non-strict points. If $\(\sX, \sD\)_{\tr}^*\(\AA_{K, S}\)^{(B_\alpha)_\alpha}$ denotes the preimage of $\bigcup_{\epsilon_\alpha \neq 0, 1} D_{\alpha, \fin}(\A_K)^{B_\alpha}$ under the natural inclusion $\(\sX, \sD\)_{\tr}^*\(\AA_{K, S}\)\hookrightarrow \bigcup_{\epsilon_\alpha \neq 0, 1} D_{\alpha, \fin}(\A_K)$, then
\[
  (\sX,\sD)^*(\A_{K, S})^{B, (B_\alpha)_\alpha} := (\sX,\sD)_{\str}^*(\A_{K, S})^B \coprod \(\sX, \sD\)_{\tr}^*\(\AA_{K, S}\)^{(B_\alpha)_\alpha}.
\]

To study local-global principles, it will be useful to work with the projections of the various Brauer--Manin sets to the $T$-adeles.

\begin{mydef} \label{def:Br-T-adeles}
Given subsets $B \subset \Br U$ and $B_\alpha \subset \Br D_{\alpha, \fin}$ for $\varepsilon_\alpha \neq 0,1$, we define $(\sX, \sD)^*(\A_{K, S \cup T}^T)^{B, (B_\alpha)_\alpha}$ to be the projection of $(\sX, \sD)^*(\A_{K, S \cup T})^{B, (B_\alpha)_\alpha}$ to $(\sX, \sD)^*(\A_{K, S}^T)$ and similarly for strict, non-strict, integral and rational points.
\end{mydef}

\begin{remark}
\label{rem:BMOinterpolation}
Taking $D=0$, we recover the Brauer--Manin obstruction for rational points on $X$, while taking all $\epsilon_\alpha = 1$ (i.e.\ $D = D_{\red} = D_{\inf}$) we recover the integral Brauer--Manin obstruction for $\calU$.
\end{remark}

\begin{note} \label{note:proj}
Note that $(\sX, \sD)^*(\A_{K, S \cup T}^T)^{\Br}$ is the projection of $(\sX, \sD)^*(\A_{K, S \cup T})^{\Br}$ to the T-adeles $(\sX, \sD)^*(\A_{K, S}^T)$; one might otherwise suspect that $(\sX, \sD)^*(\A_{K, S \cup T}^T)^{\Br}$ denotes the set of all $T$-adelic points pairing to zero with the Brauer group. However, when defining the Brauer--Manin obstruction, one must take all places under consideration, otherwise not all rational points need lie in the so-defined $T$-adelic Brauer--Manin set. We have decided to incorporate the abuse of notation $(\sX, \sD)^*(\A_{K, S \cup T}^T)^{\Br}$ for the $T$-adelic projection of the Brauer--Manin set as it makes our notation throughout the article lighter.
\end{note}

\begin{remark}\label{rmk:BMO}
There is a slew of different natural obstruction sets one can define using the stratification of Remark \ref{rem:adelicspace}. Let $\Sigma$ be a collection of subsets of $\{1,\ldots,r\}$ as before such that $\bigcup_{A\in\Sigma}D_A=X\setminus D_{\inf}$. Then we may take
\[
(\calX,\calD)_{\Sigma}^*(\A^T_{K, S \cup T})^{\Br}=\bigcup_{A\in\Sigma}(\calX,\calD)_A^*(\A^T_{K, S \cup T})^{\Br D_A}
\]
where the last set is defined by taking adelic points orthogonal to $\Br D_A$ under the natural inclusion $(\sX,\sD)_A^*(\A^T_{K, S})\subset D_A(\A^T_K)$. If $\Sigma$ consists of all possible subsets, then the obstruction is strongest, but at the cost of being more computationally heavy as one needs to compute Brauer groups for smaller open sets.
\end{remark}

\begin{remark}
Under the ``one component at a time'' perspective of Remark \ref{rem:1comp}, the natural way of studying obstructions to local-global principles would be to consider the Brauer--Manin obstruction for each one-component sub-orbifold and compute the intersection $\cap_{i=1}^r \(\sX,\sD_i\)\(\A_{K, S}\)^{\Br}$ as an alternative Brauer--Manin set.
\end{remark}

The pairing for semi-integral points above is well-defined. If $(P_v)$ is a strict semi-integral adelic point, then all but finitely many $P_v$ lie in $\sU(\sO_v)$, so, pairing with $\sA \in \Br U$, the local invariant map will vanish for all but finitely many places, hence $\sum_{v \in \Omega_K} \inv_v \(\sA\(P_v\)\)$ is well-defined. A similar argument applies to non-strict semi-integral adelic points.

A non-strict global semi-integral point will pair to zero with the Brauer group of each $D_{\alpha, \fin}$ that it lies on. The inclusion $(\sX, \sD)_{\str}^*(\sO_S) \subset U(K)$ guarantees pairing to zero between strict global semi-integral points and $\Br U$. Thus we get a sequence of inclusions 
\[
  \(\sX, \sD\)^*\(\sO_S\) \subset \(\sX, \sD\)^*\(\AA_{K, S}\)^{\Br} \subset \(\sX, \sD\)^*\(\AA_{K, S}\),
\]
and therefore the Brauer--Manin obstruction can obstruct the Hasse principle for semi-integral points. This observation leads to the following definition.

\begin{mydef} \label{def:orbifoldBMOtoHP}
Let $(X,D)$ be a Campana orbifold over a number field $K$ with $\O_S$-model $(\calX,\calD)$ as above. We say that there is a \emph{Brauer--Manin obstruction to the Hasse principle for $*$-points} if $(\calX,\calD)^*(\A_{K, S}) \neq \emptyset$ but $(\calX,\calD)^*(\A_{K, S})^{\Br} = \emptyset$. Otherwise we say that there is no Brauer--Manin obstruction to the Hasse principle for $*$-points.
\end{mydef}

There is one striking difference between the above Brauer--Manin pairing and the classical Brauer--Manin pairings for rational points on proper varieties and for integral points. In the cases of rational and integral points, the Brauer--Manin pairings were defined in a way such that the set of places where $\inv_v \circ \ev_{\sA}$ is non-constant depends only on the Brauer element chosen. Moreover, for each Brauer element this set is finite. Here, this set of places can be infinite due to the definition of semi-integral points. Indeed, Harari proved in his thesis \cite[Thm. 2.1.1]{HARA} that the Brauer--Manin pairing can be non-trivial at infinitely many places for ramified Brauer classes. We show that similar statements hold for semi-integral points below. The proof presented here follows \cite[Thm.~13.4.1]{CTS2}.

Recall that $V := X \setminus D_{\inf}$ and that the weight of $D_\alpha\subset D_{\red}$ is $m_\alpha$. Given an integral divisor $Z\subset V$, let $\partial_Z\colon \Br U\to H^1(k(Z),\Q/\Z)$ denote the residue map.

\begin{mydef} \label{def:BrXD}
Set
\[
\Br (X,D)=\{\calA\in\Br U \ : \ m_\alpha\partial_{D_{\alpha,\fin}}(\calA)=0\ \forall \alpha \in \scrA \text{ with } \epsilon_\alpha \neq 0,1\}.
\]
\end{mydef}

\begin{remark}\label{rmk:rootstack}
When $D = \epsilon_\alpha D_\alpha$ with $\epsilon_\alpha < 1$, one may interpret $\Br\left(X,D\right)$ as the Brauer group of the associated root stack $\mathscr{X} = X[\sqrt[m_\alpha]{D_\alpha}]$. Indeed, let $\pi: \mathscr{X} \rightarrow X$ be the obvious morphism; denoting by $\mathscr{U}$ the preimage of $U = X \setminus D_\alpha$ in $\mathscr{X}$ (which is isomorphic to $U$) and setting $\mathscr{E}_\alpha = \mathscr{X} \setminus \mathscr{U}$, we have the following commutative diagram:
\begin{center}
\begin{tikzcd}
\Br U \arrow[r,"\partial_{D_\alpha}"] \arrow[d] & H^1\left(D_\alpha,\mathbb{Q}/\mathbb{Z}\right) \arrow[d,"\cdot m_\alpha"] \\
\Br \mathscr{U} \arrow[r,"\partial_E"] & H^1\left(\mathscr{E}_\alpha,\mathbb{Q}/\mathbb{Z}\right).
\end{tikzcd}
\end{center}
Then Grothendieck's purity theorem together with the fact that the left vertical arrow is an isomorphism implies $\Br\left(X,D\right) \cong \Br \mathscr{X}$.
\end{remark}

\begin{theorem} \label{thm:HarariA}
Let $(X,D)$ be a Campana orbifold over a number field $K$ with $\O_S$-model $(\calX,\calD)$. Let $U=X\setminus D_{\red}$ and $V = X \setminus D_{\inf}$. Let $\sA\in\Br U\setminus \Br V$ (resp. $\calA\in\Br U\setminus \Br(X,D)$). Then for infinitely many places $v\notin S$, there exists $P_v\in (\calX,\calD)^{\cc}_{\str}(\O_v)$ (resp. $P_v\in (\calX,\calD)^{\dd}_{\str}(\O_v)$) such that $\sA(P_v)\neq0$.
\end{theorem}

\begin{proof}
The proof is largely the same as Harari's formal lemma as given in \cite[Thm.~13.4.1]{CTS2}. The only deviation will be in the end when we introduce the Campana/Darmon conditions. We reproduce a sketch of the proof given in \emph{loc.\ cit.}\ for the reader's convenience.

Let $D_{\alpha,\fin} \subset D_{\fin}$ be an irreducible component where $\sA$ is ramified (if $\calA\notin \Br (X,D)$, choose $D_\alpha$ so that $m_\alpha \partial_{D_{\alpha,\fin}}(\calA)\neq0$). Following the proof of \cite[Thm.~13.4.1]{CTS2}, there exists a connected smooth affine curve $C=\Spec A\subset X$ where $C\cap D_\alpha=P$ is a transverse intersection and a closed point cut out by $f\in A$, and $P$ lies on no other component of $D$. The restriction of $\sA$ to $C$ (which we will also denote by $\calA$) has non-trivial residue $\chi\colonequals \partial_P(\sA)\in H^1(k(P),\Q/\Z)$ at $P$ whose order equals that of $\partial_{D_{\alpha,\fin}}(\calA)$. There exists an \'{e}tale open $q: Y\to C$ and $\xi\in H^1(Y,\Q/\Z)$ such that $\xi$ is the restriction of $\chi$ and $Q=q^{-1}(P)\xrightarrow\sim P$ is an isomorphism.

Let $\sA'=\sA-(f,\xi)\in \Br (Y\setminus Q)$, where $(f,\xi)$ is a cyclic algebra. This is unramified over $Q$, hence $\sA'\in \Br Y$. Let $\calC\subset \calX,\calY$ be models of $C,Y$ over $\calO_S$ and $\calP,\calQ$ be closures of $P,Q$ in $\calC,\calY$ respectively. Enlarging $S$ as necessary, we may assume that (i) $Y\to C$ spreads out to $\calY\to \calC$, (ii) $\xi$ and $\sA'$ come from $H^1(\calY,\Q/\Z)$ and $\Br \calY$ respectively, (iii) $f\in \O_S[\calC]$, (iv) $\calQ$ is integral and $\calQ\xrightarrow\sim \calP$, and (v) $\calP\to \Spec \O_S$ is finite and \'{e}tale.

By Chebotarev's density theorem, there are infinitely many $v\in \Omega_K \setminus S$ and $w \in \Omega_{k(P)}$ of degree $1$ over $v$ such that the restriction of $\chi$ in $H^1(k(P)_{w},\Q/\Z)$ is non-trivial.

Choosing such $v$, there exists an $\O_{v}$-point $N_{v}^0$ of $\calY$ inside the closed subset $\calQ$ mapping to a point $M_{v}^0\in \calC(\O_{v})$. The existence of such a point together with smoothness imply that there exist points $N_{v}\in \calQ(\O_{v})$ where $f(N_{v})$ has arbitrarily high $v$-adic valuation.

Let $n$ be the order of $\chi$. Choose a point $N_{v}\in \calY(\O_v)$ such that $v(f(N_{v}))\not\equiv 0\bmod n$. By the proof of \cite[Thm.~13.4.1]{CTS2}, for $M_{v}=q(N_{v})$, we have
\[
\inv_{v}\sA(M_{v}) \equiv v(f(N_{v})) \not\equiv 0 \bmod n.
\]
\begin{enumerate}
\item In the Campana case, the intersection multiplicity of $M_v$ with $D_\alpha$ is precisely $v(f(N_v))$ which can be chosen larger than $m_\alpha$ to satisfy the Campana condition at $D_\alpha$. Note that $M_v$ does not meet the other components of $D_{\red}$ since $M_v\subset \calC$.

\item In the Darmon case, note that we are restricted to having $v(f(N_v))\equiv 0\bmod m_\alpha$. However, our initial assumption implies $m_\alpha \chi\neq0$, so we may still choose $N_v$ so that $v(f(N_v))\not\equiv0\bmod n$. \qedhere
\end{enumerate}
\end{proof}

The previous theorem can be used to prove a statement for a finite collection of Brauer classes using a combinatorial argument as is shown in \cite[Cor. 2.6.1]{HARA}, \cite{CTS1}.

\begin{theorem}
\label{thm:HarariB}
Let $(X,D)$ be a Campana orbifold over a number field $K$ with $\O_S$-model $(\calX,\calD)$. Let $U=X\setminus D_{\red}$ and $V = X \setminus D_{\inf}$. Let $B\subset\Br U$ be a finite subgroup. Let $(P_v)\in (\calX,\calD)^{\cc}_{\str}(\A_{K, S})^{B\cap \Br V}$ (resp. $(P_v)\in (\calX,\calD)^{\dd}_{\str}(\A_{K, S})^{B\cap \Br (X,D)}$). Then for any finite set of places $S'\supset S$,  there exists $(Q_v)\in (\calX,\calD)^{\cc}_{\str}(\A_{K, S})^{B}$ (resp. $(Q_v)\in (\calX,\calD)^{\dd}_{\str}(\A_{K, S})^{B}$) with $Q_v=P_v$ for $v\in S'$. 
\end{theorem}

\begin{proof}
The same proof as \cite[Thm.~13.4.3]{CTS2} using Theorem \ref{thm:HarariA} applies here. To sum up ideas, if for example Campana points are under consideration, we first enlarge $S'$ so that if $v \notin S'$, then $\inv_v \sA(P_v) = 0$ for all $\sA \in B$ and $\inv_v \sA'(M_v) = 0$ for all $\sA' \in B \cap \Br V$ and $M_v \in \sV(\sO_v)$. It then suffices to find a set $T$ disjoint from $S'$ such that after modifying $(P_v)$ with local semi-integral points at places in $T$ we would, for all $\sA \in B$, have
\[
  \sum_{v \in S \cup T} \inv_v \sA(P_v) = 0.
\]
We can do so as in the proof of \cite[Thm.~13.4.3]{CTS2} in view of Theorem \ref{thm:HarariA}.
\end{proof}

The following lemma is a converse to Theorem \ref{thm:HarariA} for Darmon points when each $D_\alpha$ is regular away from its intersection with $D_{\inf}$.

\begin{lemma} \label{lema:trivpairing}
Let $\calA\in \Br (X,D)$ and suppose that $D_{\alpha,\fin}$ is regular for all $\epsilon_{\alpha} \neq 0,1$. For all but finitely many places $v$, the image of the map $\inv\circ \calA\colon (\calX,\calD)^{\dd}_{\str}(\O_v)\to \Q/\Z$ is zero.
\end{lemma}

\begin{proof}
Let $v$ be a place. We may assume $\calA$ extends to $\Br \(\calU_{\O_v}\)$ and $\(\calD_{\alpha,\fin}\)_{\O_v}$ is regular. Take any strict Darmon $\O_v$-point $\calP\colon \Spec \O_v\to \calX$ and set $n_\alpha := n_v(\calD_\alpha,\calP) = n_v(\calD_{\alpha,\fin},\calP)$. By \cite[Thm.~3.7.5]{CTS2}, the pullback $\calP^*\calA\in \Br K_v$ has residue over the closed point given by the image of $n_\alpha \partial_{\calD_{\alpha,\fin}}(\sA)$ under the map
\[
\calP^*: H^1_{\text{\'et}}(\calD_{\alpha,\fin},\QQ/\ZZ) \rightarrow H^1_{\text{\'et}}(\FF_v,\QQ/\ZZ).
\]
But $m_\alpha \mid n_\alpha$, thus $n_\alpha \partial_{\calD_{\alpha,\fin}}(\sA) = 0$ since $\calA\in \Br (X,D)$. Then $\calP^*\sA \in \Br \O_v = 0$.
\end{proof}

The non-constant behaviour of the local invariant map for elements of $\Br U$ at infinitely many places, as shown in Theorem \ref{thm:HarariA}, raises the question of whether the Brauer--Manin set is closed with respect to the different topologies. Our next result answers this question.

\begin{proposition} 
  \label{prop:B-Mset}
  Let $T \subset \Omega_K$ be a finite set of places of $K$. The following hold.
  \begin{enumerate}[label=\emph{(\roman*)}]
    \item \label{BMi}
    The sets $(\sX, \sD)_{\str}^*(\AA_{K, S})^{\Br}$ and $(\sX, \sD)_{\str}^*(\AA_{K, S}^T)^{\Br}$ are closed in the adelic topology. Additionally, $(\sX, \sD)_{\tr}^*(\AA_{K, S})^{\Br}$ is closed in the adelic topology and if $D_{\alpha,\fin}$ is regular for all $\epsilon_{\alpha} \neq 0,1$, then so is $(\sX, \sD)_{\tr}^*(\AA_{K, S}^T)^{\Br}$.
    
    \item \label{BMii} Let $\overline{(\sX, \sD)_{\str}^*(\AA_{K, S}^T)^{\Br}}$ denote the closure of $(\sX, \sD)_{\str}^*(\AA_{K, S}^T)^{\Br}$ in the product topology. Assume that $\Br U / \Br K$ is finite. Then
    \[
      \begin{split}
        \overline{(\sX, \sD)_{\str}^{\cc}(\AA_{K, S}^T)^{\Br}} &= (\sX, \sD)_{\str}^{\cc}(\AA_{K, S}^T)^{\Br V}, \\
        \overline{(\sX, \sD)_{\str}^{\dd}(\AA_{K, S}^T)^{\Br}} & \subseteq \overline{(\sX, \sD)_{\str}^{\dd}(\AA_{K, S}^T)^{\Br (X,D)}}.
      \end{split}
    \]
    If, moreover, $D_{\alpha,\fin}$ is regular for all $\epsilon_{\alpha} \neq 0,1$, then $(\sX, \sD)_{\str}^{\dd}(\AA_{K, S}^T)^{\Br (X,D)}$ is closed in the product topology and
    \[
      \overline{(\sX, \sD)_{\str}^{\dd}(\AA_{K, S}^T)^{\Br}} = (\sX, \sD)_{\str}^{\dd}(\AA_{K, S}^T)^{\Br (X,D)}.
    \]

    \item \label{B-Msetiii} Let $\overline{(\sX, \sD)_{\tr}^*(\AA_{K, S}^T)^{\Br}}$ denote the closure of $(\sX, \sD)_{\tr}^*(\AA_{K, S}^T)^{\Br}$ in the product topology. For $* \in \{\cc, \dd\}$ we have
    \[
      \overline{(\sX, \sD)_{\tr}^*(\AA_{K, S})^{\Br}} = (\sX, \sD)_{\tr}^*(\AA_{K, S})^{\Br}.
    \] 
    If $D_{\alpha,\fin}$ is regular for all $\epsilon_{\alpha} \neq 0,1$, we also have
    \[
      \overline{(\sX, \sD)_{\tr}^*(\AA_{K, S}^T)^{\Br}} = (\sX, \sD)_{\tr}^*(\AA_{K, S}^T)^{\Br}.
    \]
  \end{enumerate}
\end{proposition}

\begin{proof}

  If $T = \emptyset$ the statement of (i) follows from the fact that $U(\AA_K)^{\Br U}$ is closed in $U(\AA_K)$ since the evaluation map is continuous \cite[Prop.~10.5.2]{CTS2} and $(\sX, \sD)_{\str}^*(\AA_{K, S})^{\Br} = U(\AA_K)^{\Br U} \cap (\sX, \sD)_{\str}^*(\AA_K)$. The same argument applies to $(\sX, \sD)_{\tr}^*(\AA_{K, S})^{\Br}$ since by definition it is $(\sX,\sD)_{\tr}^{*}(\A_{K, S})\bigcap \(\bigcup_{\epsilon_\alpha\neq 0, 1} D_{\alpha, \fin} (\A_{K})^{\Br D_{\alpha, \fin}}\)$ and the union is finite as there are only finitely many such $D_{\alpha, \fin}$ in the support of $D$.

  If $T \neq \emptyset$ is an arbitrary finite subset of $\Omega_K$, for any $0 \le k_{\sA} \le \ord(\sA) - 1$, where $\ord(\sA)$ is the order of $\sA \in \Br U$, which is finite since $U$ is smooth, define $W_{k_{\sA}, T}$ by
    \[
      W_{k_{\sA}, T} = \left\{(P_v)_{v \notin T} \in (\sX, \sD)_{\str}^*(\AA_{K, S}^T) \ : \ \sum_{v \notin T} \inv_v \sA(P_v) = -\frac{k_{\sA}}{\ord(\sA)} \right\}.
    \]
  Each $W_{k_{\sA}, T}$ is a closed subset of $(\sX, \sD)_{\str}^*(\AA_{K, S}^T)$ as
  \[
    W_{k_{\sA}, T} 
    = \left\{(P_v)_{v \notin T} \in U(\AA_K^T) \ : \ \sum_{v \notin T} \inv_v \sA(P_v) = -\frac{k_{\sA}}{\ord(\sA)} \right\} \cap (\sX, \sD)_{\str}^*(\AA_{K, S}^T),
  \]
  and the first set is closed in $U(\AA_K^T)$ by continuity of the evaluation map \cite[Prop.~10.5.2]{CTS2}. It remains to notice that
  \[
    (\sX, \sD)_{\str}^*(\AA_{K, S}^T)^{\sA}
    = \bigcup_{\substack{0 \le k_{\sA} \le \ord(\sA) - 1 \\ \exists (P_v') \in (\sX, \sD)_{\str}^*(\AA_{K, S}) \text{ s.t. } \sum_{v \in T} \inv \sA(P_v') = k_{\sA}/\ord(\sA)}} W_{k_{\sA}, T}.
  \]
  Hence $(\sX, \sD)_{\str}^*(\AA_{K, S}^T)^{\sA}$ is closed, as it is a finite union of closed sets. It is clear, that $(\sX, \sD)_{\str}^*(\AA_{K, S}^T)^{\Br} = \cap_{\sA \in \Br U} (\sX, \sD)_{\str}^*(\AA_{K, S}^T)^{\sA}$. This confirms that $(\sX, \sD)_{\str}^*(\AA_{K, S}^T)^{\Br}$ is closed. 

  The same argument applies to $(\sX, \sD)_{\tr}^*(\AA_{K, S}^T)^{\Br}$ in view of 
  \[
    (\sX, \sD)_{\tr}^*(\AA_{K, S}^T)^{\Br} 
    = \bigcup_{\epsilon_\alpha\neq 0, 1} \( D_{\alpha, \fin} (\A_{K}^T)^{\Br D_{\alpha, \fin}} \cap (\sX,\sD)_{\tr}^{*}(\A_{K, S}^T)\)
  \]
  and the fact that the above union is finite. Only mild modifications are needed to take into account that we are dealing with non-strict points. In particular, 
  \[
    D_{\alpha, \fin} (\A_{K}^T)^{\Br D_{\alpha, \fin}} 
    = \bigcap_{\sA \in \Br D_{\alpha, \fin}} D_{\alpha, \fin} (\A_{K}^T)^{\sA}
  \] 
  and since all $\Br D_{\alpha, \fin}$ are torsion, each $D_{\alpha, \fin} (\A_{K}^T)^{\sA}$ is a finite union
  \[
    D_{\alpha, \fin} (\A_{K}^T)^{\sA} 
    = \bigcup_{\substack{0 \le k_{\sA} \le \ord(\sA) - 1 \\ \exists (P_v') \in D_{\alpha, \fin} (\A_K) \text{ s.t. } \sum_{v \in T} \inv \sA(P_v') = k_{\sA}/\ord(\sA)}} W_{k_{\sA}, T}', 
  \] 
  where
  \[
    W_{k_{\sA}, T}' 
    = \left\{(P_v)_{v \notin T} \in D_{\alpha, \fin} (\A_{K}^T) \ : \ \sum_{v \notin T} \inv_v \sA(P_v) = -\frac{k_{\sA}}{\ord(\sA)} \right\}.
  \]

  We now prove (ii) for Campana points. The definition of the Brauer--Manin pairing implies that $(\sX, \sD)_{\str}^*(\AA_{K, S}^T)^{\Br} \subset (\sX, \sD)_{\str}^{*}(\AA_{K, S}^T)^{\Br V}$ on taking projections away from $T$. Moreover, a similar argument to the one featuring in (i) shows that $(\sX, \sD)_{\str}^{*}(\AA_{K, S}^T)^{\Br V}$ is closed in the product topology as each semi-integral point is an integral point on $V$.
  
  Let $(P_v)\in (\sX, \sD)_{\str}^{\cc}(\AA_{K, S})^{\Br V}$. We show that any open set containing $(P_v)$ meets $(\sX, \sD)^{\cc}_{\str}(\AA_{K, S})^{\Br}$. Let $W$ be such an open set. By definition of the product topology, 
  \[ 
    W=\prod_{v\in S'} W_v \times \sideset{}{'}\prod_{v\notin S'}((\calX,\calD)_{\str}^*(\O_v),\sU(\O_v)) 
  \]
  where $S' \subset \Omega_K$ is a finite set containing $S$ and $W_v$ is an open set inside $U(K_v)$ for $v \in S$ or inside $(\calX,\calD)_{\str}^*(\O_v)$ for $v \in S'\setminus S$. Since $\Br V / \Br K$ is finite, $(\sX, \sD)_{\str}^{\cc}(\AA_{K, S})^{\Br V} = (\sX, \sD)_{\str}^{\cc}(\AA_{K, S}) \cap V(\A_{K})^{\Br V}$ is open, so we may assume that $W$ is small enough so that every adelic point inside it is orthogonal to $\Br V$. Our claim then follows directly from Theorem~\ref{thm:HarariB}, choosing $B\subset \Br U$ to be the subgroup generated by a finite set of generators for $\Br U/\Br V$, and taking projections away from $T$.
  
  The inclusion $(\sX, \sD)_{\str}^{\dd}(\AA_{K, S}^T)^{\Br} \subseteq (\sX, \sD)_{\str}^{\dd}(\AA_{K, S}^T)^{\Br (X,D)}$ follows as in the Campana case, thus the same inclusion applies to their closures in the product topology. Under the assumption that $D_{\alpha,\fin}$ is regular for all $\epsilon_{\alpha} \neq 0,1$, a similar argument for Darmon points using the Darmon counterpart of Theorem \ref{thm:HarariB} shows that the closure contains $(\calX,\calD)^{\dd}_{\str}(\A_{K, S}^T)^{\Br(X,D)}$. It follows from $(\calX,\calD)^{\dd}_{\str}(\A_{K, S}^T)^{\Br(X,D)} = \cap_{\sA \in \Br(X,D)} (\sX, \sD)_{\str}^*(\AA_{K, S}^T)^{\sA}$ and Lemma~\ref{lema:trivpairing} that this set is closed in the product topology. 

  Part (iii) follows from (i) and the fact that the product and the adelic semi-integral topologies coincide at the non-strict adelic set by definition.
\end{proof}

Proposition \ref{prop:B-Mset} allows us to define semi-integral Brauer--Manin obstructions to weak and strong approximation. 

\begin{mydef}
  \label{def:BMOtoLtG}
  Let $(X,D)$ be a Campana orbifold over a number field $K$ with $\O_S$-model $(\calX,\calD)$. Let $T \subset \Omega_K$ be finite. In the notation of Proposition~\ref{prop:B-Mset}, let
  \[
    \overline{(\sX, \sD)^*(\AA_{K, S \cup T}^T)^{\Br}} =
    \overline{(\sX, \sD)_{\str}^*(\AA_{K, S \cup T}^T)^{\Br}} \coprod \overline{(\sX, \sD)_{\tr}^*(\AA_{K, S \cup T}^T)^{\Br}}.
  \] 
  \begin{enumerate}[label=(\roman*)]
    \item We say that there is a \emph{Brauer--Manin obstruction to weak approximation off $T$ for $*$-points} if 
    \[
      \overline{(\sX, \sD)^*(\AA_{K, S \cup T}^T)^{\Br}} \neq (\sX, \sD)^*(\AA_{K, S}^T).
    \]
    \item We say that there is a \emph{Brauer--Manin obstruction to strong approximation off $T$ for $*$-points} if 
    \[
      (\sX, \sD)^*(\AA_{K, S \cup T}^T)^{\Br} \neq (\sX, \sD)^*(\AA_{K, S}^T).
    \]
  \end{enumerate}
  We shall omit ``off $T$'' in each definition above if $T = \emptyset$. 

  Alternatively, if the above assumptions are not satisfied we say that there is no Brauer--Manin obstruction to weak or strong approximation for semi-integral points respectively.
\end{mydef}

Before establishing Theorem~\ref{thm:main0} we provide a simple observation, which is well-known in the settings of rational and integral points, but which explains the need of the assumption on $T$ in Theorem~\hyperref[thm:main0]{\ref*{thm:main0}\emph{\ref*{m0iii}}--\emph{\ref*{m0iv}}}. 

\begin{proposition}
  \label{prop:values-of-inv-at-T}
  Let $\sA \in \Br U$. Then the following hold.
  \begin{enumerate}[label=\emph{(\roman*)}]
    \item \label{Ti} If $T \subset \Omega_K$ is such that $\sA$ is prolific at $T$ with respect to $U$, then 
    \[  
      (\sX, \sD)_{\str}^*(\AA_{K, S \cup T}^T)^{\sA} = (\sX, \sD)_{\str}^*(\AA_{K, S}^T).
    \]

    \item \label{Tii} Assume additionally that $\sA \in \Br U \setminus \Br V$ if $* = \cc$ or $\sA \in \Br U \setminus \Br(X, D)$ if $* = \dd$. If $T \subset \Omega_K$ is such that $\sA$ is not prolific at $T$ with respect to $U$, then 
    \[
      (\sX, \sD)_{\str}^*(\AA_{K, S \cup T}^T)^{\sA} \subsetneq (\sX, \sD)_{\str}^*(\AA_{K, S}^T).
    \]
  \end{enumerate}
\end{proposition}

\begin{proof}
  We begin with (i). Assume that $(\sX, \sD)_{\str}^*(\AA_{K, S \cup T}^T)^{\sA} \subsetneq (\sX, \sD)_{\str}^*(\AA_{K, S}^T)$ and let $(P_v)_{v \notin T}$ be a $T$-adele in the complement of $(\sX, \sD)_{\str}^*(\AA_{K, S \cup T}^T)^{\sA}$.  Then $(P_v)_{v \notin T}$ extends to $(P_v) \in (\sX, \sD)_{\str}^*(\AA_{K, S})$ such that $\sum_{v \in \Omega_K}\inv_v \sA (P_v) \neq 0$. On the other hand, since the sum of local invariant maps of $\sA$ over all places of $T$ takes all possible values, define $(P_v') \in (\sX, \sD)_{\str}^*(\AA_{K, S})$ as follows: set $P_v' := P_v$ for all $v \notin T$ and choose $(P_v')_{v \in T}$ so that $\sum_{v \in T}\inv_v \sA(P_v') = - \sum_{v \in \Omega_K \setminus T} \inv_v \sA(P_v)$. Thus $(P_v') \in (\sX, \sD)_{\str}^*(\AA_{K, S})^{\sA}$, hence it projects down to a point $(P_v')_{v \notin T} \in (\sX, \sD)_{\str}^*(\AA_{K, S \cup T}^T)^{\sA}$. But this image is the same as the image of $(P_v)$, as we have only made modifications at places of $T$, a contradiction.

  We continue with (ii). Assume that $\sA \in (\Br U)[n]$ for some $n \in \mathbb{Z}_{>0}$ and let $T \subset \Omega_K$ be such that $\sum_{v \in T}\inv_v \sA (P_v) \in \{k_1/n, \dots, k_r/n \}$ for any given $(P_v)$ in $(\sX, \sD)_{\str}^*(\AA_{K, S})$, where the set $\{k_1, \dots, k_r\}$ is a strict subset of $\{ 0, 1, \dots, n-1\}$. To see that $(\sX, \sD)_{\str}^*(\AA_{K, S \cup T}^T)^{\sA} \subsetneq (\sX, \sD)_{\str}^*(\AA_{K, S}^T)$ holds, choose $(P_v)_{v \notin T} \in (\sX, \sD)_{\str}^*(\AA_{K, S}^T)$ so that $-\sum_{v \in \Omega_K \setminus T} \inv_v \sA(P_v)$ does not lie in $\{k_1/n, \dots, k_r/n\}$. The existence of such $(P_v)$ is verified by the same argument as the one used in the proof of \cite[Thm.~13.4.3]{CTS2} using Theorem~\ref{thm:HarariA}. It is then clear that any extension $(P_v)$ of $(P_v)_{v \notin T}$ satisfies $\sum_{v \in \Omega_K} \inv_v \sA(P_v) \neq 0$, thus $(P_v)_{v \notin T} \in (\sX, \sD)_{\str}^*(\AA_{K, S}^T) \setminus (\sX, \sD)_{\str}^*(\AA_{K, S \cup T}^T)^{\sA}$ which proves (ii).
\end{proof}

We are now in position to prove Theorem~\ref{thm:main0} from the introduction.

\begin{proof}[Proof of Theorem \ref{thm:main0}]
Parts (i) and (ii) follow directly from Proposition~\hyperref[prop:B-Mset]{\ref*{prop:B-Mset}\emph{\ref*{BMii}}}, while (iii) and (iv) are implied by Proposition~\hyperref[prop:B-Mset]{\ref*{prop:B-Mset}\emph{\ref*{BMi}}} and Proposition~\hyperref[prop:values-of-inv-at-T]{\ref*{prop:values-of-inv-at-T}\emph{\ref*{Tii}}}. 
\end{proof}

The following proposition gives the relationship for strong approximation and the sufficiency of Brauer--Manin obstruction with respect to the different notions of integral, semi-integral and rational points. 

\begin{proposition}
  \label{prop:BMOequiv}
  Let $(X,D)$ be a Campana orbifold over a number field $K$ and let $(\calX,\calD)$ be an $\calO_S$-model for some finite set $S \subset \Omega_K$ containing $\Omega_K^\infty$. Let $T \subset \Omega_K$ be finite. Then (under the convention in Definition~\ref{def:orbifoldBMOtoHP}) the following statements are equivalent.
  \begin{enumerate}[label=\emph{(\roman*)}]
    \item $\sU(\O_{S'\cup T})$ is dense in $\sU(\A_{\sO_{S' \cup T}}^T)^{\Br}$ (respectively, $\sU(\A_{\sO_{S'}}^T)$) for any finite set of places $S'$ containing $S$.

    \item $(\calX,\calD)_{\str}^{\dd}(\O_{S' \cup T})$ is dense in $(\sX, \sD)_{\str}^{\dd}(\AA_{K, S' \cup T}^T)^{\Br}$ (respectively, $(\sX, \sD)_{\str}^{\dd}(\AA_{K, S' \cup T}^T)$) for any finite set of places $S'$ containing $S$. 

    \item $(\calX,\calD)_{\str}^{\cc}(\O_{S' \cup T})$ is dense in $(\sX, \sD)_{\str}^{\cc}(\AA_{K, S' \cup T}^T)^{\Br}$ (respectively, $(\sX, \sD)_{\str}^{\cc}(\AA_{K, S' \cup T}^T)$) for any finite set of places $S'$ containing $S$.

    \item $U(K)$ is dense in $U(\A_K^T)^{\Br}$ (respectively, $U(\A_K^T)$).
  \end{enumerate}
\end{proposition}

\begin{proof}
We first prove that (iv)$\implies$(iii)$\implies$(ii)$\implies$(i). We have the containments
\begin{equation}
  \label{eqn:chain}
  \sU(\A_{\sO_{S' \cup T}}^T)^{\Br}\subseteq (\sX, \sD)_{\str}^{\dd}(\AA_{K, S' \cup T}^T)^{\Br}\subseteq (\sX, \sD)_{\str}^{\cc}(\AA_{K, S' \cup T}^T)^{\Br}\subseteq U(\A_K^T)^{\Br}
\end{equation}
of open subsets, and similarly without the $\Br$ superscript. Intersecting each set with $U(K)$ produces the following chain: 
\[
  \sU(\sO_{S' \cup T}) \subseteq (\sX, \sD)_{\str}^{\dd}(\O_{S' \cup T}) \subseteq (\sX, \sD)_{\str}^{\cc}(\O_{S' \cup T})\subseteq U(K).
\]
If $U(K)$ has dense intersection with one of the adelic sets in \eqref{eqn:chain}, then it has dense intersection with the smaller adelic sets also. Noting that enlarging $S$ to $S'$ enlarges the set of (semi-)integral points, we obtain (iv)$\implies$(iii)$\implies$(ii)$\implies$(i).

We now show that (i)$\implies$(iv) and conclude the proof. Let $(M_v)_{v\notin T} \in U(\A^T_K)^{\Br}$. Then $(M_v)_{v\notin T}$ extends to some $(M_v) \in U(\A_K)^{\Br}$. Let $S'\supset S$ be a finite set of places outside of $T$ for which we want to find a point in $U(K)$ approximating $M_v$ for $v\in S'$. Enlarging $S'$ if necessary, we assume that $M_v\in \sU(\O_v)$ for every $v\notin S' \cup T$. Thus $(M_v)_{v\notin T} \in \sU(\AA_{\sO_{S' \cup T}}^T)^{\Br}$. Using (i), we obtain a point $M \in \sU(\O_{S' \cup T})$ that approximates $(M_v)_{v\in S'}$ arbitrarily well. The fact that $M\in U(K)$ follows from the inclusion $\sU(\O_{S' \cup T}) \subset U(K)$.
\end{proof}

\section{Semi-integral points under maps} \label{section:maps}

In this section we consider the behaviour of semi-integral points under maps. Note that intersection multiplicity and pullback of effective Cartier divisors are both defined via fibre products, and so one may hope that they are compatible. The following proposition shows that this is the case over all but finitely many places.

\begin{proposition} \label{prop:compat}
Let $\phi: Y \rightarrow X$ be a dominant morphism of smooth proper varieties over a number field $K$, and let $D_\alpha \subset X$ be an effective Cartier divisor on $X$. Let $v \not\in \Omega_K^\infty$ be a place of $K$ such that there exist $\O_v$-models $\calX$ and $\calY$ of $X$ and $Y$ respectively. Suppose that the following conditions hold:
\begin{enumerate}[label=\emph{(\alph*)}]
\item $\phi$ spreads out to a dominant morphism $\Phi: \calY \rightarrow \calX$.
\item The closure $\calD_\alpha$ of $D_\alpha$ in $\calX$ is an effective Cartier divisor.
\end{enumerate}
Let $y\in Y(K_v)$. Write $\Phi^*\calD_\alpha=\sum t_\beta \calF_\beta$ where $t_\beta>0$ and $\calF_\beta \subset \calY$ are irreducible divisors. 
\begin{enumerate}[label=\emph{(\roman*)}]
\item Intersection multiplicity is additive on components: for $\calD_{\beta_v}, \beta_v \mid \alpha$ the reduced $\O_v$-components of $\calD_\alpha$ and $e_{\beta_v}$ the multiplicity of  $\calD_{\beta_v}$, we have, for all $x \in X(K_v)$,
\[
n_v\left(\calD_\alpha,x\right) = \sum_{\beta_v \mid \alpha}e_{\beta_v} n_v\left(\calD_{\beta_v},x\right). \label{prop:compati}
\]
\item Intersection multiplicity commutes with pullback: we have the equality
\[
n_v\left(\Phi^*\calD_\alpha,y\right) = n_v\left(\calD_\alpha,\phi\left(y\right)\right).
\]
In particular, we have $n_v(\calD_\alpha,\phi(y))\geq t_\beta n_v(\calF_\beta,y)$ for all $\beta$.
\label{prop:compatii}
\end{enumerate}
\end{proposition}

Note that (b) holds when the special fibre of $\calX$ is locally factorial, e.g.\ regular.

\begin{proof}

Let $x \in X(K_v)$. Since intersection multiplicity can be computed locally, we may assume that $\overline{y}\in \calY=\Spec B$ and $\overline{x}\in \calX=\Spec A$ such that $\calD_\alpha = \Spec A/(t)$. Let $f_x: A \rightarrow \O_v$ be the ring morphism corresponding to $\overline{x}$. Then $\overline{x}^*\calD_\alpha = \Spec \O_v/ (f_x(t))$.

Now we prove part (i). Shrinking $\Spec A$ further if necessary, we assume that for each $\beta_v$, there exists $t_{\beta_v} \in A$ such that $\calD_{\beta_v}=\Spec A/\left(t_{\beta_v}\right)$. Then we have $t = \prod_{\beta_v \mid \alpha} t_{\beta_v}^{e_{\beta_v}}$, hence we have $f_x(t) = \prod_{{\beta_v} \mid \alpha} (f_x(t_{\beta_v}))^{e_{\beta_v}}$, from which the result follows.

Next, we prove part (ii). Set $x := \phi(y)$. Integrality of $\calX$ and $\calY$ together and dominance of $\Phi$ collectively imply that the pullback of $\calD_\alpha$ is Cartier (see \cite[Tag~01WQ]{SP}). Let $\varphi\colon A\to B$ be the map on rings corresponding to $\Phi$. Define $f_y: B \rightarrow \O_v$ analogously to $f_x$. Then $\Phi^*\calD_\alpha=\Spec B/(\varphi(t))$, and

\[
n_v(\calD_\alpha,x)=v(f_x(t))=v(f_y(\varphi(t))) = n_v\left(\Phi^*\calD_\alpha,y\right).
\]

Indeed, we have the following picture (where $\overline{y} \in \calY\left(\O_v\right)$ denotes the closure of $y$):

\begin{equation}\label{cd:cubes}
\begin{tikzcd}[column sep=0.2cm,row sep=0.2cm]
& \Spec \O_v/\left(f_y\left(\varphi\left(t\right)\right)\right) \ar[rr] \ar[ld] \ar[dd] && \Spec B/\left(\varphi\left(t\right)\right) \ar[dd] \ar[ld]\\
\Spec \O_v/\left(f_x\left(t\right)\right) \ar[rr, crossing over] \ar[dd] && \Spec A/\left(t\right) \\
& \Spec \O_v \ar[rr, "\hspace{-1em} \overline{y}"] \ar[ld] && \Spec B. \ar[ld, "\Phi"] \\
\Spec \O_v \ar[rr,swap,"\overline{x}"] && \Spec A \ar[uu,  crossing over, leftarrow]
\end{tikzcd}
\end{equation}
Applying the equivalence of the categories of affine schemes and of rings to the bottom square, we have $f_x = f_y \circ \varphi$. The final inequality follows from part (i).
\end{proof}

As noted in \cite[\S3.6]{PSTVA}, the quantitative behaviour of semi-integral points is sensitive to birational modifications. This is perhaps unsurprising, since spreading out a map of varieties to a map of models may alter its properties, and one must also consider the effect of the map on divisors via pushforward and pullback. Here, we develop conditions on maps precisely to ensure that they behave well on semi-integral points geometrically.

First, we give Campana's definition of a morphism of orbifolds over a field $k$ (see \cite[Defs.~3.4,~3.7]{CAM11a}). Let $(X,D)$ and $(Y,F)$ be Campana orbifolds over a field $k$. Write $D = \sum_{\alpha \in \scrA} \(1-1/m_{\alpha}\) D_{\alpha}$ and $F = \sum_{\beta \in \scrB} \(1-1/n_{\beta}\) F_{\beta}$.

\begin{mydef} \label{def:OM}
An \emph{orbifold morphism} $f: (Y,F) \rightarrow (X,D)$ is a morphism $f: Y \rightarrow X$ such that
\begin{enumerate}[label=(\roman*)]
\item $f(Y) \not\subset D_{\textrm{red}}$, and 
\item If $f(F_{\beta}) \subset D_\alpha$, then $n_{\beta} \geq m_{\alpha}/t_{\alpha,\beta}$ for $t_{\alpha,\beta}$ the multiplicity of $F_{\beta}$ in $f^*D_\alpha$.
\end{enumerate}
We say that $f$ is \emph{divisible} if, in (ii), we may replace $n_{\beta} \geq m_{\alpha}/t_{\alpha,\beta}$ by $m_{\alpha} \mid t_{\alpha,\beta}n_{\beta}$. 
\end{mydef}

We make a further definition for birational morphisms. As is well-known (see \cite[Thm.~8.1.24]{LIU}), any birational morphism $f: Y \rightarrow X$ of projective varieties can be identified as the blowup of $X$ in some closed subscheme $Z$, which we call the \emph{centre} of $f$.

\begin{mydef}
An \emph{elementary birational orbifold morphism (EBOM)} is an orbifold morphism $f: (Y,F) \rightarrow (X,D)$ such that
\begin{enumerate}[label=(\roman*)]
\item $f: Y \rightarrow X$ is a birational morphism, and
\item $f_*(F) = D$.
\end{enumerate}
Here, $f_*(F)$ denotes the $\mathbb{Q}$-linear extension of proper pushforward as in \cite[Tag~02R3]{SP}. Since the continuous image of an irreducible space is irreducible, we may write
\[
F = \sum_{\alpha \in \scrA}\left(1-\frac{1}{m_\alpha}\right)\widetilde{D_\alpha} + \sum_{i=1}^r\left(1-\frac{1}{n_i}\right)E_i,
\]
where $\widetilde{D}_\alpha$ is the strict transform of $D_\alpha$, the $E_i$ are exceptional divisors and the weights $n_i$ satisfy $n_i\geq \max_{f\left(E_i\right) \subset D_\alpha}m_\alpha/t_{\alpha,i}$ for $t_{\alpha,i}$ the multiplicity of $E_i$ in $f^*D_\alpha$ when $f\left(E_i\right) \subset D_\alpha$.

In particular, when the centre of $f$ (viewed as the blowup of a closed subscheme of $X$) is disjoint with $D_{\red}$, the restriction on the $n_i$ is simply $n_i \geq 1$.
\end{mydef}

In particular, given a birational morphism to an orbifold, there exists a ``minimal'' orbifold structure on the source making the morphism an EBOM, defined as follows.

\begin{mydef}\label{def:induced}
Let $(X,D)$ be a Campana orbifold. Given a birational morphism $\varphi\colon\widetilde{X}\to X$, the {\it induced orbifold structure from $\varphi$} on $\widetilde{X}$ is the orbifold $(\widetilde{X},\widetilde{D})$ defined as follows: write $D = \sum_{\alpha \in \scrA} \(1-1/m_\alpha\)D_{\alpha}$, and denote by $E_1,\dots,E_r$ the irreducible components of the exceptional divisor of $\varphi$. For each $\alpha \in \scrA$, write $\varphi^*D_\alpha = \widetilde{D_\alpha} + \sum_{i=1}^r e_{\alpha,i} E_i$ with $\widetilde{D}_{\alpha}$ the strict transform of $D_{\alpha}$, and define $n_i := \max \( \{\lceil m_\alpha/e_{\alpha,i} \rceil \mid e_{\alpha,i}>0\} \cup \{1\} \)$. Finally, define $\widetilde{D} := \sum_{\alpha \in \scrA} \(1-1/m_\alpha\) \widetilde{D_\alpha} + \sum_{i=1}^r\(1-1/n_i\)E_i$.
\end{mydef}

We give two examples of orbifold isomorphisms which illustrate that semi-integrality of points is not always preserved by such maps.

\begin{example} \label{example:P1}
Consider the orbifold automorphism $f$ on $\(\mathbb{P}^1,\(1-1/2\)[0:1]\)$ given by $f([x:y]) = [x:2y]$. Using the obvious model $(\mathbb{P}^1_\Z,\(1-1/2\)[0:1])$, we see that $f$ fails to preserve $2$-adic Campana points: indeed, $P \in \mathbb{P}^1(\Q_2)$ is a local Campana point if and only if, choosing coprime $2$-adic integer coordinates $P = [x:y]$, we have $2^2 \mid x$. Then $[4:1]$ is a $2$-adic Campana point, but $f([4:1]) = [2:1]$ is not. Note that $f$ does not spread out over the prime $2$. To see this, note that the points $[0:1], [2:1]$ have the same reduction modulo $2$, while their images ($[0:1]$ and $[1:1]$ respectively) do not.
\end{example} 

\begin{example}
Consider the orbifold $(C,D)$, where
\[
C: x^2 + y^2 = 9z^2 \subset \mathbb{P}^2, \quad D = \left(1-\frac{1}{3}\right)Z\left(x-y\right).
\]
Taking the obvious model $(\calC,\calD)$ in $\mathbb{P}^2_{\Z}$, we have $(\calC,\calD)^{\cc}(\Z_3) = \emptyset$. Indeed, suppose we have $P \in C(\mathbb{Q}_3)$ with coprime $\mathbb{Z}_3$-coordinates $[x:y:z]$. Note that $v_3(x-y) \geq 1$. Geometrically, the special fibre of $\calC$ is a singular non-split conic with unique $\mathbb{F}_3$-point $[0:0:1]$, so every $\mathbb{Q}_3$-point must reduce to this point. On the other hand, note that if $v_3(x-y) \geq 3$, then $y = x + 27x'$ for some $x' \in \mathbb{Z}_3$, hence $2x^2 \equiv 9z^2 \pmod{27}$; since $3 \nmid z$ and $3 \mid x$ (as seen by considering the defining equation modulo $3$), this implies that $2$ is a quadratic residue modulo $3$, which is false, hence we deduce that $v_3(x-y) \leq 2$.

On the other hand, $C$ is isomorphic to $\mathbb{P}^1$ via the isomorphism $[x:y:z] \mapsto [3z-x:y]$ for $[x:y:z] \neq [3:0:1]$, $[3:0:1] \mapsto [0:1]$ (this is the projection from $[3:0:1] \in C(\Q)$ onto the line $z=0 \subset \mathbb{P}^2$, which we identify with $\mathbb{P}^1$). This takes $Z(x-y)$ to the divisor cut out by the equation $(x+y)^2 = 2y^2$ on $\mathbb{P}^1$. Further composing by the automorphism $[x:y] \mapsto [x+y:y]$, we obtain an ``orbifold isomorphism''
\[
\begin{aligned}
g: \(C,\left(1-\frac{1}{3}\right)Z\left(x-y\right)\) & \rightarrow \(\mathbb{P}^1,\left(1-\frac{1}{3}\right)Z\left(x^2 - 2y^2\right)\), \\
[x:y:z] & \mapsto
\begin{cases}
[3z-x+y:y] & \textrm{ if } [x:y:z] \neq [3:0:1], \\
[1:1] & \textrm{ if } [x:y:z] = [3:0:1].
\end{cases}
\end{aligned}
\]
In this context, we think of the orbifold divisor as being associated to the norm form for the obvious basis of $\mathbb{Q}(\sqrt{2})$, and we have $3$-adic local integral points: they correspond to coprime integers $x$ and $y$ such that $x^2 -2y^2$ is not divisible by $3$, e.g.\ $x = 1$, $y=0$.

In particular, while $g^{-1}$ is an orbifold isomorphism, its source has many $\mathbb{Z}$-Campana points, while its target has none. As with the previous example, the chosen isomorphism $g$ does not spread out over the prime $3$: note that $[3:0:1],[0:3:1] \in C(\Q)$ have the same $3$-adic reduction, but $g([3:0:1]) = [1:1]$ and $g([0:3:1]) = [2:1]$ do not.
\end{example}

However, throwing away finitely many places, EBOMs preserve Campana points.

\begin{proposition}\label{prop:ebom}
Let $f\colon (Y,F)\to (X,D)$ be an EBOM over a number field $K$. Let $(\calY,\calF)$ and $(\calX,\calD)$ be $\O_S$-models for $(Y,F)$ and $(X,D)$ respectively for some finite $S \subset \Omega_K$ containing $\Omega_K^\infty$. Let $Z \subset X$ and $E \subset Y$ denote the centre and exceptional divisor of $f$ respectively. Then there is a finite set of places $T \supset S$ such that the following are true.
\begin{enumerate}[label=\emph{(\roman*)}]
\item \label{ebomi} Let $\calZ$ be the closure of $Z$ in $\calX$. The points in $\left(\calX,\calD\right)^*\left(\O_T\right)$ (resp.\ $\(\calX \setminus \calD_{\red}\)\(\O_T\)$) whose closures do not meet $\calZ_T$ are in $f((\calY,\calF)^*(\calO_T))$ (resp.\ $f\(\(\calY \setminus \calF_{\red}\)\(\O_T\)\)$).
\item \label{ebomii} If $Z \cap D_{\red} = \emptyset$, then $f(\calY,\calF)^{*}(\calO_T)\subset (\calX,\calD)^{*}(\calO_T)$ and $f(\calY\setminus \calF_{\red})(\calO_T)\subset (\calX\setminus \calD_{\red})(\calO_T)$.
\item \label{ebomiii} Assume that, for each irreducible component $E_j$ of $E$ and each component $D_\alpha$ of $D$, we have $f(E_j) \cap D_\alpha = \emptyset$ or $f(E_j) \subset D_\alpha$. Then $f(\calY,\calF)^{\cc}(\calO_T)\subset (\calX,\calD)^{\cc}(\calO_T)$ and $f(\calY\setminus\calF_{\red})(\calO_T)\subset (\calX\setminus\calD_{\red})(\calO_T)$.
\end{enumerate}
\end{proposition}

\begin{proof}
Choosing $T \supset S$ sufficiently large, it follows from spreading out (see \cite[\S3.2]{POO}) that $\calY_T \cong (\Bl_\calZ \calX)_T$. Note that $f$ spreads out canonically to a birational morphism $\overline{f}: \calY_T \rightarrow \calX_T$ with centre $\calZ$. Letting $\calE$ be the closure of $E$ in $Y$, we see that $\overline{f}$ restricts to an isomorphism $\calU\colonequals (\calY \setminus \calE)_T \xrightarrow{\sim} \calV\colonequals(\calX \setminus \calZ)_T$.

First we prove (i). Let $x \in (\calX,\calD)^*(\O_T)$ or $\(\calX\setminus\calD_{\red}\)^*\(\O_T\)$ be such that $\overline{x}_T$ lies in $\calV$. Letting $y \in Y(K)$ be the unique preimage of $x$, we have $\overline{y}_T \in \calU$. For any component $F_\beta$ of $F$, we have either that $F_\beta \subset E$ (i.e.\ $F_\beta$ is contracted by $f$) or else that there exists a unique component $D_\alpha$ of $D$ such that $f(F_\beta) = D_{\alpha}$. In the former case, we have $\calF_\beta \subset \calE$, hence $n_v(\calF_\beta,y) = 0$ for $v \not\in T$. In the latter case, we see that $\overline{f}$ restricts to an isomorphism $\calF_\beta \setminus (\calF_\beta \cap \calE) \xrightarrow{\sim} \calD_\alpha \setminus (\calD_\alpha \cap \calZ)$, from which we deduce that $n_v(\calF_\beta,y) = n_v(\calD_\alpha,x)$ for $v \not\in T$. Since the weights $m_\alpha$ and $n_\beta$ are equal, we deduce (i).

We now turn to (ii) and (iii). Note that the result in (ii) for integral points follows immediately from the isomorphism $\calU \cong \calV$ and the inclusion $\left(\calD_{\red}\right)_T \subset \calV$ upon enlarging $T$. Let $y \in (\calY,\calF)^{*}(\O_T)$, and set $x := f(y)$. Let $v \in \Omega_K \setminus T$. Our goal is to prove, under the hypotheses specified in each part, that $x \in (\calX,\calD)^{*}(\O_v)$.

Let $D_\alpha$ be a component of $D$ with weight $m_{\alpha}$. We must show that $\overline{x}$ intersects $\calD_\alpha$ with the correct multiplicity. If $n_v(\calD_\alpha,x)=0$, then there is nothing to prove, so assume for the rest of the proof that $n_v(\calD_\alpha,x)>0$. We have either $\overline{x}_v\in \calU_v(\F_v)$ or $\overline{x}_v\in \calZ_v(\F_v)$.

Assume that $\overline{x}_v\in \calU_v(\F_v)$. There is a unique component $F_{\beta}$ of $F$ which is the strict transform of $D_\alpha$ with weight $n_\beta = m_\alpha$. Moreover, $n_v(\calD_\alpha,x) = n_v(\calF_\beta,y)$ by Proposition \hyperref[prop:compat]{\ref*{prop:compat}\emph{\ref*{prop:compatii}}}. In particular, if $y$ satisfies the local Campana, Darmon or integrality condition at $v$ with respect to $\calF_\beta$, then the same is true of $x$ with respect to $\calD_\alpha$ (which amounts to a contradiction to $n_v(\calD_\alpha,x)>0$ when $m_\alpha = \infty$). To complete the proof of (ii) and (iii), we may therefore assume that $\overline{x}_v\in \calZ_v(\F_v)$.

We first make some observations regarding exceptional and orbifold divisors. Let $Z\cap D_\alpha$ be the scheme-theoretic intersection, and denote by $\overline{Z\cap D_\alpha}$ its closure in $\calX$. There exist finitely many $v$ with $\left(\overline{Z\cap D_\alpha}\right)_v\neq (\calZ\cap\calD_\alpha)_v$. Let $T$ contain those places for each $\alpha\in \scrA$.

Denoting by $E_j, j=1,\dots,r$ the irreducible components of the exceptional divisor $E \subset Y$, note that the irreducible components of $\calE$ are the closures $\calE_j = \cl_{\calY}E_j$, $j=1,\dots,r$. In particular, setting $Z_j = f(E_j)$ and $\calZ_j = \cl_\calX Z_j$, we have $\calZ_T = \cup_{j=1}^r \left(\calZ_j\right)_T$. For $v \not\in T$, we deduce that $\calZ_v = \cup_{j=1}^r (\calZ_j)_v$. Note that there does not necessarily exist a bijection between the irreducible schemes $Z_j \subset Z$ and the irreducible components of $Z$.

Recall that we are under the assumption $\overline{x}_v\in \calZ_v(\F_v)$. By the previous paragraph, we deduce that $\overline{x}_v \in (\calZ_j)_v(\F_v)$ for some $j \in \{1,\dots,r\}$. Since $\overline{x}_v\in(\calD_\alpha)_v(\F_v)$ also, we deduce that $(\calZ_j \cap \calD_\alpha)_v \neq \emptyset$. Since $(\calZ_j \cap \calD_\alpha)_v = \cl_\calX(Z_j \cap D_\alpha)$ (in turn since $v \not\in T$), we deduce that $Z_j \cap D_\alpha = f\left(E_j\right) \cap D_\alpha \neq \emptyset$. This contradicts the hypothesis of (ii), and so we deduce (ii) from the observations in the case $\overline{x}_v\in \calU_v(\F_v)$ above.

To finish the proof of (iii), it follows from (ii) that we may assume that $f\left(E_j\right) \subset D_\alpha$. Since $f$ is an orbifold morphism, we deduce that $E_j = F_\beta$ for some component $F_\beta$ of $F$ and that the weight $1-1/n_\beta$ attached to $F_\beta$ satisfies $t_{\alpha,\beta}n_\beta \geq m_\alpha$, where $t_{\alpha,\beta}$ denotes the multiplicity of $E_j = F_\beta$ in the pullback of $D_\alpha$. In particular, $n_\beta = \infty$ if $m_\alpha = \infty$.

Since $y \in (\calY,\calF)^{\cc}(\O_T)$, we have $n_v(\calF_\beta,y) \in \Z_{\geq n_\beta} \cup \{0,\infty\}$ if $n_\beta$ is finite and zero otherwise. However, since $\overline{f(y)}_v \in (\calZ_j)_v$, it follows from compatibility of blowups and base change/reduction that $\overline{y}_v \in \calE_j = \calF_\beta$, i.e.\ $n_v(\calF_\beta,y) > 0$. We get a contradiction when $m_\alpha = \infty$ and otherwise get $n_v(\calF_\beta,y) \geq n_\beta \geq m_\alpha/t_{\alpha,\beta}$. It now suffices to note that $n_v(\calD_\alpha,x) \geq t_{\alpha,\beta}n_v(\calF_\beta,y)$ by Proposition \hyperref[prop:compat]{\ref*{prop:compat}\emph{\ref*{prop:compatii}}}.
\end{proof}

Using similar ideas, we recover the following result, which was used in \cite{NS} to study Campana points on del Pezzo surfaces.

\begin{lemma}
Given an orbifold $(X,D)$ with $\O_v$-model $(\calX,\calD)$ and $z\in X(K_v)$, set
\[
Y:=\Bl_{z} X \xrightarrow{\rho} X,\quad \calY:=\Bl_{\overline{z}} \calX \xrightarrow{\overline{\rho}} \calX, \quad F:=\rho^*D,\quad \calF:=\overline{\rho}^{*}\calD,
\]
and let $\calE$ be the exceptional divisor of $\overline{\rho}$.
Then $(\calY,\calF)$ is an $\O_v$-model for $(Y,F)$. Let $y\in Y(K)$ and $x=\rho(y)$. Assume  that $\overline{x}_v \neq \overline{z}_v$. Then $y$ is a Campana $\O_v$-point if and only if $x$ is a Campana $\O_v$-point.
\end{lemma}

\begin{proof}
That $\left(\calY,\calF\right)$ is an $\O_v$-model for $\left(Y,F\right)$ follows from \cite[Prop.~8.1.12]{LIU}, which gives $\calY_{\left(0\right)} \cong Y$ and integrality of $\calY$, and \cite[Prop.~4.3.9]{LIU}, which gives flatness of $\calY$ over $\Spec \O_v$.

Note that $x$ is a local Campana point not meeting $\overline{z}$, that $\rho$ is minimal in the above sense and that $\rho$ satisfies condition (iii) in the above proposition. Then the result follows from local versions of the proofs of (i) and (iii) in the previous proposition.  
\end{proof}

\section{Campana points and quadrics} \label{section:CPQ}
In this section we prove Theorems~\hyperref[thm:main1]{\ref*{thm:main1}\emph{\ref*{m1i}}} and \ref{thm:main2}.

\begin{proof}[Proof of {Theorem~\hyperref[thm:main1]{\ref*{thm:main1}\emph{\ref*{m1i}}}}]
First, we show that $(\P^{n+1}_{\O_S},\calQ)^{\cc}(\A_K) \neq \emptyset$. We will show the stronger result that $(\P^{n+1}_{\O_S}\setminus\calQ)(\calO_v) \neq \emptyset$ for all $v \in \Omega_K$. For $v \in S$ this is clear, since $(\P^{n+1}_{\O_S}\setminus\calQ)(\calO_v) = (\P^{n+1}\setminus Q)(K_v)$ which must be non-empty. Now let $v \not\in S$. It suffices by Hensel's lemma to show that $\calQ_v(\F_v) \neq \P^{n+1}(\F_v)$ which would imply existence of an $\O_v$-point on $\P^{n+1}_{\O_v}\setminus\calQ_{\O_v}$. In turn, it suffices to show that $\calQ_v$ is a (possibly reducible) quadric hypersurface in $\P^{n+1}_{\F_v}$. Indeed, by \cite{SER}, the maximum number of $\F_v$-points on a quadric hypersurface in $\P^{n+1}_{\F_v}$ is $ 2q_v^n + \#\P^{n-1}(\F_v) < \#\P^{n+1}(\F_v)$. Choosing a defining quadratic polynomial $f$ for $Q$ and viewing $f$ over $K_v$, we may multiply every coefficient by a suitable power of a $v$-adic uniformiser to produce a polynomial $\widetilde{f}$ such that the minimum $v$-adic valuation of the coefficients is zero and the zero loci of $f$ and $\widetilde{f}$ coincide in $\mathbb{P}^{n+1}_{K_v}$. It is easily seen that the zero locus of $\widetilde{f}$ in $\P^{n+1}_{\O_v}$ is precisely the closure of $Q_{K_v}$ in $\P^{n+1}_{\O_v}$, and that this closure coincides with $\calQ_{\O_v}$, but by composition of fibre products, the special fibre of $\calQ_{\O_v}$ is $\calQ_v$, and the special fibre of the zero locus of $\widetilde{f}$ is the zero locus of the reduction of $\widetilde{f}$ modulo $v$, which is a degree-$2$ polynomial defined over $\F_v$.

Let $(x_v)\in (\mathbb{P}_{\O_S}^{n+1},\calQ)^{\cc}(\A_K)$ be an adelic Campana point. Let $T\subset \Omega_K$ be a finite set of places at which we would like to approximate $(x_v)$ by a Campana point. Using weak approximation on $\mathbb{P}^{n+1}$, we may choose $y \in \mathbb{P}^{n+1}(K)$ which is $v$-adically close to $x_v$ for all $v \in T$. Let $L \subset \mathbb{P}^{n+1}$ be a $K$-rational line containing $y$ which intersects $Q$ in two distinct (not necessarily) points, and let $\calL_S$ be the closure of $L$ in $\P^{n+1}_{\O_S}$. The Campana orbifold $(\calL_S,\(1-1/m\)\calL_S\cap \calQ_S)$ satisfies CWA by \cite[Prop.~3.5,~Prop.~3.15]{NS} and has at least one adelic Campana point by the argument in the first paragraph applied to $L \cong \mathbb{P}^1$. For each $v\in T\setminus S$, observe that we have the following equality of ideals of $\calO_v$,
\[
(\overline{y}_{\O_v})^{*}(\calL_S\cap \calQ_S)=(\overline{y}_{\O_v})^{*}\calQ_S=(\overline{x_v})^{*}\calQ_S.
\]
Hence, $y \in (\calL_S, \(1-1/m\)(\calL_S\cap \calQ_S))(\calO_v)$ if $v\in T\setminus S$. It follows that there exists $z \in (\calL_S,\(1-1/m\)\calL_S\cap \calQ_S)(\calO_S)$ which approximates $y$ at each $v \in T$. Then $(\overline{z}_{\O_v})^{*}(\calL_S\cap \calC_S)=(\overline{z}_{\O_v})^{*}\calQ$ for every $v\notin S$, so that $z\in(\mathbb{P}_{\O_S}^{n+1},\(1-1/m\)\calQ)^{\cc}(\O_S)$ as well. Hence we have found a Campana point $z$ approximating our original $(x_v)$.
\end{proof}

Before proving Theorem \ref{thm:main2}, we recall some basic facts on quadric hypersurfaces. Let $Q \subset \P^{n+1}$ for $n\geq2$ be a smooth quadric hypersurface over a field $k$ such that $Q(k) \neq \emptyset$. Projection from $z \in Q(k)$ to a hyperplane not containing $z$ (which we identify with $\P^n$) induces a birational map $\phi\colon Q\dashrightarrow \P^n$. Denoting by $T_zQ$ the tangent hyperplane to $Q$ at $z$, it is easily seen that $\phi$ is defined away from $z$ and an isomorphism away from $T_z Q \cap Q$. Further, $\phi$ is resolved by blowing up $z$ to obtain a smooth projective variety $X = \Bl_zQ$ and a birational morphism $f: \Bl_z Q \rightarrow \mathbb{P}^n$, which contracts the lines on the strict transform of $T_zQ\cap Q$. Defining $B$ to be the smooth quadric $T_zQ \cap Q \cap \P^n \subset \P^n$, it follows that $X \to \P^n$ is the blow up of $B$.

Now let $H \subset \P^{n+1}$ be a hyperplane, and take $m \in \Z_{\geq 2}$. Set $\Delta_m := \(1-1/m\)Q \cap H$, and assume without loss of generality that $z \not\in Q \cap H$ in the above. We claim that $f$ is an isomorphism on $Q \cap H$. It suffices to show that $f|_{Q \cap H}$ is injective. Two points in $Q$ map to the same point of $\P^n$ under $f$ if and only if they lie on a line in $T_z Q \cap Q$, and all lines in the quadric cone $T_zQ \cap Q \subset T_zQ$ pass through $z$ \cite[p.~283]{HAR}, so injectivity follows since $z \not\in H$. Then $f(Q \cap H)$ is a quadric in $\mathbb{P}^n$, which we denote by $C$.

Let $m \in \Z_{\geq 2}$. Writing $C_m = \(1-1/m\)C$, we observe that $f: (Q,\Delta_m) \rightarrow (\P^n,C_m)$ is an orbifold birational map. Denote by $E_z,E_B \subset X$ the exceptional divisors contracted by the morphisms $X \rightarrow Q$ and $X \rightarrow \P^n$ respectively.

\begin{proof}[Proof of Theorem \ref{thm:main2}]
For $n=1$, we have $Q\cong \P^1$, so the result follows from \cite[Prop.~3.5,~Prop.~3.15]{NS}. For the rest of the proof, let $n\geq2$. If $(\calQ,\DDelta_m)^{\cc}(\A_K) = \emptyset$, then Campana weak approximation holds trivially, so assume for the rest of the proof that $(\calQ,\DDelta_m)^{\cc}(\A_K) \neq \emptyset$. In particular, we have $\prod_{v \in \Omega_K} Q(K_v) \neq \emptyset$, hence $Q(K) \neq \emptyset$ by the Hasse--Minkowski theorem. Further, projection from a rational point furnishes a birational map to $\mathbb{P}^n$, so $Q(K)$ is Zariski dense. On the local side, we see that $(\calQ,\DDelta_m)^{\cc}(\O_v)$ is infinite for all $v \in \Omega_K$. Indeed, projection from an initial local Campana point gives a birational map to $\mathbb{P}^n_{K_v}$, hence $Q(K_v)$ is dense in the $v$-adic topology, and there are infinitely many points in the non-empty $v$-adic open set $(\calQ,\DDelta_m)^{\cc}(\O_v)$.

Let $T \subset \Omega_K$ be finite, and let $x_v \in (\calQ,\DDelta_m)^{\cc}(\O_v)$, $v \in T$, be local Campana points which we would like to approximate by a Campana point. Since $Q(K)$ is Zariski dense, we may choose $z\in Q(K)$ not on $H$ such that $z\neq x_v$ for all $v\in T$. Denote by $\phi: Q \dashrightarrow \mathbb{P}^n$ the induced birational map to $\mathbb{P}^n$. Henceforth, we will expand $T$ whenever necessary by choosing further arbitrary points $x_v\in (\calQ,\DDelta_m)^{\cc}(\O_v)$. Note that $(\calQ,\DDelta_m)^{\cc}(\O_v)$ is infinite for all $v$, since $(\calQ,\DDelta_m)^{\cc}(\O_v)$ is a non-empty $v$-adic open subset of $Q(K_v)$, so we may choose these additional $x_v$ without violating our assumption that $z\neq x_v$ for all $v\in T$.

Denote by $T_zQ$ the tangent space of $Q$ at $z$. Let $\A^n\subset \P^n$ be the complement of the hyperplane $T_zQ \cap \mathbb{P}^n$, and let $U\subset Q$ be the complement of the singular quadric $T_zQ\cap Q$. Then $\phi$ restricts to an isomorphism $f\colon U\to \mathbb{A}^n$. Enlarging $T$ if necessary, we can assume that $f$ spreads out to an isomorphism $\widetilde{f}\colon \calU_T  \to \A^n_{\O_T}$.

Let $\calH$ be the Zariski closure of $H$ in $\calQ$. Enlarging $T$ if necessary, the closure $\overline{z}\subset \calQ_T$ does not meet $\calH_T$. As observed in \cite[\S22]{HAR}, we may identify the graph of $\phi$ with the blowup of $\mathbb{P}^n$ along the smooth quadric $B = T_zQ \cap Q \cap \mathbb{P}^n$ in the hyperplane $T_z Q \cap \mathbb{P}^n \subset \mathbb{P}^n$. Choose $w_1\in (T_zQ \cap \mathbb{P}^n)(K)$ not lying in $B$. Enlarge $T$ so that $\overline{w_1}$ does not meet $\calB_T$.

Recall that $X=\Bl_z Q=\Bl_B\P^n$ and $C = f\left(Q \cap H\right), C_m = \(1-1/m\)C$. Let $(X,\widetilde{C}_m)$ be the orbifold structure induced from $X\to \P^n$ (Definition \ref{def:induced}). We have the picture:

\[
\begin{tikzcd}[column sep=small]
& (X,\widetilde{C}_m)\arrow[dl]\arrow[rd] \\
(Q,\Delta_m) \arrow[rr,dotted,"f"] && (\P^n,C_m).
\end{tikzcd}
\]

Observe that $(X,\widetilde{C}_m)\to (\P^n,C_m)$ and $(X,\widetilde{C}_m)\to (Q,\Delta_m)$ are both EBOMs. Indeed, the induced orbifold structure associated to a birational morphism makes it an EBOM, while the map down to $Q$ contracts the exceptional divisor to a point not on $\Delta$, and the underlying divisor of $\widetilde{C}_m$ is the strict transform of both $C$ and $\Delta=Q\cap H$ under their respective maps. Since $(\calX, \widetilde{\calC}_m)\to (\calQ,\DDelta_m)$ is an EBOM, by Proposition \hyperref[prop:ebom]{\ref*{prop:ebom}\emph{\ref*{ebomii}}}, we can and will enlarge the set of places $T$ so that Campana points map to Campana points. Then we have a map
\begin{equation}\label{eqn:blowdown}
(\calX, \widetilde{\calC}_m)^{\cc}(\O_T)\to (\calQ,\DDelta_m)^{\cc}(\O_T).
\end{equation}

Using strong approximation on $\A^n_{\O_T}$, choose an $\O_T$-point $w_2$ of $\A^n_{\O_T}$ approximating each $f(x_v)$ for $v\in T$. Consequently, $\overline{w_1}$ and $\overline{w_2}$ do not meet over $\calO_T$. Let $L$ be the line that connects $w_1$ and $w_2$. Let $\calL_T$ be its closure in $\mathbb{P}^n_{\O_T}$.

By Theorem \ref{thm:main1}, $(\calL_T,\calC_T\cap \calL_T)$ satisfies CWA and has at least one Campana point, so we can choose $y\in (\calL_T,\calC_T\cap \calL_T)^{\cc}(\O_T)$ $v$-adically approximating $w_2$ for all $v\in T$.

If $\overline{y}$ meets $\calB_T$, then $\overline{y}_v\in f(E)$ for some $v\notin T$. Since $\calL\cap f(E)=\overline{w_1}$, this implies $\overline{y}_v = (\overline{w_1})_v$, contradicting that $\overline{w_1}$ does not meet $\calB_T$. Then $\overline{y}$ does not meet $\calB_T$ and is by Proposition \hyperref[prop:ebom]{\ref*{prop:ebom}\emph{\ref*{ebomi}}} in the image of $(\calX, \widetilde{\calC}_m)^{\cc}(\O_T)\to (\P_{\O_T}^n,\calC_m)^{\cc}(\O_T)$; let $y'$ be its preimage.

The image $f^{-1}(y)$ of $y'$ under \eqref{eqn:blowdown} is a Campana point; since $f^{-1}(y)$ approximates $x_v$ for each $v\in T\cap S$, it is Campana for $v\in T\cap S$ as well, so $f^{-1}(y)\in (\calQ,\calH)^{\cc}(\O_S)$.
\end{proof}

\section{Darmon points for a quadric divisor} \label{section:D1}

In this section we prove Theorem \hyperref[thm:main1]{\ref*{thm:main1}\emph{\ref*{m1ii}}}. A key observation regarding the abundance and distribution of Darmon points is that DWWA implies the Darmon Hilbert property. That is, if the Darmon points of an orbifold satisfy weak approximation away from some finite set of places of the ground field, then they must form a non-thin set of rational points. The Campana analogue of this result \cite[Thm.~1.1]{NS} was used by the second and third authors to prove the non-thinness of Campana points for many ``linear'' orbifold structures on projective space. The Darmon version follows in an entirely analogous way: the key ingredient of the proof, that the local Campana points form an open subset in the local analytic topology, remains true for Darmon points.

Here, we shall use the contrapositive of this result to prove that orbifolds fail DWWA. This is perhaps not surprising: the conditions imposed on Darmon points, requiring certain values to be a perfect $m$th power up to units, are of an inherently thin nature. Indeed, thin subsets of rational points are those which in some sense come from solubility of non-linear equations over the ground field, which is rare in fields of arithmetic interest. The following is a restatement of Theorem \hyperref[thm:main1]{\ref*{thm:main1}\emph{\ref*{m1ii}}}. 

\begin{proposition}\label{prop:Darmonthin}
Let $Q\subset \P^{n+1}$ be a quadric defined by a quadratic form $f(X_0,\dots,X_{n+1})$ over a number field $K$, and let $m\in\Z_{\ge2}$. Let $Q_m$ be the $\QQ$-divisor $\(1-1/m\)Q$ of $\PP^{n+1}$. Let $S \subset \Omega_K$ be a finite subset of places containing $\Omega_K^\infty$, and take the $\O_S$-model $(\PP^{n+1}_{\O_S},\calQ_m)$ of $(\PP^{n+1},Q_m)$. If $m$ is even, then the set of Darmon points $(\PP^{n+1}_{\O_S},\calQ_m)^{\dd}(\O_S)$ is thin.
\end{proposition}

\begin{proof}
Assume first that $m$ is even, so that $m = 2m'$ for some $m' \in \mathbb{Z}_{>0}$. Since $(\PP^{n+1}_{\O_S},\calQ_{2m'})^{\dd}(\O_S) \subset (\PP^{n+1}_{\O_S},\calQ_2)^{\dd}(\O_S)$, it suffices to consider the case $m=2$. We will construct an explicit collection of covers of $\PP^{n+1}$ such that the image of rational points from the covering varieties contains the Darmon points.

Note that, for $P=[a_0:\ldots:a_{n+1}]\in \P^{n+1}(K)$ to be a Darmon point, we must have $f(a_0,\ldots,a_{n+1}) = ut^2$ for some $t \in \O_S$ and unit $u\in\O_S^*$. It follows from Dirichlet's unit theorem for $S$-units that the set $\O_S^*/\O_S^{*2}$ is finite. Choose a set of representatives $\{u_1,\dots,u_r\}$ for this finite quotient.

Define for each $i \in \{1,\dots,r\}$ the smooth quadric hypersurface
\[
Y_i: f\left(X_0,\ldots,X_{n+1}\right) = u_iX_{n+2}^2\subset \P^{n+2}.
\]
The maps $\pi_i: Y_i \rightarrow \PP^{n+1}$, $\pi_i([X_0:\dots:X_{n+1}:X_{n+2}]) = [X_0:\dots:X_{n+1}]$ are finite of degree $2$ and that $(\PP^{n+1}_{\O_S},\calQ_m)^{\dd}(\O_S) = \bigcup_{i=1}^rf_i(Y_i(K))$, hence $(\PP^{n+1}_{\O_S},\calQ_m)^{\dd}(\O_S)$ is thin.
\end{proof}

It is unclear whether $(\PP^{n+1}_{\O_S},\calQ_m)^{\dd}(\O_S)$ is thin when $m$ is odd. However, we give an example demonstrating that DWWA often fails in the case $n=0$.

\begin{example}\label{ex:DWWAfail}
Let $m\ge 3$ be odd. Let $Q$ be given by a binary quadratic form $f(X,Y)$ with coprime integer coefficients. Let the discriminant of $f$ be $d$ and set $L:=\Q(\sqrt{d})$. Then $f(X,Y)$ factors as $c\ell(X,Y)\ell'(X,Y)$ over $L$ for some $c \in \Q$, for $\ell$ and $\ell'$ conjugate linear forms. We will show that, choosing $m$ coprime to $d$, DWWA fails for $\left(\mathbb{P}^1_{\Z},\calQ_m\right)$.

By Dirichlet's unit theorem, $\O_L^*/\O_L^{* m}$ is finite. Choose a prime $p$ that is inert in $L/\Q$ such that the coefficients of $f,\ell$ and $\ell'$ are $p$-adic units. Then the residue field of $\O_L$ at $p$ is $\F_{p^2}$. Since $m$ and $d$ are coprime, it follows from the Chinese remainder theorem that we may assume that $m \mid (p-1)$. We do so; then $\F_p^*/\F_p^{* m}$ is non-trivial. The injection
\[
\frac{\F_p^*}{\F_p^{* m}}\hookrightarrow \frac{\F_{p^2}^*}{\F_{p^2}^{* m}}
\]
has image of index $p+1$. Consider the reduction map $\pi\colon \O_L^* \to \F_{p^2}^*/\F_{p^2}^{* m}$. The image of $\pi$ is bounded independently of $p$ as it factors through $\O_L^*/\O_L^{*m}$. Hence, we assume $p$ is large enough so that $\pi(\O_L^*)\F_p^*$ does not generate all of $\F_{p^2}^*/\F_{p^2}^{* m}$. Choose $\overline{x},\overline{y}\in \F_p^*$ so that $\ell(\overline{x},\overline{y})\in \F_{p^2}^*$ lies in the complement. Lift this to $x',y'\in \Q_p^*$. Then $[x':y']$ is a $p$-adic Darmon point, since the $p$-adic valuation of $f(x',y')=c\ell(x',y')\ell'(x',y')$ is $0$. We claim this point cannot be approximated by a global Darmon point. Let $[x:y]\in\P^1(\Q)$ be a global Darmon point where $x,y\in\Z$ and $\min(v_p(x),v_p(y))=0$. The Darmon condition at $p$ implies $f(x,y)=cut^m$ where $u,t\in\Z$ and $v_p(u)=0$. Rescaling $[x:y]=[xu^{(m-1)/2}:yu^{(m-1)/2}]$, we may assume that $f(x,y)=ct^m$, hence $\ell(x,y)\ell'(x,y)=t^m$. Since $m$ is odd, $t$ is then also a norm from $L/\Q$, i.e.\ $t=N(a)$ for some $a\in L^*$. Since $p$ is inert, $v_p(t)=2v_p(a)$, so $v_p(a)\geq0$. Hence, $N_{L/\Q}(a)^m=\ell(x,y)\ell'(x,y)$. Moreover,
\[
a^mu=\ell(x,y)
\]
for some $u\in L^*$ of norm $1$. Note that $v_p(u)=0$ as $p$ is inert. Thus we may consider $a^mu$ as an element in $\O_{L,p}$. If $[x:y]$ approximates $[x':y']$, then up to scaling by some $t\in \F_p^{*}$,
\[
a^mu\bmod v= t\ell(\overline{x},\overline{y}).
\]
But this would mean that $\ell(\overline{x},\overline{y})\in \pi(\O_L^*)\F_p^*\bmod \F_{p^2}^{* m}$, which is a contradiction.
\end{example}

\section{Darmon points on quadric hypersurfaces} \label{section:D2}

In this section we shall prove Theorem~\ref{thm:main3}. We set the following conventions: we reserve $(Q,\Delta_m)$ for a Campana orbifold as in the statement of Theorem~\ref{thm:main2} with $\O_S$-model $(\sQ,\DDelta_m)$ for some finite set $S \supset \Omega_K^\infty$. We also set $U := Q \setminus \Delta$ and $\sU = \sQ \setminus \DDelta$. As we shall soon see, $\Br U / \Br K$ is either trivial or cyclic of order 2. If $\Br U / \Br K \cong \ZZ/2\ZZ$ define $\Omega_{K,\sA}$ as the set of places of $K$ where the local invariant map of the generator $\sA \in \Br U$ of $\Br U / \Br K$ is constant and let $T \subset \Omega_{K,\sA}$ be a finite subset.

Our proof is based on the existence of a large place $v$, ``zooming in'' at which allows us to deduce both the lack of a Brauer--Manin obstruction to the existence of strict Darmon points and the presence of an obstruction to Darmon strong approximation off $T$.

\begin{proof}[Proof of Theorem~\ref{thm:main3}]
  Recall that strict Darmon points of $(\sQ, \DDelta_m)$ lie on $U$, while non-strict Darmon points come from $\Delta$. Since $m$ is finite and $\Delta$ is irreducible, it follows that the set of non-strict Darmon points is the set of $K$-rational points on $\Delta$. The assumption on $\Delta$ now implies that it is an $(n - 1)$-dimensional smooth projective quadric, so it satisfies the Hasse principle and weak approximation for $K$-rational points. In view of this and the fact that the adelic and the product topologies coincide for adelic spaces of projective varieties, given a non-empty non-strict adelic set we can deduce that all semi-integral local-global principles are satisfied on the non-strict part. To investigate Brauer--Manin obstructions we can thus focus on strict Darmon points. Assume that $(\sQ, \DDelta_m)_{\str}^{\dd}(\A_K) \neq \emptyset$. It follows that $Q(K) \neq \emptyset$ by the Hasse--Minkowski theorem. Moreover, not all rational points on $Q$ are contained in $\Delta$, so $U(K) \neq \emptyset$. Therefore $\Br K$ injects into $\Br U$.   

  Assume first that $n = 2$. Let $q(x,y,z,t)$ be a quadratic form defining $Q \subset \PP^3$ and let $\ell$ be a linear form defining $H$. A similar argument to the one used in \cite[Lem.~2.1]{VAV} or \cite[\S5.8]{CTX} shows that there is $c \in K \setminus \{0\}$ and linear forms $\ell_1, \ell_2, \ell_3, \ell_4$ such that $q = \ell_1\ell_2 - c(\ell_3^2 - d\ell_4^2)$ for $d$ the discriminant of $q$. Indeed, since $Q(K) \neq \emptyset$, the quadratic space associated to $Q$ contains a hyperbolic plane. Taking a diagonal basis of the orthogonal component to this hyperbolic plane gives the desired representation of $q$, so we have
  \begin{equation}
  	\label{eqn:Xl_i}
  	Q: \ \ell_1 \ell_2 = c\(\ell_3^2- d\ell_4^2\).
  \end{equation}

  To proceed with proving Theorem~\ref{thm:main3}, we must understand the structure of $\Br U$, which is done in the next lemma.

  \begin{lemma}
  	\label{lem:quadricsBGp}
  	We have 
  	\[
  		\Br U/\Br K =
  		\begin{cases}
  			0 &\text{if } d \in K^{*2}, \\
  			\ZZ/2\ZZ &\text{otherwise}.
  		\end{cases}
  	\]
    Moreover, when non-trivial, this group is generated by the image of $\sA = (\ell_1/\ell, d) \in \Br U$.
  \end{lemma}

  \begin{proof}
  	It suffices to check that, after a linear change of variables, $q(x, y, z, t)$ takes the shape $q'(X, Y, Z) - nT^2$, where $\Delta$ is defined by $T = 0$. This puts us in a position to argue as in \cite[Prop.~7.3.2]{CTS2}, which verifies the first claim in the statement.

  	To see that $q$ takes this shape after a linear change of variables, note that $\ell = 0$, which defines the irreducible divisor $\Delta$, defines a $3$-dimensional subspace of the quadratic space associated to $q$. It is thus associated to a quadratic form $q'(X, Y, Z)$ in three variables. Its orthogonal component is of dimension one, hence under a suitable change of variables, $q(x, y, z, t)$ becomes $q'(X, Y, Z) - nT^2$ for some $n \in K$, where $T = 0$ defines $\Delta$ in $Q$.

    Assume from now on that $d \in K^* \setminus K^{*2}$ and therefore $\Br U / \Br K \simeq \ZZ/2\ZZ$. Following the same line of arguing as in \cite[\S5.8]{CTX} or \cite[Prop.~7.3.2]{CTS2} one verifies that the image of $\sA = (\ell_1/\ell, d) \in \Br U$ generates $\Br U / \Br K$. Indeed, \eqref{eqn:Xl_i} shows that $(\ell_1/\ell, d) = (c\ell_2/\ell, d)$ is unramified on $Q$ away from the hyperplane section with $H$ and the finitely many points given by $\ell_1 = \ell_2 = 0$ which form a codimension-$2$ subvariety. Thus $\sA$ lies in $\Br U$ by the purity theorem. Moreover, \eqref{eqn:Xl_i} implies that the residue of $\sA$ along $\Delta$ equals the class of $d$ in $K^*/K^{*2}$ which by assumption is non-trivial. Therefore $\sA$ does not lie in $\Br K \subset \Br U$, hence its image in $\Br U / \Br K$ generates that group.
  \end{proof}

  Going back to the case of general $n \ge 2$ we will employ Proposition~\ref{prop:BMOequiv} and the status of local-global principles for rational points on $U$ to deduce the claims of Theorem~\ref{thm:main3}. By a similar argument to the one appearing in the proof of Lemma~\ref{lem:quadricsBGp}, it suffices to consider the case where $U$ is given by $q(\bfx) = n$ with $n \in K$ and $q$ a quadratic form in $n+1$ variables. Then $\Delta(K_{v_0}) \neq \emptyset$ ensures that $q$ is isotropic over $K_{v_0}$.

$\Br U/ \Br K$ is trivial if $n \ge 3$ by \cite[Thm.~7.3.1]{CTS2} or if $n = 2$ and $d \in K^{*2}$ by Lemma~\ref{lem:quadricsBGp}. Thus $(\sQ, \DDelta_m)_{\str}^*(\AA_{K, S \cup \{v_0\}}^{\{v_0\}})^{\Br} = (\sQ, \DDelta_m)_{\str}^*(\AA_{K, S}^{\{v_0\}})$ and (ii) is implied by (i) in this case.

  The claim in (i) follows from the combination of Proposition~\ref{prop:BMOequiv} with the existing results in the literature about strong approximation off $v_0$ for $K$-rational points on $U$. For $n \ge 3$ these are established by Eichler and Kneser (\cite[Prop.~3.2]{CTX13}) and for $n = 2$ strong approximation with Brauer--Manin obstruction off $v_0$ follows from \cite[Prop.~4.5]{CTX13}.

  \begin{note}
    Alternatively, in view of Proposition~\ref{prop:BMOequiv}, an interplay between integral models shows that results about strong approximation off $v_0$ for $(S \cup \{v_0\})$-integral points on $\sU$ may be used instead. These are Colliot-Th\'el\`ene and Xu's \cite[Thm.~6.3]{CTX} if $n = 2$ and Kneser's \cite[Thm.~6.1]{CTX} if $n \ge 3$. In fact, one will need slightly generalised versions of the cited theorems for $(S \cup \{v_0\})$-integral points (instead of $\sO_{\{v_0\}}$-points) which follow from \cite[Thm.~3.7]{CTX} by taking $S_0 = \{v_0\}$ in its proof.
  \end{note}

  Assume now that $n = 2$ and $\Br U / \Br K \cong \ZZ/2\ZZ$. The statement of (iii) follows from Theorem~\hyperref[thm:main0]{\ref*{thm:main0}\emph{\ref*{m0iii}--\ref*{m0iv}}} as $\Br Q = \Br K$ and $\Br (Q, \Delta_m) = 0$ for odd $m$ subject to the non-prolific condition on $\sA$ at $T$. If $d \in K_v^{*2}$, then that condition is certainly met at $v$ since the local invariant map of $\sA$ vanishes there. Take $\Omega_{K, \sA}$ to be the set of places $v$ for which $d \in K_v^{*2}$ and those $v$, where one of $d$ or $1/d$ is not a unit of $\sO_v$ but $\inv_v \sA$ is constant at $v$. This shows that for any $T \subset \Omega_{K, \sA}$ the Brauer element $\sA$ is not prolific at $T$. Moreover, $\Omega_{K, \sA} \subset \Omega_K$ clearly has a density $1/2$ as $d, 1/d \in \sO_v^*$ for all but finitely many $v$. At the same time $d$ is a square in $K_v$ for half of the places of $\Omega_K$ where it is a unit of $\sO_v$ while it is a non-square for the other half.

  There is also a direct way to prove (iii), giving an explicit version of Theorem~\ref{thm:HarariA} in this case and highlighting that arbitrarily big places of $K$ play a role in the Brauer--Manin obstruction, which we include below.

  To prove Theorem~\ref{thm:main3} (iii), it suffices to show that there is$v \in \Omega_K$ for which the local invariant map of $\sA$ takes all possible values when evaluated at the $v$-adic strict $*$-points. Indeed, let $P_1, P_2 \in (\sQ, \DDelta_m)_{\str}^{*}(\sO_v)$ such that $\inv_v \sA(P_1) = 0$ and $\inv_v \sA (P_2) = 1/2$. Let $(M_w)_w \in (\sQ, \DDelta_m)_{\str}^{*}(\A_K)$. If $\sum_{w \neq v} \inv_w \sA(M_w) = 0$, then replace $M_v$ by $P_1$ to obtain a point in $(\sQ, \DDelta_m)_{\str}^{*}(\AA_K)^{\Br}$. If instead $\sum_{w \neq v} \inv_w \sA(M_w) = 1/2$, then replace $M_v$ by $P_2$ to get the same conclusion. Thus there is no Brauer--Manin obstruction to the existence of strict $(S \cup \{v_0\})-*$-points on $(\calQ, \DDelta_m)$. 

  Let $\Omega_{K, \sA}$ be the set of places $v \in \Omega_K$ for which $\inv_v \sA$ is constant. Making the same substitutions as above but swapping the roles of $P_1$ and $P_2$ shows that there is a $*$-adelic point outside $(\sQ, \DDelta)_{\str}^{*}(\AA_K)^{\Br}$, that is $(\sQ, \DDelta)_{\str}^{*}(\AA_K)^{\Br} \subsetneq (\sQ, \DDelta)_{\str}^{*}(\AA_K)$. By Proposition~\hyperref[prop:B-Mset]{\ref*{prop:B-Mset}\emph{\ref*{BMi}}} the closure of $(\sQ, \DDelta)_{\str}^{*}(\sO_{S \cup T})$ is contained in $(\sQ, \DDelta)_{\str}^{*}(\AA_K)^{\Br}$, thus there is a Brauer--Manin obstruction to $T$-$*$-strong approximation for any choice of $T \subset \Omega_{K, \sA}$. Indeed, the local invariant map of $\sA$ is constant at each place of $T$, so the value of the sum of local invariant maps does not depend on the choice of components at places in $T$ of an adelic point, thus $(\sQ, \DDelta)_{\str}^{*}(\AA_K^T)^{\Br} \subsetneq (\sQ, \DDelta)_{\str}^{*}(\AA_K^T)$.

  It remains to find a place of $K$ satisfying our claim. Let $v$ be a non-archimedean place of $K$ of inertia degree one, let $K_v$ be the associated completion of $K$ with ring of integers $\O_v$, a uniformiser $\pi_v$ and let $\Fv = \O_v/(\pi_v)$ be the residue field of $v$. Assume that $d \in \O_v$ is a unit and that it is not a square in $K_v^*$.

  Since the forms $\ell_i$ appearing in \eqref{eqn:Xl_i} are $K$-linearly independent, there are constants $c_i \in K$ such that $\ell = c_1\ell_1 + \dots + c_4\ell_4$. We can take $v$ so that $c$ from \eqref{eqn:Xl_i}, all of the non-zero $c_i$ and the non-zero coefficients in front of $x, y, z, t$ in $\ell_i$ are all units in $\O_v$. Let $A$ be the matrix formed from the coefficients in front of $x, y, z, t$ in $\ell_i$ where the coefficients of $\ell_1$ form the first row, the coefficients of $\ell_2$ form the second row and so on. We shall additionally require the discriminant of $A$ to be a unit in $\O_v$. Note that there are infinitely many choices of $v$ satisfying the assumption in the previous paragraph, while there are only finitely many $v$ failing the ones made here. Thus not only a place $v$ satisfying our assumptions exists but such places have density $1/2$ in $\Omega_K$ and we can take $v$ as large as we wish. In particular, this and the shape of $\sA$ also imply that $\Omega_{K, \sA}$ differs by finitely many places form the set of $v \in \Omega_K$ such that $d$ is a quadratic residue mod $v$.

  What is left is to show the existence of strict $*$-$\sO_v$-points on $(\sQ, \DDelta_m)$ for which $\inv_v \sA$ takes two values. This is a direct consequence of Theorem~\ref{thm:HarariA} together with the fact that for all but finitely many places $v$, there exists $P_v\in\sU(\sO_v)$ such that $\inv_v\calA(P_v)=0$. However, this claim can also be verified fairly easily on the level of Darmon points with an explicit argument which we include below.

  \begin{lemma}
    \label{lem:inv-surjectivity}
    There exist $P_1, P_2 \in (\sQ, \DDelta_m)_{\str}^{\dd}(\O_v)$ with $\inv_v \sA(P_1) = 0$, $\inv_v \sA(P_2) = 1/2$.  
  \end{lemma}

  \begin{proof}
    Let $k \in \ZZ_{>0}$. We claim that there is a choice of $(x,y,z,t) \in \O_v^4$ such that
    \begin{equation}
    	\label{eqn:systeml_i}
    	\begin{split}
    		c_1\ell_1 + \dots + c_4\ell_4 &= \pi^{km}, \\
    		\ell_1 &= 1, \\
    		\ell_1\ell_2 &= c(\ell_3^2 - d\ell_4^2).
    	\end{split}
    \end{equation}
    The assumption on $\Delta$ being geometrically irreducible implies that it is not possible that $\ell$ is a scalar multiple of either $\ell_1$ or $\ell_2$. Indeed, otherwise $\Delta$ over $\bar{K}$ would have been given by $\ell_3^2 - d\ell_4^2$ which is a union of two conjugate lines.

    To see the claim, we now look at the reduction of the above system over $\Fv$. Treating the $\ell_i$'s as variables in \eqref{eqn:systeml_i}, we first show that there exists a solution with $\ell_i\in\O_v$. If $c_1 \neq 0$ we can argue for example by showing that $(\ell_3, \ell_4) \nequiv (0, 0) \bmod v$. Indeed, if $(\ell_3, \ell_4) \equiv (0, 0) \bmod v$, the last equation of \eqref{eqn:systeml_i} would imply that $\ell_2 \equiv 0 \bmod v$. Since $\ell_1 \equiv 1 \bmod v$, the first equation of the system would now imply that $c_1$ vanishes mod $v$, which contradicts the choice of $v$. On the other hand, $c_1 \nequiv 0 \bmod v$ and $(\ell_3, \ell_4) \nequiv (0, 0) \bmod v$ imply that there is a minor of the Jacobian of \eqref{eqn:systeml_i} of rank 3. Hensel's lemma is thus applicable, hence the existence of an $\sO_v$-solution to \eqref{eqn:systeml_i} is established. 

    Alternatively, if $c_1 = 0$, then since $\ell$ is not a scalar multiple of either $\ell_1$ or $\ell_2$, our assumption on $v$ shows that at least one of $c_3$ or $c_4$ does not vanish mod $v$. Taking $(\ell_2, \ell_3, \ell_4) = (0, 0, 0) \bmod v$ gives a solution of \eqref{eqn:systeml_i} for which the Jacobian of the system has a minor of rank 3. 

    Let $\ell_i = a_i$ for $i = 1, \dots, 4$ be a solution to the system \eqref{eqn:systeml_i} with $a_i \in \O_v$. The assumption on the discriminant of $A$ allows us to apply Cramer's rule to find $(x,y,z,t) \in \O_v^4$ satisfying $\ell_i = a_i$ for $i = 1, \dots, 4$. Finally, the choice of system above was selected in a way such that $(x,y,z,t)$ lies on the quadric $Q$ but not on $\Delta$ and matches the local multiplicity condition for a $v$-adic Darmon point. Therefore this point lies in $(\sQ, \DDelta_m)_{\str}^{\dd}(\O_v)$.

    Let $P_1, P_2 \in (\sQ, \DDelta_m)_{\str}^{\dd}(\O_v)$ be two points as above with $k=2$ for $P_1$ and $k=3$ for $P_2$. Then $\inv_v \sA (P_1)$ vanishes since $\ell_1(P_1)/\ell(P_1) = (1/\pi^{m})^2$ is clearly a norm from the quadratic extension $K_v(\sqrt{d})$. On the other hand, $\inv_v \sA (P_2) = 1/2$ since $\ell_1(P_2)/\ell(P_2)$ has an odd $v$-adic valuation and so cannot be a norm from $K_v(\sqrt{d})$. 
  \end{proof}

  The conclusion of Theorem \ref{thm:main3} follows from Lemma~\ref{lem:inv-surjectivity} and the analysis above.
\end{proof}

Let us now illustrate that the statement of Theorem~\hyperref[thm:main3]{\ref*{thm:main3}\emph{\ref*{m3iiib}}} is false if $\Delta$ is smooth and $m$ is even or $m = \infty$. In the following example the lack of integral points on $\sU$ is explained by the lack of Darmon points for all even $m > 2$.

\begin{example} 
  \label{ex:quadrics-even}
  Let $Q \subset \PP_\QQ^3$ be given by 
    \[
      Q: \quad 9x^2 - 3y^2 = (t - 4z)(t + 4z) 
    \]
with hyperplane section $\Delta:t = 0$. Let $m > 2$ be an even integer and let $\Delta_m = (1 - 1/m)\Delta$. Fix $(\sQ, \DDelta_m)$ to be the obvious integral model of $(Q, \Delta_m)$ over $\ZZ$. The point $([x_p: y_p: z_p: t_p])_p$, where $[x_p: y_p: z_p: t_p] = [1:0:0:3]$ for $p \neq 3$ and $[x_3: y_3: z_3: t_3] = [0:0:1:4]$, is an integral adelic point, hence it is a strict Darmon adelic point for any $m \ge 2$. On the other hand, $\Delta(\QQ_3) = \emptyset$, thus $(\sQ, \DDelta_m)^{\dd}(\AA_\QQ) = (\sQ, \DDelta_m)_{\str}^{\dd}(\AA_\QQ)$
    
  Arguing as in the proof of Theorem~\ref{thm:main3}, we have that $\Br U / \Br \QQ \cong \ZZ / 2\ZZ$ is generated by the image of the quaternion algebra $\sA = (1 - 4z/t, 3)$. We claim that for any adelic Darmon point $([x_p:y_p:z_p:t_p])_p$, the value of $\inv_p \sA$ is $1/2$ if $p = 3$ and zero otherwise.

  Indeed, we work with points not lying on $\Delta$, hence $t \neq 0$. To compute the local invariant maps we can then use the fact that, for $(\cdot, \cdot)_p$ the Hilbert symbol,
  \[
    \inv_p(1 - 4z/t, 3) = \frac{1 - (1 - 4z/t, 3)_p}{4} = \frac{1 - (t^2 - 4zt, 3)_p}{4}.
  \] 

  If $p = \infty$, then the local invariant map vanishes since $3 > 0$. If $p \neq 2, 3$ is a finite prime and $3$ is a quadratic residue mod $p$ then $\inv_p$ vanishes. On the other hand, the Hilbert symbol is bilinear, thus $(t^2 - 4zt, 3)_p = (t - 4z, 3)_p (t, 3)_p$. It is now clear that whenever $3$ is not a quadratic residue mod $p$, then the left hand side of the defining equation of $\sQ$ does not vanish, so both $3$ and $t - 4z$ are units in $\Fp$. Therefore $(t^2 - 4zt, 3)_p = (t, 3)_p$. Our claim now follows form the expression of $(t, 3)_p$ in terms of Legendre symbols and the fact that $v(t)$ is even. 

  If $p = 3$, we claim that the only possible $3$-adic valuations of any non-zero $t$ are $0$ and $1$. Since $m \ge 4$, this implies that $3$ does not divide $t$. To see this, first reduce $\sQ$ mod 3. We see that if $3 \mid t$, then $3 \mid z$, hence $9 \mid 3y^2$, so $3 \mid y$. Dividing through by $9$ and reducing mod $3$ now shows that if $3^2$ divides $t$, then $3 \mid x$, which is a contradiction since $3$ cannot divide the coprime coordinates of $[x:y:z:t]$. Observe now that $(1 - 4z/t, 3) = (1 + 4z/t, 3)$ as quaternion algebras on $Q \setminus \Delta$. Then clearly $t - 4z \equiv 0 \bmod 3$ or $t + 4z \equiv 0 \bmod 3$ with both $t$ and $z$ units mod $3$, which, together with the two different representatives for the generator of the Brauer group, implies that $\inv_3(1 - 4z/t, 3) = 1/2$. 

  Finally, if $p = 2$ the reduction modulo powers of 2 shows that at most $2^3$ divides $t$. Since $m \ge 4$, it follows that $t$ must be odd and then $1 - 4z/t \equiv 1 \bmod 4$. This implies that the corresponding Hilbert symbol vanishes.

In conclusion, we see that for even $m > 2$, there cannot be any global Darmon points on $(\sQ, \DDelta_m)$, since $\sum_p \inv_p \sA = \inv_3 \sA = 1/2$, which gives a Brauer--Manin obstruction to their existence away from $\Delta$.
\end{example}

\begin{example}\label{ex:DHP}
The previous example leads to an example of failure of the Darmon Hasse principle for the orbifolds considered in Section \ref{section:D1}. Let $Q,\Delta$ be as in Example \ref{ex:quadrics-even}. Let $H\subset \P^3$ be the hyperplane defined by $t=0$. Note that $\calH\cong \P^2_{\O_S}$. Let $Q'\subset \P^3$ be the quadric obtained from $Q$ by changing the coefficient of $t^2$ to $-t^2$. The projection away from $[x:y:z:t]=[0:0:0:1]$ gives rise to double covers $f\colon Q\to H$ and $g\colon Q'\to H$. Then we have the following decomposition
\[
(\calH,\DDelta_{2m})(\Z)=f((\calQ,\DDelta_m)(\Z))\cup g((\calQ',\DDelta_m)(\Z)).
\]
We claim that $(\calH,\DDelta_{2m})$ fails the Darmon Hasse principle. By the beginning of the proof of Theorem~\hyperref[thm:main1]{\ref*{thm:main1}\emph{\ref*{m1i}}}, it follows that $(\calH,\DDelta_{2m})$ has local integral (hence Darmon) points everywhere.
Having shown in the previous example that $(\calQ,\DDelta_m)(\Z)=\emptyset$, it suffices to show that $(\calQ',\DDelta_m)(\Z)=\emptyset$ as well. The latter follows from an easy reduction mod $3$ argument that shows $(\calQ',\DDelta_m)(\Z_3)=\emptyset$.
\end{example}

When $\Delta$ is not smooth in Theorem~\ref{thm:main3}, the Brauer group coincides with $\Br K$.

\begin{proposition} \label{prop:main3non-smooth-b}
Let $n \geq 1$, and assume that $\Delta = Q \cap H$ is singular. Let $T$ be a finite non-empty set of places of $K$. 
\begin{enumerate}[label=\emph{(\roman*)}]
\item We have $\Br(Q \setminus \Delta) \cong \Br K$.
\item The orbifold $(\sQ, \DDelta_m)$ satisfies the Hasse principle and strong approximation off $T$ for strict Campana and Darmon points.
\end{enumerate}
\end{proposition}

\begin{proof}
We begin with (i). Assume first that $n=1$. Then $\Delta = Q \cap H$ is a double point, hence the intersection of the tangent line to the smooth conic $Q$ at the point $\Delta \in Q\left(K\right)$, and it follows that $Q \setminus \Delta \cong \AA^1_K$. Then the result follows from the fact that $\Br \AA^1_k \cong \Br k$ for any perfect field $k$ \cite[Thm.~5.6.1(viii)]{CTS2}. Now assume that $n \geq 2$. If $\Delta = Q \cap H$ is singular, then $H$ is the tangent hyperplane to $Q$ at some point $P_0 \in Q$, and $Q \cap H$ is a quadric cone in $H$ with vertex $P_0$ (see \cite[p.~283]{HAR}). It is easily seen that $P_0 \in Q(K)$ as the unique singular point of the $K$-rational variety $Q \cap H$. Let $\H' \subset \PP^{n+1}$ be a plane not containing $P_0$, and consider projection from $P_0$ as a birational map $\phi: Q \dashrightarrow H$. The map $\phi$ restricts to an isomorphism between $U = Q \setminus (Q \cap H) \subset Q$ and the open subset $H' \setminus Z$, where $Z = Q \cap H \cap H'$. Note that the triple intersection $Z$ is of codimension $2$ in $H'$, hence $\Br (Q \setminus \Delta) \cong \Br (H' \setminus Z) \cong \Br(H')$, where the final second isomorphism follows from Grothendieck's purity theorem \cite[Thm.~3.7.6]{CTS2}. Identifying $H'$ with $\PP^n$, we have $\Br (H') \cong \Br (\PP^n) \cong \Br K$, and we deduce the claim.

Now we prove (ii). The case $n=1$ follows from the fact that $\AA^1_\ZZ$ satisfies the integral Hasse principle. Now assume that $n \geq 2$. As shown in the proof of (i), we have $Q \setminus \Delta \cong \PP^n \setminus W$ for $W$ a closed subset of codimension $2$. By \cite[Thm.,~p.~1]{WEI}, $\PP^n \setminus W$ satisfies strong approximation off any place $v_0 \in \Omega_K$. 
It follows via expanding and retracting the set $S$ as in the proof of \cite[Lem.~2.8]{NS} that $(\sQ \setminus \DDelta)(\O_{S \cup \{v_0\}})$ is dense in $(\sQ \setminus \DDelta)(\A_{\O_{S \cup \{v_0\}}}^{\{v_0\}})$. The result then follows from Proposition \ref{prop:BMOequiv}.
\end{proof}

We finish this section with an example of $(\sQ, \DDelta_m)$ with $m$ even where Darmon weak approximation holds. This illustrates the difference between $(\sQ, \DDelta_m)$ and the orbifolds $(\mathbb{P}^{n+1}_{\O_S},\calQ_m)$ studied in Theorem~\ref{thm:main1} which fail weak approximation for even weights.

\begin{example} 
  \label{ex:DWA}
  Let $Q \subset \mathbb{P}^3_{\QQ}$ be the quadric surface given by 
  \[
    Q: 49x^2 - 7y^2 + 16z^2 = t^2,
  \]
  and set $\Delta := Q \cap H$ for $H:t=0$. Let $m \in \ZZ_{\ge 4}$ be even and define $\Delta_m := (1 - 1/m)\Delta$. Set $U := Q \setminus \Delta$. Let $\sQ$, $\DDelta$ and $\sU$ be the obvious $\ZZ$-models of $Q$,$\Delta$ and $U$ respectively. 

  It is clear that $\sU(\Zp) \neq \emptyset$ for $p \neq 2, 7$. Further, $[1:0:0:1]$ mod 8 and $[0:0:2:1]$ mod $7$ lift by Hensel's lemma to $2$-adic and $7$-adic points respectively, hence $\sU(\AA_\ZZ) \neq \emptyset$, which in turn implies that $(\sQ, \DDelta_m)^{\dd}(\AA_\QQ) \neq \emptyset$. Note that $\Delta(\QQ_7) = \emptyset$, hence $(\sQ, \DDelta_m)^{\dd}(\AA_\QQ) = (\sQ, \DDelta_m)_{\str}^{\dd}(\AA_\QQ)$, so weak approximation concerns only strict Darmon points.

  Lemma~\ref{lem:quadricsBGp} implies that $\Br U / \Br \QQ \cong \ZZ/2\ZZ$ is generated by the quaternion algebra $(1 -4z/t, 7)$. The same argument as in Example~\ref{ex:quadrics-even} implies that $\inv_p \sA = 0$ for all primes $p$. Indeed, if $p \neq 2,7$ then either $7$ is a square or $\valp(1 - 4z/t)$ is even. For $p = 2,7$ this follows from the fact that there are no $p$-adic Darmon points outside $\sU(\Zp)$ and a simple computation at integral local points. Then $(\sQ, \DDelta_m)^{\dd}(\AA_\QQ)^{\Br} = (\sQ, \DDelta_m)_{\str}^{\dd}(\AA_\QQ)$. Moreover, the vanishing of $\inv_p \sA = 0$ for all primes $p$ also implies that $\sU(\AA_\ZZ)^{\Br} = \sU(\AA_\ZZ)$. 

  On the other hand, $\Delta(\RR) \neq \emptyset$, thus strong approximation off $\infty$ for integral points on $\sU$ holds by \cite[Thm.~6.3]{CTX}. Proposition~\ref{prop:BMOequiv} now implies that Darmon strong approximation off $\infty$ holds and so Darmon weak approximation must also hold.
\end{example}

\section{Density of Hasse failures} \label{section:fail}

In this section we prove Theorem~\ref{thm:main4}. The upper bound in the statement of Theorem~\ref{thm:main4} follows from the upper bound in \cite[Thm.~1.1]{SAN} since $N_n(B)$ is bounded above by the number of affine diagonal ternary quadrics which fail the integral Hasse principle, with no extra condition imposed upon them. In order to prove the lower bound of Theorem~\ref{thm:main4}, we consider the following family of projective quadrics over the rational numbers:
\[
  Q: \quad 5ab^2x^2 - 25ad^2y^2 + 16c^2z^2 = t^2
\]
for non-zero integers $a, b, c, d$ coprime to $5$, with the additional assumptions that 
\begin{enumerate}[label=(\roman*)]
  \item $a \equiv 1 \bmod 40$ is squarefree and coprime to $c$, and
  \item $bd$ is coprime to $2c$.
\end{enumerate}
To fix the orbifold structure on $Q \subset \PP^3_{\QQ}$, let $m \ge 2$ be an integer and let $D_m = (1 - 1/m)Z(t)$. Finally, let $(\sQ, \DDelta_m)$ be the closure of $(Q,\Delta_m)$ in $\P^3_{\Z}$.

We claim that $\sQ \setminus \DDelta$ lacks integral points because of a Brauer--Manin obstruction but that there is no Brauer--Manin obstruction to the existence of strict Darmon points for any choice of weight $m \ge 1$. Since $5ab^2x^2 - 25ad^2y^2 + 16c^2z^2$ is indefinite, it will then follow from Proposition~\ref{prop:BMOequiv} and \cite[Thm.~6.3]{CTX} that $(\sQ, \DDelta_m)_{\str}^{\dd}(\ZZ) \neq \emptyset$ for all $m \ge 1$. 

To see this, we first need to verify that $(\sQ \setminus \DDelta)(\AA_\ZZ) \neq \emptyset$. It is clear that $[0:0:1:4c] \in (\sQ \setminus \DDelta)(\Zp)$ for $p = \infty$ and for any $p \nmid 2c$.  If $p \mid c$ and $p \neq 2$, then the reduction of the conic $5ab^2x^2 - 25ad^2y^2 - 1 = 0$ mod $p$ has a smooth $\Fp$-point which lifts to a $\Zp$-point with $t = 1$ by Hensel's lemma. This can be verified, for example, by comparing the values of $5ab^2x^2$ and $25ad^2y^2 + 1$ mod $p$. Finally, when $p = 2$, it suffices to look at the reduction of $5ab^2x^2 - 25ad^2y^2 - 1 = 0$ mod 8. Since $b$ and $d$ are odd and $a \equiv 1 \bmod 8$, it is given by $5x^2 - y^2 - 1 \equiv 0 \bmod 8$. It is clear that $(x, y) \equiv (1,2) \bmod 8$ is a solution of this equation mod $8$, and such a solution lifts to a $\ZZ_2$-point on $\sQ \setminus \DDelta$.

The next step is to verify that there is a Brauer--Manin obstruction to the existence of integral points on the affine surface $\sQ \setminus \DDelta$. The Brauer group $\Br (Q \setminus \Delta) / \Br \QQ$ is isomorphic to $\ZZ/2\ZZ$ since $16 \times 125 \times  a^2b^2d^2c^2 \sim 5 \bmod \QQ^{*2}$. This follows from \cite[Prop.~7.3.2]{CTS2}. Moreover, this group is generated by the quaternion Azumaya algebra $((t - 4cz)/t, 5)$, which we will also write as $(1 - 4cz, 5)$ for computational purposes when working over $\sQ \setminus \DDelta$. All local invariant maps at primes different from $5$ vanish for this Brauer element, while the invariant map at $5$ equals 1/2 when evaluated at any $\ZZ_5$-point. This can be seen, for example, by taking into account that 
\[
  \inv_p(1 - 4cz, 5) = \frac{1 - (1 - 4cz, 5)_p}{4},
\]
where the last entry is the Hilbert symbol at $p$. If $p = 2$, the claim follows from considering the congruence class of $5$ modulo $8$. If $p \neq 5$, then the defining equation of $\sQ \setminus \DDelta$ implies that either $5$ is a quadratic residue mod $p$ or the $p$-adic valuation of $(1 - 4cz)$ is even. This last claim follows from the fact that the sum of $1 - 4cz$ and $1 + 4cz$ is always $2$, hence if an odd prime divides one of them, it cannot divide the other. Lastly, if $p = 5$, we take advantage of the fact that $(1 - 4cz, 5) = (1 + 4cz, 5)$ as quaternion algebras over $Q \setminus \Delta$, hence at any given local point, their invariant maps must be equal. It follows from considering the reduction of $\sQ \setminus \DDelta$ modulo $5$ that $4cz$ is congruent to either $1$ or $-1$ modulo $5$. In both cases, we see that $(1 - 4cz, 5)_5 = \(\frac{2}{5}\)$ in terms of the Legendre symbol. Since $2$ is a quadratic non-residue modulo $5$, this confirms our claim.

Lastly, we must show that there is no Brauer--Manin obstruction to the existence of strict Darmon points of any weight $m \ge 1$. For $m = 1$, this is implied by the Hasse principle for rational points on $U$. If $m \ge 2$, by our analysis above, it suffices to show that there is a $5$-adic strict Darmon point for which the local invariant map vanishes. To see that such a point exists, we look at local Darmon points close to $\DDelta$. Consider $[x:y:z:t] = [5:c^2:-5cd:d5^m]$. Under the assumption $m \ge 2$, the above point makes the defining equation for $\sQ \bmod 125$ vanish. Moreover, its partial derivative with respect to $z$ has $5$-adic valuation exactly $1$, so Hensel's lemma implies that this point lifts to a $5$-adic point on $Q$ with $z/5 \equiv -cd \bmod 5$ and $\val_5(t) = m$, thus it lies in $(\sQ, \DDelta_m)_{\str}^{\dd}(\ZZ_5)$. Computing the local invariant map for this point in terms of the Hilbert symbol gives 
\[
  \inv_5(t^2 - 4czt, 5) = \frac{1 - (4c^2d^2, 5)_5}{4} = 0.
\]

Let $B \ge 1$ be a real number. The above analysis shows that if we multiply the defining equation of $\sQ$ through by $n$, the pair $(\sQ, \DDelta_m)$ will be counted by $N_n(B)$ provided that all the coefficients of $\sQ$ are of size at most $B/n$. Since $n$ is fixed, we can redefine $B$ and so, in order to prove Theorem~\ref{thm:main4}, it suffices to count the number of $\sQ$ as above for which $0 < |5ab^2|, |25ad^2|, |16c^2| \le B$. In fact, since we seek a lower bound for this quantity, the sign of the coefficients of $\sQ$ is immaterial, thus, for $\mu(\cdot)$ the M{\"o}bius function,
\begin{equation}
  \label{eqn:quadsum}
  N_n(B) 
  \gg \sum_{\substack{a \le B/25 \\ a \equiv 1 \bmod 40}} \mu^2(a)
  \sum_{\substack{b \le \sqrt{B/5a} \\ (b, 10) = 1}}
  \sum_{\substack{d \le \sqrt{B/25a} \\ (d, 10) = 1}}
  \sum_{\substack{c \le \sqrt{B}/4 \\ (c, 5abd) = 1}} 1.
\end{equation}

If $k$ is a fixed positive integer, then one can count the number of positive integers up to $X$ which are coprime to $k$ using $\mu(\cdot)$ to define an indicator function for the coprimality condition. In particular, if $\tau(k)$ denotes the number of divisors of $k$, then
\begin{equation}
  \label{eqn:n<X, (n,k) = 1}
  \sum_{\substack{n \le X \\ (n,k) = 1}} 1
  = \sum_{r \mid k} \mu(r) \sum_{\substack{n \le X \\ n \equiv 0 \bmod r}} 1
  = X \sum_{r \mid k} \frac{\mu(r)}{r} + O(\tau(k))
  = \frac{\varphi(k)}{k}X + O(\tau(k)).
\end{equation}

To estimate the quadruple sum appearing in \eqref{eqn:quadsum} we begin with the inside sum over $c$ for which we apply the above asymptotic. 
Since $a, b, c$ are all coprime to $5$ and $\varphi$ is multiplicative. the coefficient of the leading term is $4\varphi(abd)/5abd$. Applying $\varphi(abd) \ge \varphi(a) \varphi(b) \varphi(d)$ and using the fact that both $\tau(abd)$ and $\log B$ are $O(B^{\varepsilon})$ thus gives
\begin{equation} \label{eqn:step1}
  N(B)
  \gg B^{\frac{1}{2}} \sum_{\substack{a \le B/25 \\ a \equiv 1 \bmod 40}} \mu^2(a) \frac{\varphi(a)}{a}
  \sum_{\substack{b \le \sqrt{B/5a} \\ (b, 10) = 1}} \frac{\varphi(b)}{b}
  \sum_{\substack{d \le \sqrt{B/25a} \\ (d, 10) = 1}} \frac{\varphi(d)}{d} 
  + O\(B^{1 + \varepsilon}\).
\end{equation}
We then split the summation over $a$ in two smaller sums. If $B/50 < a < B/25$, we bound the sums over $b$ and $d$ trivially by $O(1)$, which produces a total error term of size $O(B^{3/2})$ after applying the trivial bound for the corresponding sum over $a$. If $a \le B/50$ we use that $\varphi(n)/n$ is $1$ on average and so the sums over $b$ and $d$ are both of magnitude $\sqrt{B/a}$. This can be seen, for example, by expanding the Dirichlet convolution of $\varphi(d)/d$. Indeed,
\[
  \begin{split}
    \sum_{\substack{d \le \sqrt{B/25a} \\ (d, 10) = 1}} \frac{\varphi(d)}{d}
    = \sum_{\substack{r \le \sqrt{B/25a} \\ (r, 10) = 1}} \frac{\mu(r)}{r} \sum_{\substack{d \le \sqrt{B/25ar^2} \\ (d, 10) = 1}} 1
    = c_1 \frac{B^{\frac{1}{2}}}{\sqrt{a}} + O\(\log\(\frac{B}{25a}\)\),
  \end{split}
\]
where $c_1$ is a positive constant that can be calculated explicitly by completing the sum over $r$. The above error term comes from trivially bounding the sum over $r$ inside it. Treating the sum over $b$ analogously shows that there is a positive constant $c_2$ such that
\[
  \sum_{\substack{b \le \sqrt{B/5a} \\ (b, 10) = 1}} \frac{\varphi(b)}{b}
  = c_2 \frac{B^{\frac{1}{2}}}{\sqrt{a}} + O\(\log\(\frac{B}{5a}\)\).
\]
Substituting the above asymptotic for the sums over $b$ and over $d$ in \eqref{eqn:step1} generates error terms of magnitude 
\[
  B^{\frac{1}{2}}\sum_{a \le B/50}\( \log\(\frac{B}{5a}\) \)^2 \quad \quad \text{and} \quad \quad
  B \sum_{a \le B/50} \frac{1}{\sqrt{a}} \log\(\frac{B}{5a}\),
\]
where we have used that $\log(B/25a) \ll \log(B/5a)$. Abel's summation formula now implies that each of these error terms is $O(B^{3/2})$. 

We claim that the order of magnitude of the remaining sum over $a$ is $\log B$. To see this, encode $a \equiv 1 \bmod 40$ via orthogonality of Dirichlet characters mod $40$:
\begin{equation}
  \label{eqn:Dirichlet}
  \begin{split}
    \sum_{\substack{a \le B/50 \\ a \equiv 1 \bmod 40}} \mu^2(a) \frac{\varphi(a)}{a^2} 
    = \frac{1}{\varphi(40)}\sum_{\chi \bmod 40} \sum_{\substack{a \le B/50}} \mu^2(a)\frac{\varphi(a)}{a^2} \chi(a).
  \end{split}
\end{equation}
Let $\chi$ mod 40 be a non-principal character mod 40. For $X \ge 1$ a real parameter, Abel's summation formula shows that
\begin{equation}
  \label{eqn:char_partial}
  \sum_{a \le X} \frac{\chi(a)}{a} 
  = \frac{1}{X} \sum_{a \le X} \chi(a) + \int_{1}^{X} \( \sum_{a \le t} \chi(a) \) \frac{dt}{t^2} 
  \ll 1,
\end{equation}
where the upper bound follows from the fact that $\sum_{a \le X} \chi(a)$ is $O(1)$ and the implied constant depends only on the conductor which is a fixed constant in this case. On the other hand, if $\chi$ mod 40 is principal, then the same method and \eqref{eqn:n<X, (n,k) = 1} yield
\begin{equation}
  \label{eqn:char_partial_principal}
  \sum_{a \le X} \frac{\chi(a)}{a} 
  = \frac{1}{X} \sum_{\substack{a \le X \\ (a, 40) = 1}} 1 + \int_{1}^{X} \( \sum_{\substack{a \le t \\ (a, 40) = 1}} 1 \) \frac{dt}{t^2} 
  = \frac{2}{5} \log X + O(1).
\end{equation}
Let $S(\chi, B)$ denote the inside sum over $a$ in \eqref{eqn:Dirichlet}. From the Dirichlet convolution of $\varphi(a)/a$, the multiplicativity of $\chi$ and the fact that the M\"obius function is supported on squarefree integers, it follows that
\[
  \begin{split}
    S(\chi, B) &= \sum_{\substack{a \le B/50}} \mu^2(a)\frac{\varphi(a)}{a^2} \chi(a) \\
    &= \sum_{\substack{a \le B/50}} \frac{\mu^2(a)}{a} \chi(a) 
      \sum_{r \mid a} \frac{\mu(r)}{r}
    = \sum_{r \le B/50} \frac{\mu(r)}{r^2} \chi(r) 
      \sum_{\substack{a \le B/50r \\ (a, r) = 1}} \frac{\mu^2(a)}{a} \chi(a).
  \end{split}
\] 
Expanding $\mu^2(a) = \sum_{d^2 \mid a} \mu(d)$ then yields
\[
  S(\chi, B)
  = \sum_{r \le B/50} \frac{\mu(r)}{r^2} \chi(r) 
    \sum_{\substack{d \le \sqrt{B/50r} \\ (d, r) = 1}} \frac{\mu(d)}{d^2}\chi(d^2) 
    \sum_{\substack{a \le B/50rd^2 \\ (a, r) = 1}} \frac{\chi(a)}{a}.
\]
Lastly, using the indicator function $\sum_{t \mid (a, r)} \mu(t)$ for the condition $(a,r) = 1$ gives 
\[
  S(\chi, B)
  = \sum_{d \le \sqrt{B/50}} \frac{\mu(d)}{d^2}\chi(d^2)
    \sum_{\substack{r \le B/50d^2 \\ (r, d) = 1}} \frac{\mu(r)}{r^2} \chi(r) 
    \sum_{t \mid r} \frac{\mu(t)}{t} \chi(t)
    \sum_{a \le B/50rtd^2} \frac{\chi(a)}{a}.
\]

If $\chi$ mod 40 is non-principal, we may apply \eqref{eqn:char_partial}. Alternatively, for the principal character mod $40$, we apply \eqref{eqn:char_partial_principal}. After doing so, in either case, the sum over $t$ is at most $\tau(r) \ll r^{\varepsilon}$, which in combination with the properties of the logarithm implies that the remaining sums over $r$ and $d$ are convergent, hence all of the sums $S(\chi, B)$ for non-principal characters $\chi$ contribute to an error of size $O(1)$ in \eqref{eqn:Dirichlet}, while $S(\chi, B)$ for the principal character mod $40$ is asymptotically $c_3 \log B$ for some positive constant $c_3$.
Thus, in view of \eqref{eqn:step1} we have $N_n(B) \gg B^{3/2}\log B$, which completes the proof of Theorem~\ref{thm:main4}.



\begin{thebibliography}{1}
  \bibitem{BP} M. Bhargava, B. Poonen, {\em The local-global principle for integral points on stacky curves.} \texttt{arXiv:2006.00167v1}.
	\bibitem{CAM04} F. Campana, {\em Orbifolds, special varieties and classification theory.} Ann. Inst. Fourier (Grenoble) 54 (2004), no. 3, 499–630.
	\bibitem{CAM05} F. Campana, {\em Fibres multiples sur les surfaces: aspects geom\'etriques, hyperboliques et arithm\'etiques.} Manuscripta Math. 117 (2005), no. 4, 429–461.
	\bibitem{CAM11} F. Campana, {Orbifoldes g\'eom\'etriques sp\'eciales et classification bim\'eromorphe des vari\'et\'es k\"ahl\'eriennes compactes.} J. Inst. Math. Jussieu 10 (2011), no. 4, 809--934.
	\bibitem{CAM11a} F. Campana, {\em Special orbifolds and birational classification: a survey.} Classification of algebraic varieties, 123--170, EMS Ser. Congr. Rep., Eur. Math. Soc., Zürich, 2011.
	\bibitem{CAM15} F. Campana, {\em Special manifolds, arithmetic and hyperbolic aspects: a short survey.} Rational points, rational curves, and entire holomorphic curves on projective varieties, 23--52, Contemp. Math., 654, Centre Rech. Math. Proc., Amer. Math. Soc., Providence, RI, 2015.
  \bibitem{CX} Y. Cao, F. Xu, {\em Strong approximation with Brauer-Manin obstruction for toric varieties.} Ann. Inst. Fourier (Grenoble) 68 (2018), no. 5, 1879–1908.
	\bibitem{CTS1} J.-L. Colliot-Th\'el\`ene, A. N. Skorobogatov, {\em Descent on fibrations over ${\bf P}^1_k$ revisited.} Math. Proc. Cambridge Philos. Soc. 128 (2000), no. 3, 383--393.
  \bibitem{CTS2} J.-L. Colliot-Th\'el\`ene, A. N. Skorobogatov, {\em The Brauer-Grothendieck group.} Results in Mathematics and Related Areas. 3rd Series. A Series of Modern Surveys in Mathematics, 71. Springer, Cham, 2021.
	\bibitem{CTX} J.-L. Colliot-Th\'el\`ene, F. Xu, {\em Brauer-Manin obstruction for integral points of homogeneous spaces and representation by integral quadratic forms.} Compos. Math. 145 (2009), no. 2, 309--363, with an appendix by D. Wei and F. Xu.
	\bibitem{CTX13} J.-L. Colliot-Th\'el\`ene, F. Xu, {\em Strong approximation for the total space of certain quadratic fibrations.} Acta Arith. 157 (2013), no. 2, 169--199.
	\bibitem{DAR} H. Darmon, {\em Faltings plus epsilon, Wiles plus epsilon, and the generalized Fermat equation.} C. R. Math. Rep. Acad. Sci. Canada 19 (1997), no. 1, 3--14.
	\bibitem{DK} A. Dhillon, I. Kobyzev, {\em $G$-theory of root stacks and equivariant $K$-theory.} Ann. K-Theory 4 (2019), no. 2, 151--183.
  \bibitem{FI} J. Friedlander, H. Iwaniec, {\em Ternary quadratic forms with rational zeros.} J. Th\'eor. Nombres Bordeaux 22 (2010), no. 1, 97--113.
  \bibitem{GR} A. Granville, O. Ramar\'e, {\em Explicit bounds on exponential sums and the scarcity of squarefree binomial coefficients.} Mathematika 43 (1996), no. 1, 73--107.
	\bibitem{HARA} D. Harari, {\em M\'ethode des fibrations et obstruction de Manin.} Duke Math. J. 75 (1994), no. 1, 221--260.
	\bibitem{HAR} J. Harris, {\em Algebraic geometry. A first course.} Graduate Texts in Mathematics, 133. Springer-Verlag, New York, 1992.
	\bibitem{KT} A. Kresch, Y. Tschinkel, {\em Two examples of Brauer–Manin obstruction to integral points.} Bull. Lond. Math. Soc. 40 (2008), no. 6, 995--1001.
	\bibitem{LIU} Q. Liu, {\em Algebraic geometry and arithmetic curves.} Oxford Graduate Texts in Mathematics, 6. Oxford Science Publications. Oxford University Press, Oxford, 2002.
	\bibitem{LM} D. Loughran, V. Mitankin, {\em Integral Hasse principle and strong approximation for Markoff surfaces.} Int. Math. Res. Not. IMRN 2021, no. 18, 14086--14122.
  \bibitem{LW23} C. Lv and H. Wu, {\em The Brauer-Manin obstruction on algebraic stacks.} Preprint, \texttt{arXiv:2306.14426},
	\bibitem{MIT} V. Mitankin, {\em Failures of the integral Hasse principle for affine quadric surfaces.} J. Lond. Math. Soc. (2) 95 (2017), no. 3, 1035--1052.
	\bibitem{NS} M. Nakahara, S. Streeter, {\em Weak approximation and the Hilbert property for Campana points.} Michigan Math. J., to appear, \texttt{arXiv:2010.12555}.
	\bibitem{NX1} B. Nasserden, S. Xiao, {\em The density of rational points on $\mathbb{P}^1$ with three stacky points.} Preprint, \texttt{arXiv:2011.06586v1}.
	\bibitem{NX2} B. Nasserden, S. Xiao, {\em Heights and quantitative arithmetic on stacky curves.} Preprint, \texttt{arXiv:2108.04411v1}.
	\bibitem{PSTVA} M. Pieropan, A. Smeets, S. Tanimoto, A. V\'arilly-Alvarado, {\em Campana points of bounded height on vector group compactifications.} Proc. Lond. Math. Soc. (3) 123 (2021), no. 1, 57--101.
	\bibitem{POO} B. Poonen, {\em Rational points on varieties.} Graduate Studies in Mathematics, 186. American Mathematical Society, Providence, RI, 2017.
  \bibitem{SAN} T. Santens, {\em Integral points on affine quadric surfaces.} J. Th\'eor. Nombres Bordeaux 34 (2022), no. 1, 141--161.
  \bibitem{SAN23} T. Santens, {\em The Brauer-Manin obstruction for stacky curves}, \texttt{arXiv:2210.17184}
	\bibitem{SER} J.-P. Serre, {\em Lettre \`a Tsfasman.} Journ\'ees Arithm\'etiques, 1989 (Luminy, 1989). Ast\'erisque No. 198-200, (1991), 11, 351-353 (1992).
	\bibitem{SHU1} A. Shute, {\em Sums of four squareful numbers.} Preprint, \texttt{arXiv:2104.06966v1}.
	\bibitem{SHU2} A. Shute, {\em On the leading constant in the Manin-type conjecture for Campana points.} Acta Arith. 204 (2022), no. 4, 317--346.
	\bibitem{STR} S. Streeter, {\em Campana points and powerful values of norm forms.} Math. Z. 301 (2022), no. 1, 627--664. 
  \bibitem{UEM} T. Uematsu, {\em On the Brauer group of affine diagonal quadrics.} J. Number Theory 163 (2016), 146--158.
	\bibitem{VAV} A. V\'arilly-Alvarado, B. Viray, {\em Arithmetic of del Pezzo surfaces of degree 4 and vertical Brauer groups.} Advances in Mathematics 255 (2014), 153--181.
	\bibitem{SP} Various authors, {\em Stacks Project.} \url{https://stacks.math.columbia.edu}
	\bibitem{WEI} D. Wei, {\em Strong approximation for a toric variety.} Acta Math. Sin. (Engl. Ser.) 37 (2021), no. 1, 95–103.
\end{thebibliography}
\end{document}